\documentclass[11pt,oneside,english,american]{amsart}
\usepackage[T1]{fontenc}
\usepackage[utf8]{inputenc}
\usepackage{amstext}
\usepackage{amsthm}
\usepackage{amssymb}
\usepackage{stmaryrd}
\usepackage{geometry}
\usepackage[all]{xy}

\makeatletter
\numberwithin{equation}{section}
\numberwithin{figure}{section}
\theoremstyle{plain}
\newtheorem{thm}{\protect\theoremname}[section]
\theoremstyle{plain}
\newtheorem{conjecture}[thm]{\protect\conjecturename}
\theoremstyle{definition}
\newtheorem{defn}[thm]{\protect\definitionname}
\theoremstyle{remark}
\newtheorem{rem}[thm]{\protect\remarkname}
\theoremstyle{definition}
\newtheorem{example}[thm]{\protect\examplename}
\theoremstyle{plain}
\newtheorem{lem}[thm]{\protect\lemmaname}
\theoremstyle{plain}
\newtheorem{prop}[thm]{\protect\propositionname}
\theoremstyle{plain}
\newtheorem{cor}[thm]{\protect\corollaryname}
\theoremstyle{definition}
\newtheorem{condition}[thm]{\protect\conditionname}

\usepackage[hidelinks]{hyperref}

\makeatother

\usepackage{babel}
\addto\captionsamerican{\renewcommand{\conditionname}{Condition}}
\addto\captionsamerican{\renewcommand{\conjecturename}{Conjecture}}
\addto\captionsamerican{\renewcommand{\corollaryname}{Corollary}}
\addto\captionsamerican{\renewcommand{\definitionname}{Definition}}
\addto\captionsamerican{\renewcommand{\examplename}{Example}}
\addto\captionsamerican{\renewcommand{\lemmaname}{Lemma}}
\addto\captionsamerican{\renewcommand{\propositionname}{Proposition}}
\addto\captionsamerican{\renewcommand{\remarkname}{Remark}}
\addto\captionsamerican{\renewcommand{\theoremname}{Theorem}}
\addto\captionsenglish{\renewcommand{\conditionname}{Condition}}
\addto\captionsenglish{\renewcommand{\conjecturename}{Conjecture}}
\addto\captionsenglish{\renewcommand{\corollaryname}{Corollary}}
\addto\captionsenglish{\renewcommand{\definitionname}{Definition}}
\addto\captionsenglish{\renewcommand{\examplename}{Example}}
\addto\captionsenglish{\renewcommand{\lemmaname}{Lemma}}
\addto\captionsenglish{\renewcommand{\propositionname}{Proposition}}
\addto\captionsenglish{\renewcommand{\remarkname}{Remark}}
\addto\captionsenglish{\renewcommand{\theoremname}{Theorem}}
\providecommand{\conditionname}{Condition}
\providecommand{\conjecturename}{Conjecture}
\providecommand{\corollaryname}{Corollary}
\providecommand{\definitionname}{Definition}
\providecommand{\examplename}{Example}
\providecommand{\lemmaname}{Lemma}
\providecommand{\propositionname}{Proposition}
\providecommand{\remarkname}{Remark}
\providecommand{\theoremname}{Theorem}

\begin{document}
\title{The Batyrev--Manin conjecture for DM stacks II}
\author{Ratko Darda}
\address{Faculty of Engineering and Natural Sciences, Sabancı University, Tuzla,
Istanbul, Turkey}
\email{ratko.darda@gmail.com, ratko.darda@sabanciuniv.edu}
\author{Takehiko Yasuda}
\address{Department of Mathematics, Graduate School of Science, the University
of Osaka, Toyonaka, Osaka 560-0043, JAPAN}
\address{Kavli Institute for the Physics and Mathematics of the Universe, the
University of Tokyo, Kashiwa, Chiba, 277-8583, Japan}
\email{yasuda.takehiko.sci@osaka-u.ac.jp}
\begin{abstract}
In this paper, we propose a new framework for studying the distribution
of rational points on DM stacks of positive characteristic. Our primary
focus is on wild stacks, which existing frameworks do not address.
There was not even a satisfactory notion of heights for such stacks.
First, we introduce a new kind of height function that extends the
authors' idea from their preceding paper on characteristic-zero stacks.
This new height function is more general and flexible than the previous
one. Examples of the new height function include discriminants of
torsors, minimal discriminants, and conductors of elliptic curves
in characteristic three. Next, we formulate a generalization of the
Batyrev--Manin conjecture for rational points of DM stacks in positive
characteristic relative to this new type of height function. We provide
several pieces of evidence for this generalization. 
\end{abstract}

\maketitle
\selectlanguage{english}%
\global\long\def\bigmid{\mathrel{}\middle|\mathrel{}}%

\global\long\def\AA{\mathbb{A}}%

\global\long\def\CC{\mathbb{C}}%

\global\long\def\EE{\mathbb{E}}%

\global\long\def\FF{\mathbb{F}}%

\global\long\def\GG{\mathbb{G}}%

\global\long\def\LL{\mathbb{L}}%

\global\long\def\MM{\mathbb{M}}%

\global\long\def\NN{\mathbb{N}}%

\global\long\def\PP{\mathbb{P}}%

\global\long\def\QQ{\mathbb{Q}}%

\global\long\def\RR{\mathbb{R}}%

\global\long\def\SS{\mathbb{S}}%

\global\long\def\ZZ{\mathbb{Z}}%

\global\long\def\bA{\mathbf{A}}%

\global\long\def\ba{\mathbf{a}}%

\global\long\def\bb{\mathbf{b}}%

\global\long\def\bc{\mathbf{c}}%

\global\long\def\bd{\mathbf{d}}%

\global\long\def\bf{\mathbf{f}}%

\global\long\def\bg{\mathbf{g}}%

\global\long\def\bh{\mathbf{h}}%

\global\long\def\bj{\mathbf{j}}%

\global\long\def\bk{\mathbf{k}}%

\global\long\def\bm{\mathbf{m}}%

\global\long\def\bp{\mathbf{p}}%

\global\long\def\bq{\mathbf{q}}%

\global\long\def\br{\mathbf{r}}%

\global\long\def\bs{\mathbf{s}}%

\global\long\def\bt{\mathbf{t}}%

\global\long\def\bv{\mathbf{v}}%

\global\long\def\bw{\mathbf{w}}%

\global\long\def\bx{\boldsymbol{x}}%

\global\long\def\by{\boldsymbol{y}}%

\global\long\def\bz{\mathbf{z}}%

\global\long\def\cA{\mathcal{A}}%

\global\long\def\cB{\mathcal{B}}%

\global\long\def\cC{\mathcal{C}}%

\global\long\def\cD{\mathcal{D}}%

\global\long\def\cE{\mathcal{E}}%

\global\long\def\cF{\mathcal{F}}%

\global\long\def\cG{\mathcal{G}}%

\global\long\def\cH{\mathcal{H}}%

\global\long\def\cI{\mathcal{I}}%

\global\long\def\cJ{\mathcal{J}}%

\global\long\def\cK{\mathcal{K}}%

\global\long\def\cL{\mathcal{L}}%

\global\long\def\cM{\mathcal{M}}%

\global\long\def\cN{\mathcal{N}}%

\global\long\def\cO{\mathcal{O}}%

\global\long\def\cP{\mathcal{P}}%

\global\long\def\cQ{\mathcal{Q}}%

\global\long\def\cR{\mathcal{R}}%

\global\long\def\cS{\mathcal{S}}%

\global\long\def\cT{\mathcal{T}}%

\global\long\def\cU{\mathcal{U}}%

\global\long\def\cV{\mathcal{V}}%

\global\long\def\cW{\mathcal{W}}%

\global\long\def\cX{\mathcal{X}}%

\global\long\def\cY{\mathcal{Y}}%

\global\long\def\cZ{\mathcal{Z}}%

\global\long\def\fa{\mathfrak{a}}%

\global\long\def\fb{\mathfrak{b}}%

\global\long\def\fc{\mathfrak{c}}%

\global\long\def\fd{\mathfrak{d}}%

\global\long\def\fe{\mathfrak{e}}%

\global\long\def\ff{\mathfrak{f}}%

\global\long\def\fj{\mathfrak{j}}%

\global\long\def\fk{\mathfrak{k}}%

\global\long\def\fn{\mathfrak{n}}%

\global\long\def\fm{\mathfrak{m}}%

\global\long\def\fp{\mathfrak{p}}%

\global\long\def\fs{\mathfrak{s}}%

\global\long\def\ft{\mathfrak{t}}%

\global\long\def\fx{\mathfrak{x}}%

\global\long\def\fv{\mathfrak{v}}%

\global\long\def\fC{\mathfrak{C}}%

\global\long\def\fD{\mathfrak{D}}%

\global\long\def\fF{\mathfrak{F}}%

\global\long\def\fJ{\mathfrak{J}}%

\global\long\def\fG{\mathfrak{G}}%

\global\long\def\fK{\mathfrak{K}}%

\global\long\def\fM{\mathfrak{M}}%

\global\long\def\fN{\mathfrak{N}}%

\global\long\def\fO{\mathfrak{O}}%

\global\long\def\fS{\mathfrak{S}}%

\global\long\def\fV{\mathfrak{V}}%

\global\long\def\fX{\mathfrak{X}}%

\global\long\def\fY{\mathfrak{Y}}%

\global\long\def\ru{\mathrm{u}}%

\global\long\def\rv{\mathbf{\mathrm{v}}}%

\global\long\def\rw{\mathrm{w}}%

\global\long\def\rx{\mathrm{x}}%

\global\long\def\ry{\mathrm{y}}%

\global\long\def\rz{\mathrm{z}}%

\global\long\def\AdGp{\mathrm{AdGp}}%

\global\long\def\Aff{\mathbf{Aff}}%

\global\long\def\Alg{\mathbf{Alg}}%

\global\long\def\age{\operatorname{age}}%

\global\long\def\Ann{\mathrm{Ann}}%

\global\long\def\Aut{\operatorname{Aut}}%

\global\long\def\B{\operatorname{\mathrm{B}}}%

\global\long\def\Bl{\mathrm{Bl}}%

\global\long\def\C{\operatorname{\mathrm{C}}}%

\global\long\def\calm{\mathrm{calm}}%

\global\long\def\can{\mathrm{can}}%

\global\long\def\center{\mathrm{center}}%

\global\long\def\characteristic{\operatorname{char}}%

\global\long\def\cjun{c\textrm{-jun}}%

\global\long\def\codim{\operatorname{codim}}%

\global\long\def\Coker{\mathrm{Coker}}%

\global\long\def\Conj{\operatorname{Conj}}%

\global\long\def\cw{\mathbf{cw}}%

\global\long\def\D{\mathrm{D}}%

\global\long\def\Df{\mathrm{Df}}%

\global\long\def\dec{\mathrm{dec}}%

\global\long\def\det{\operatorname{det}}%

\global\long\def\diag{\mathrm{diag}}%

\global\long\def\discrep#1{\mathrm{discrep}\left(#1\right)}%

\global\long\def\dom{\mathrm{dom}}%

\global\long\def\doubleslash{\sslash}%

\global\long\def\E{\operatorname{E}}%

\global\long\def\Emb{\operatorname{Emb}}%

\global\long\def\et{\textrm{ét}}%

\global\long\def\etop{\mathrm{e}_{\mathrm{top}}}%

\global\long\def\el{\mathrm{e}_{l}}%

\global\long\def\Exc{\mathrm{Exc}}%

\global\long\def\Ext{\operatorname{Ext}}%

\global\long\def\Fal{\mathrm{Fal}}%

\global\long\def\FConj{F\textrm{-}\Conj}%

\global\long\def\Fitt{\operatorname{Fitt}}%

\global\long\def\fMov{\overline{\mathfrak{Mov}}}%

\global\long\def\fPEff{\overline{\mathfrak{Eff}}}%

\global\long\def\fr{\mathrm{fr}}%

\global\long\def\Fr{\mathrm{Fr}}%

\global\long\def\Gal{\operatorname{Gal}}%

\global\long\def\GalGps{\mathrm{GalGps}}%

\global\long\def\GL{\mathrm{GL}}%

\global\long\def\Gor{\mathrm{Gor}}%

\global\long\def\Grass{\mathrm{Grass}}%

\global\long\def\gw{\mathbf{gw}}%

\global\long\def\H{\operatorname{\mathrm{H}}}%

\global\long\def\hattimes{\hat{\times}}%

\global\long\def\hatotimes{\hat{\otimes}}%

\global\long\def\Hilb{\mathrm{Hilb}}%

\global\long\def\Hodge{\mathrm{Hodge}}%

\global\long\def\Hom{\operatorname{Hom}}%

\global\long\def\hyphen{\textrm{-}}%

\global\long\def\I{\operatorname{\mathrm{I}}}%

\global\long\def\id{\mathrm{id}}%

\global\long\def\Image{\operatorname{\mathrm{Im}}}%

\global\long\def\ind{\mathrm{ind}}%

\global\long\def\injlim{\varinjlim}%

\global\long\def\Inn{\mathrm{Inn}}%

\global\long\def\iper{\mathrm{iper}}%

\global\long\def\Iso{\operatorname{Iso}}%

\global\long\def\isoto{\xrightarrow{\sim}}%

\global\long\def\J{\operatorname{\mathrm{J}}}%

\global\long\def\Jac{\mathrm{Jac}}%

\global\long\def\kConj{k\textrm{-}\Conj}%

\global\long\def\KConj{K\textrm{-}\Conj}%

\global\long\def\Ker{\operatorname{Ker}}%

\global\long\def\Kzero{\operatorname{K_{0}}}%

\global\long\def\lc{\mathrm{lc}}%

\global\long\def\lcr{\mathrm{lcr}}%

\global\long\def\lcm{\operatorname{\mathrm{lcm}}}%

\global\long\def\length{\operatorname{\mathrm{length}}}%

\global\long\def\M{\operatorname{\mathrm{M}}}%

\global\long\def\MC{\mathrm{MC}}%

\global\long\def\MHS{\mathbf{MHS}}%

\global\long\def\mld{\mathrm{mld}}%

\global\long\def\mod#1{\pmod{#1}}%

\global\long\def\Mov{\overline{\mathrm{Mov}}}%

\global\long\def\mRep{\mathbf{mRep}}%

\global\long\def\mult{\mathrm{mult}}%

\global\long\def\N{\operatorname{\mathrm{N}}}%

\global\long\def\Nef{\mathrm{Nef}}%

\global\long\def\nor{\mathrm{nor}}%

\global\long\def\nr{\mathrm{nr}}%

\global\long\def\NS{\mathrm{NS}}%

\global\long\def\op{\mathrm{op}}%

\global\long\def\orb{\mathrm{orb}}%

\global\long\def\ord{\operatorname{ord}}%

\global\long\def\P{\operatorname{P}}%

\global\long\def\PEff{\overline{\mathrm{Eff}}}%

\global\long\def\PGL{\mathrm{PGL}}%

\global\long\def\pt{\mathbf{pt}}%

\global\long\def\pur{\mathrm{pur}}%

\global\long\def\perf{\mathrm{perf}}%

\global\long\def\perm{\mathrm{perm}}%

\global\long\def\Pic{\mathrm{Pic}}%

\global\long\def\pr{\mathrm{pr}}%

\global\long\def\Proj{\operatorname{Proj}}%

\global\long\def\projlim{\varprojlim}%

\global\long\def\Qbar{\overline{\QQ}}%

\global\long\def\QConj{\mathbb{Q}\textrm{-}\Conj}%

\global\long\def\R{\operatorname{\mathrm{R}}}%

\global\long\def\Ram{\operatorname{\mathrm{Ram}}}%

\global\long\def\rank{\operatorname{\mathrm{rank}}}%

\global\long\def\rat{\mathrm{rat}}%

\global\long\def\Ref{\mathrm{Ref}}%

\global\long\def\rig{\mathrm{rig}}%

\global\long\def\rj{\mathrm{rj}}%

\global\long\def\red{\mathrm{red}}%

\global\long\def\reg{\mathrm{reg}}%

\global\long\def\rep{\mathrm{rep}}%

\global\long\def\Rep{\mathbf{Rep}}%

\global\long\def\sbrats{\llbracket s\rrbracket}%

\global\long\def\Sch{\mathbf{Sch}}%

\global\long\def\sep{\mathrm{sep}}%

\global\long\def\Set{\mathbf{Set}}%

\global\long\def\sing{\mathrm{sing}}%

\global\long\def\SL{\mathrm{SL}}%

\global\long\def\sm{\mathrm{sm}}%

\global\long\def\small{\mathrm{small}}%

\global\long\def\Sp{\operatorname{Sp}}%

\global\long\def\Spec{\operatorname{Spec}}%

\global\long\def\Spf{\operatorname{Spf}}%

\global\long\def\ss{\mathrm{ss}}%

\global\long\def\st{\mathrm{st}}%

\global\long\def\Stab{\operatorname{Stab}}%

\global\long\def\Supp{\operatorname{Supp}}%

\global\long\def\spars{\llparenthesis s\rrparenthesis}%

\global\long\def\Sym{\mathrm{Sym}}%

\global\long\def\T{\operatorname{T}}%

\global\long\def\tame{\mathrm{tame}}%

\global\long\def\tbrats{\llbracket t\rrbracket}%

\global\long\def\top{\mathrm{top}}%

\global\long\def\tors{\mathrm{tors}}%

\global\long\def\tpars{\llparenthesis t\rrparenthesis}%

\global\long\def\Tr{\mathrm{Tr}}%

\global\long\def\ulAut{\operatorname{\underline{Aut}}}%

\global\long\def\ulHom{\operatorname{\underline{Hom}}}%

\global\long\def\ulInn{\operatorname{\underline{Inn}}}%

\global\long\def\ulIso{\operatorname{\underline{{Iso}}}}%

\global\long\def\ulSpec{\operatorname{\underline{{Spec}}}}%

\global\long\def\Utg{\operatorname{Utg}}%

\global\long\def\utg{\mathrm{utg}}%

\global\long\def\Unt{\operatorname{Unt}}%

\global\long\def\Var{\mathbf{Var}}%

\global\long\def\Vol{\mathrm{Vol}}%

\global\long\def\wt{\mathrm{wt}}%

\global\long\def\Y{\operatorname{\mathrm{Y}}}%
\selectlanguage{american}%

\tableofcontents{}

\section{Introduction}

This is a sequel to the paper \cite{darda2024thebatyrevtextendashmanin}.
In that paper, we generalized the Batyrev--Manin conjecture to DM
stacks over a number field by introducing height functions associated
to raised line bundles, which refined a conjecture proposed by Ellenberg--Satriano--Zureick-Brown
\cite{ellenberg2023heights} before. Like the original Batyrev--Manin
conjecture, it predicts that under a suitable assumption, the number
of rational points with height at most $B$ grows like $CB^{a}(\log B)^{b-1}$
for some positive constants $C,a,b$, provided that we exclude points
in an ``accumulation subset.'' To determine the constants $a$ and
$b$ in a similar way as in the original conjecture, we defined orbifold
Néron--Severi spaces and orbifold pseudo-effective cones, incorporating
contributions of sectors, that is, connected components of the cyclotomic
inertia stack (called the stack of twisted 0-jets in \cite{darda2024thebatyrevtextendashmanin}).

The main subject of this paper is DM stacks over a finite field $k$
say of characteristic $p>0$ and the distribution of $F$-points on
such a stack for a global field $F$ over $k$. A DM stack over $k$
is said to be \emph{wild }if the automorphism group of some point
has order divisible by $p$. Otherwise, it is called \emph{tame}.
While tame stacks resemble stacks in characteristic zero in many ways,
wild stacks are often much more difficult to study. Many results and
theories apply only to tame stacks, or are much more complicated in
the case of wild stacks. In the context of the distribution of rational
points, the difficulty of the wild stack is illustrated by Landesman's
work \cite{landesman2024stackyheights}. He demonstrated that the
Faltings height of elliptic curves of characteristic three is not
a heigh function on the stack $\overline{\cM_{1,1}}$ in the sense
of \cite{ellenberg2023heights}. The framework of \cite{darda2024thebatyrevtextendashmanin},
as well as any other previously existing framework, appears to be
inapplicable in this context, too. 

Our first goal is to introduce a new type of height function. The
new height function is flexible enough to include many height-like
invariants as special cases. Notably, it includes the aforementioned
Faltings height of elliptic curves in characteristic three. Furthermore,
it inherits nice properties from standard height functions on varieties,
that is, a functorial property and the Northcott property. The key
idea is very simple and natural. In the preceding paper \cite{darda2024thebatyrevtextendashmanin},
we associated a height function to a line bundle equipped with a raising
function, which is a function on the set of twisted sectors, denoted
by $\pi_{0}(\cJ_{0}\cX)$. We upgrade it into a function on the space
of twisted arcs, denoted by $|\cJ_{\infty}\cX|$, and again call it
a \emph{raising function}. The necessity of this upgrade is already
evident in the case of the classifying stack $\B G$ associated with
a finite group $G$ of order divisible by $p$. The discriminant exponent
of a $G$-torsor over the local field $k\tpars$ can take infinitely
many distinct values, and thus cannot be expressed as a function on
the finite set of sectors. As a substitute to the notion of sectors,
we introduce the notion of \emph{sectoroids}. A sectoroid is a subset
of the space of twisted arcs which satisfies several properties. We
fix a subdivision of the space $|\cJ_{\infty}\cX|$ into at most countably
many sectoroids and assume that a raising function is constant on
each of those sectoroids. If $\cX$ is smooth and tame, then all the
fibers of the natural map 
\[
|\cJ_{\infty}\cX|\to\pi_{0}(\cJ_{0}\cX),
\]
which associates a sector to each twisted arc, are sectoroids. If
we use these sectoroids and if we assume that our raising function
$c$ is induced from a function on $\pi_{0}(\cJ_{0}\cX)$ via the
above map, then the height function that we introduce in this paper
is reduces to the one defined in \cite{darda2024thebatyrevtextendashmanin}. 

Our next goal is to formulate the Batyrev--Manin type conjecture
for DM stacks in positive characteristic, using our new kind of height
functions. We closely follow the formulation in \cite{darda2024thebatyrevtextendashmanin}
for DM stacks over a number field. We define the \emph{augmented Néron--Severi
space} of a DM stack $\cX$ as the direct sum
\[
\fN^{1}(\cX):=\N^{1}(\cX)_{\RR}\oplus\prod_{i}\RR[A_{i}]^{*},
\]
where $\N^{1}(\cX)_{\RR}$ is the usual Néron--Severi space and $A_{i}$'s
are sectoroids appearing in the fixed subdivision of the space of
twisted arcs. Since there can be infinitely many of $A_{i}$'s, the
augmented Néron--Severi space is of infinite dimension in general.
Nevertheless, we can define the counterpart of the pseudo-effective
cone in $\fN^{1}(\cX)$ called the \emph{augmented pseudo-effective
cone} and denoted by $\fPEff(\cX)$. We can also define a special
element $\fK_{\cX}\in\fN^{1}(\cX)$ called the \emph{augmented canonical
class,} which is a counterpart of the numerical class $[K_{X}]$ of
the canonical divisor for a variety $X$. Recall that we used the
\emph{age} invariants of sectors in the definition of the counterpart
of this element in the framework of \cite{darda2024thebatyrevtextendashmanin}.
In the definition of the augmented canonical class, we use the \emph{Gorenstein
weights} of sectoroids in place of the age invariants. The raised
line bundle $(\cL,c)$, that is, the pair of a line bundle $\cL$
and a raising function $c$, induces an element $\llbracket\cL,c\rrbracket\in\fN^{1}(\cX)$
in a natural way. Using the cone $\fPEff(\cX)$ and the two elements
$\fK_{\cX}$ and $\llbracket\cL,c\rrbracket$, we define the $a$-invariant
$a(\cL,c)$ and the $b$-invariant $b(\cL,c)$ in the same way as
in the case of varieties. Notice that since $\fN^{1}(\cX)$ is of
infinite dimension, it is a subtle problem when these invariants are
finite. The following is the main conjecture that we would like to
propose in this paper as a prototype of Batyrev--Manin type conjectures
for DM stacks in positive characteristic:
\begin{conjecture}[Conjecture \ref{conj:main}]
\label{conj:main-intro}Let $\cX$ be a geometrically irreducible
DM stack of finite type over $k$ and let $F$ be a global field over
$k$. Assume that $\cX$ has a geometrically rationally connected
coarse moduli space and that $\cX$ has only normal $\QQ$-Gorenstein
singularities. Assume that $F$-points of $\cX$ are Zariski dense.
Assume that the raised line bundle $(\cL,c)$ satisfies the conditions
corresponding to the bigness and the nefness in our framework. Assume
that $b(\cL,c)<\infty$. Finally, assume that $c$ belongs to a certain
large enough class of functions which include most important functions
appearing in applications. Then, there exists a thin subset $T\subset\cX\langle F\rangle$
such that the following asymptotic formula holds:
\[
\#\{x\in\cX\langle F\rangle\setminus T\mid H_{\cL,c}(x)\le B\}\asymp B^{a(\cL,c)}(\log B)^{b(\cL,c)-1}\quad(B\to\infty).
\]
\end{conjecture}

To the best of the authors' knowledge, even including the simplest
case of classifying stacks, the above conjecture is the first attempt
to formulate such an asymptotic formula for the number of rational
points with bounded height in a wild situation. Notice that in the
conjecture, the stack $\cX$ is not only allowed to be wild, but also
allowed to have mild singularities. Thus, our new formalism provides
a unifying framework for a wide range of objects. Assumptions in the
conjecture are stated rather loosely. This is because, at present,
the wild examples for which asymptotic formulas have been obtained
are quite limited, making it difficult to formulate precise conditions.
In particular, the conditions that the raising function $c$ should
satisfy are not very clear, though candidates of such properties are
discussed in \cite{darda2025counting}. To identify the necessary
assumptions and arrive at the final form of the conjecture, we would
need to examine more concrete examples.

Thus, while there remains the issue of making the scope of the conjecture
more precise, we present the following evidences for the conjecture
in Sections \ref{sec:Tame-smooth-stacks} to \ref{sec:Elliptic-curves}:
\begin{enumerate}
\item If $\cX$ is smooth and tame and if the raising function $c$ factors
through $\pi_{0}(\cJ_{0}\cX)$, then Conjecture \ref{conj:main-intro}
reduces to the stacky Batyrev--Manin conjecture in \cite{darda2024thebatyrevtextendashmanin}
up to minor differences in settings. 
\item When applied to classifying stacks associated to finite groups, the
conjecture matches the existing results on counting Galois extensions
for abelian $p$-groups \cite{lagemann2015distribution,lagemann2012distribution,potthast2025onthe,gundlach2024counting}. 
\item The conjecture explains Manin's conjecture for Fano varieties with
canonical quotient singularities.
\item It is compatible with taking products of stacks.
\item The conjecture holds for the moduli stack $\overline{\cM_{1,1}}$
of elliptic curves in characteristic three and for raised line bundles
of product type. 
\end{enumerate}
In a separate paper \cite{darda2025counting}, the authors prove the
conjecture for the classifying stack $\B G$ of an abelian $p$-group
under a certain condition on $c$.

The paper contains four appendices to show supplementary results,
which might be of independent interest. Firstly, we generalize the
main result of \cite{tonini2020moduliof} by allowing the group scheme
in question to be defined over the punctured formal disk rather than
over its coefficient field. Secondly, we construct the morphism $f_{\infty}\colon\cJ_{\infty}\cY\to\cJ_{\infty}\cX$
associated to a morphism $f\colon\cY\to\cX$ of DM stacks (or more
generally a morphism of formal DM stacks over a formal disk). Previously,
the morphism was constructed only when $\cX$ is an algebraic space
\cite{yasuda2024motivic2}. Moreover, if $\rho\colon\cX\to\cX^{\rig}$
is the rigidification of a DM stack, then the induced morphism $f_{\infty}\colon\cJ_{\infty}\cX\to\cJ_{\infty}\cX^{\rig}$
is countably component-wise of finite type. Thirdly, we show an asymptotic
formula for a product height on a product set. In the final appendix,
we give an erratum to the previous paper. 

To conclude the introduction, we briefly mention a few issues that
are not addressed in this paper. While considering stacks over global
fields allows us to handle more general situations, this paper does
not address them. Compared to this general situation, the situation
we deal with can be called the isotrivial case. Since even the isotrivial
case requires a substantial theoretical framework, we have chosen
to restrict ourselves to the isotrivial case to avoid introducing
further complexity. We do not attempt to specify the removed subset
$T$ in the conjecture. To do so, we would need to work with the notion
of breaking thin subsets introduced by Lehmann--Sengputa--Tanimoto
\cite{lehmann2022geometric} for DM stacks over global fields. 

\subsection*{Notation}

We work over a finite base field $k$ of characteristic $p>0$. By
$F$, we mean a global field containing $k$ as its constant field,
that is, the function field of a smooth projective and geometrically
irreducible curve over $k$, unless otherwise noted. We denote the
set of places of $F$ by $M_{F}$. For a DM stack $\cX$ over $k$,
we denote by $\cX\langle F\rangle$ the set of isomorphism classes
of $F$-points of $\cX$. When $\cX$ is a DM stack, we denote its
coarse moduli space by $X$. We denote the point set of a stack $\cX$
by $|\cX|$. 

\subsection*{Acknowledgements}

We thank Boaz Moerman and Sho Tanimoto for helpful comments. Ratko
Darda has received a funding from the European Union’s Horizon 2023
research and innovation programme under the Maria Skłodowska-Curie
grant agreement 101149785. Takehiko Yasuda was supported by JSPS KAKENHI
Grant Numbers JP21H04994, JP23H01070 and JP24K00519.

\section{Twisted formal disks and twisted arcs\label{sec:Twisted-formal-disks}}

\subsection{Twisted formal disks and Galois extensions of the field of Laurent
power series}

Let $k'/k$ be an algebraically closed field. We define the \emph{formal
disk} over $k'$ to be $\D_{k'}:=\Spec k'\tbrats$ and the \emph{punctured
formal disk} over $k'$ to be $\D_{k'}^{*}:=\Spec k'\tpars$, which
is an open subscheme of $\D_{k'}$. 
\begin{defn}
By a \emph{twisted formal disk} \emph{over }$k'$, we mean a DM stack
of the form $[E/G]$, where $E$ is a normal connected Galois cover
of $\D_{k'}$ with a Galois group $G$. 
\end{defn}

A twisted formal disk $\cE$ over $k'$ is equipped with a coarse
moduli morphism $\cE\to\D_{k'}$, which is an isomorphism over $\D_{k'}^{*}$.
If $K/k'\tpars$ is a finite Galois extension and if $\cO_{K}$ is
the integral closure of $k'\tbrats$ in $K$, then $[\Spec\cO_{K}/\Gal(K/k'\tpars)]$
is a twisted formal disk. This construction gives the following one-to-one
correspondence:
\begin{align*}
\{\text{finite Galois extension of \ensuremath{k'\tpars}}\}/{\cong} & \leftrightarrow\{\text{twisted formal disk over \ensuremath{k'}}\}/{\cong}\\
K/k'\tpars & \mapsto[\Spec\cO_{K}/\Gal(K/k'\tpars)]
\end{align*}
 
\begin{defn}
We say that a finite group $G$ is \emph{Galoisian} if it is isomorphic
to the Galois group of some Galois extension $L/k'\tpars$ with $k'$
an algebraically closed field over $k$. By $\GalGps$, we denote
a representative set of the isomorphism classes of Galoisian groups. 
\end{defn}

Let $\Lambda$ be the stack of fiberwise connected finite étale torsors
over $\D_{k}^{*}$ with locally constant ramification \cite[Definition 5.12]{yasuda2024motivic2}.
This stack is written as 
\[
\Lambda=\coprod_{G\in\GalGps}\coprod_{[\br]\in\Ram(G)/\Aut(G)}\Lambda_{[G]}^{[\br]},
\]
the disjoint union of countably many stacks $\Lambda_{[G]}^{[\br]}$.
Here $[\br]$ runs over $\Aut(G)$-orbits of ramification data for
$G$. Let $|\Lambda|$ denote the point set of $|\Lambda|$ defined
as in \cite[tag 04XE]{stacksprojectauthors2022stacksproject}. We
have the identifications 
\begin{equation}
\begin{aligned}|\Lambda| & =\varinjlim_{k'/k}\{\text{finite Galois extension over \ensuremath{\D_{k'}^{*}}}\}/{\cong}\\
 & =\varinjlim_{k'/k}\{\text{twisted formal disk over \ensuremath{k'}}\}/{\cong}.
\end{aligned}
\label{eq:|Lambda|}
\end{equation}

There exist reduced DM stacks $\Theta_{[G]}^{[\br]}$ of finite type
over $k$ such that $\Lambda_{[G]}^{[\br]}$ are the ind-perfections
of $\Theta_{[G]}^{[\br]}$, respectively \cite[Proposition 5.19 and Definition 5.20]{yasuda2024motivic2}.
In particular, the induced maps $|\Theta_{[G]}^{[\br]}|\to|\Lambda_{[G]}^{[\br]}|$
are bijective. We set
\[
\Theta:=\coprod_{G\in\GalGps}\coprod_{[\br]\in\Ram(G)/\Aut(G)}\Theta_{[G]}^{[\br]}.
\]
We have the ind-perfection morphism $\Theta\to\Lambda$, which induces
a bijection $|\Theta|\to|\Lambda|$. 
\begin{defn}
We say that a morphism $f\colon\cY\to\cX$ of DM stacks is \emph{ccft
(countably component-wise finite-type) }or $\cY$ is \emph{ccft over
$\cX$ }if there exists a countable family $\{\cY_{i}\}_{i\in I}$
of open and closed substacks $\cY_{i}\subset\cY$ such that the restrictions
$f|_{\cY_{i}}\colon\cY_{i}\to\cX$ are of finite type for all $i\in I$. 
\end{defn}

We see that $\Theta$ is ccft over $\Spec k$.

\subsection{Twisted arcs}
\begin{defn}
A \emph{quasi-nice stack over} $k$ means a normal DM stack $\cX$
over $k$ satisfying the following conditions:
\begin{enumerate}
\item $\cX$ is separated, geometrically irreducible and $\QQ$-Gorenstein
over $k$, 
\item a coarse moduli space of $\cX$ is projective over $k$,
\item $\cX$ is not $k$-isomorphic to $\Spec k$.
\end{enumerate}
\end{defn}

From now on, $\cX$ denotes a quasi-nice DM stack over $k$ and $X$
denotes its coarse moduli space. The associated morphism $\cX\to X$
is denoted by $\pi$. 
\begin{defn}
Let $k'$ be an algebraically closed field over $k$. A \emph{twisted
arc} of $\cX$ over $k'$ is a representable morphism $\cE\to\cX$
with $\cE$ a twisted formal disk over $k'$. Two twisted arcs over
$k'$, $\cE\to\cX$ and $\cE'\to\cX$, are \emph{isomorphic }if there
exists an isomorphism making the following diagram 2-commutative:
\[
\xymatrix{\cE\ar[r]\ar[d]\ar[dr] & \cX\\
\D_{k'} & \cE'\ar[u]\ar[l]
}
\]
\end{defn}

For a twisted arc $\cE\to\cX$ over $k'$, the composition $\D_{k'}^{*}\hookrightarrow\cE\to\cX$
is an $k'\tpars$-point of $\cX$. Conversely, for a $k'\tpars$-point
$\D_{k'}^{*}\to\cX$, there exists a unique twisted arc $\cE\to\cX$
extending this morphism. We can prove this in the same way as the
proof of \cite[Lemma 2.16]{darda2024thebatyrevtextendashmanin}, except
that the induced twisted formal disk is not generally of the form
$[\Spec k'\llbracket t^{1/l}\rrbracket/\mu_{l}]$. Thus, we have a
one-to-one correspondence
\[
\{\text{twisted arc of \ensuremath{\cX} over \ensuremath{k'}}\}/{\cong}\longleftrightarrow\cX\langle k'\tpars\rangle.
\]

We can construct a stack $\cJ_{\infty}\cX$ of twisted arcs \cite[Sections 7 and 16]{yasuda2024motivic2}.
This is a DM stack over $k$ and defined as the projective limit $\projlim\cJ_{n}\cX$
of stacks of twisted jets; each of $\cJ_{n}\cX$ is ccft over $k$.
For each algebraically closed field $k'/k$, a $k'$-point of $\cJ_{\infty}\cX$
corresponds to a twisted arc of $\cX$ over $k'$. Via the above one-to-one
correspondence, the point set $|\cJ_{\infty}\cX|$ is identified with
$\varinjlim_{k'/k}\cX\langle k'\tpars\rangle$.
\begin{rem}
\label{rem:not unique}In the construction of stacks $\cJ_{\infty}\cX$
and $\cJ_{n}\cX$, we need to choose some stack denoted by $\Gamma_{\cX}$,
however the point sets $|\cJ_{\infty}\cX|$, $|\cJ_{n}\cX|$ and $|\Gamma_{\cX}|$
are uniquely determined. See Appendix \ref{sec:Morphisms} for details. 
\end{rem}

\begin{example}
If $\cX$ is a scheme, then $\cJ_{\infty}\cX$ is the usual arc scheme
$\J_{\infty}\cX$ that parametrizes usual arcs $\D_{k'}\to\cX$. If
$\cX$ is an affine space $\AA_{k}^{n}$, then $\J_{\infty}\cX$ is
the projective limit of finite dimensional affine spaces $\varprojlim_{m}\AA_{k}^{m}$.
\end{example}

\begin{example}
If $\cX$ is a classifying stack $\B G$ associated to a finite étale
group scheme $G$ over $k$, then $\cJ_{\infty}\cX$ is ccft over
$k$ and its coarse moduli space is a scheme model of the P-moduli
space $\Delta_{G}$ of $G$-torsors over $\D_{k}^{*}$, which follows
from computation in the proof of \cite[Corollary 14.4]{yasuda2024motivic2}.
\end{example}

We have a natural map
\begin{equation}
\tau\colon|\cJ_{\infty}\cX|\to|\Lambda|,\quad(\cE\to\cX)\mapsto[\cE],\label{eq:tau}
\end{equation}
where $[\cE]$ is the point of $|\Lambda|$ corresponding to the twisted
formal disk corresponding to $\cE$ via (\ref{eq:|Lambda|}). Moreover,
this map is induced from a certain morphism $\cJ_{\infty}\cX\to\Gamma_{\cX}$
of DM stacks (see Section \ref{subsec:Constr-tw} for details). The
following lemma will be used in Section \ref{sec:Elliptic-curves}. 
\begin{lem}
\label{lem:twisted-arc-minimality}Let $k'$ be an algebraically closed
field. Let $x\colon\D_{k'}^{*}\to\cX$ be a morphism. Then, there
exists a unique minimal finite separable cover $E^{*}\to\D_{k'}^{*}$
such that the induced morphism $E^{*}\to\cX$ extends to a morphism
$E\to\cX$ from the integral closure $E$ of $\D_{k'}$ in $E^{*}$.
Moreover, $E\to\D_{k'}$ is a Galois cover say with Galois group $G$,
and if $\cD\to\cX$ is the twisted arc corresponding to $x$, then
$\cD$ is isomorphic to $[E/G]$ over $\D_{k'}$.
\end{lem}

\begin{proof}
Aside from its uniqueness, obviously there exists a minimal finite
separable cover $E^{*}\to\D_{k'}^{*}$ as in the lemma. In the following
commutative diagram, the solid arrows form a commutative diagram:
\[
\xymatrix{E^{*}\ar[r]\ar[d] & \D_{k'}^{*}\ar[d]\\
E\ar@{-->}[r]\ar[d] & \cD\ar[d]\\
\cX\ar[r]^{\id_{\cX}} & \cX
}
\]
Here, $\cD\to\cX$ and $E\to\cX$ are the normalizations in $\D_{k'}\to\cX$
and $E^{*}\to\cX$, respectively. From the universality of normalization
(see \cite[(6.3.5), Chapter II]{grothendieck1961elements}, \cite[Tag 035J]{stacksprojectauthors2022stacksproject}),
we have the unique dashed arrow making the whole diagram commutative.
Note that the cited references only treat the case of schemes. However,
since the normalization is compatible with étale base changes \cite[Tag 0ABP]{stacksprojectauthors2022stacksproject},
these results generalize to DM stacks. 

Suppose now that $\cD$ is written as $[E'/G]$, where $E'$ is a
regular Galois cover of $\D_{k'}$ with Galois group $G$ and consider
the following Cartesian diagram induced by the obtained morphism $E\to\cD$
and the canonical morphism $E'\to\cD$:
\[
\xymatrix{E\times_{\cD}E'\ar[r]\ar[d]^{\text{ét.}} & E'\ar[d]^{\text{ét.}}\\
E\ar[r]\ar@{-->}@/^{1.5pc}/[u]\ar@{-->}[ur] & \cD
}
\]
Since the vertical arrows are étale and finite and since the residue
field of the closed point of $E$ is algebraically closed, we see
that the projection $E\times_{\cD}E'\to E$ admits a section and composing
it with the projection $E\times_{\cD}E'\to E'$ gives a morphism $E\to E'$
factoring the morphism $E\to\cD$. But, the minimality of $E$ implies
that that the induced morphism $E\to E'$ is an isomorphism. Thus,
we have $\cD\cong[E/G]$ and $E\to\D_{k'}$ is a Galois extension
with Galois group $G$. Since the twisted formal disk $\cD$ is uniquely
determined from the given $\D_{k'}^{*}$-point $x$, we get the desired
uniqueness of minimal covers $E\to\D_{k'}$. 
\end{proof}
Let $f\colon\cY\to\cX$ be a (not necessarily representable) morphism
of quasi-nice stacks over $k$. For a twisted arc $\gamma\colon\cE\to\cY$
over an algebraically closed field $k'$, the composition $f\circ\gamma\colon\cE\to\cX$
uniquely factors as $\cE\to\cE'\to\cX$ in such a way that $\cE'\to\cX$
is a twisted arc over $k'$ and $\cE\to\cE'$ is compatible with the
morphisms $\cE\to\D_{k'}$ and $\cE'\to\D_{k'}$. This is a special
case of the canonical factorization of a morphism of DM stacks obtained
in \cite[Lemma 25]{yasuda2006motivic} (this is called the relative
coarse moduli space in \cite[Theorem 3.1]{abramovich2011twisted}).
This construction defines a map
\begin{equation}
f_{\infty}\colon|\cJ_{\infty}\cY|\to|\cJ_{\infty}\cX|,\label{eq:f_infty}
\end{equation}
which is compatible with maps $\cY(k'\tpars)\to\cX(k'\tpars)$. See
Section \ref{subsec:mor-tw} for how to realize this map as a stack
morphism $\cJ_{\infty}\cY\to\cJ_{\infty}\cX$. 

\subsection{Rigidification}

Let $\I\cX$ be the inertia stack and let $\I^{\dom}\cX$ be the union
of those connected components of the inertia stack $\I\cX$ that map
onto $\cX$. We see that $\I^{\dom}\cX$ is a subgroup stack of $\I\cX$. 
\begin{lem}
The morphism $h\colon\I^{\dom}\cX\to\cX$ is étale, in particular,
flat. 
\end{lem}

\begin{proof}
By the definition of DM stacks, the diagonal morphism $\Delta\colon\cX\to\cX\times_{k}\cX$
is unramified. It follows that the projection $\I\cX\to\cX$ is also
unramified, since it is the base change of $\Delta$ by $\Delta$.
Since $h$ is obtained by restricting $\I\cX\to\cX$ to an open and
closed substack $\I^{\dom}\cX$, we see that $h$ is also unramified.
The fiber of $h$ over every point $x\colon\Spec k'\to\cX$ with $k'$
a field is a subgroup scheme of $\ulAut(x)$ and hence an étale finite
group scheme. Thus, $h$ is a representable, unramified and equidimensional
morphism of finite type over the normal (hence reduced and geometrically
unibranch) DM stack $\cX$ and has smooth fibers. From \cite[Corollary 17.5.6 and Theorem 17.6.1]{grothendieck1967elements},
$h$ is étale. 
\end{proof}
From the lemma and \cite[Theorem A.1]{abramovich2008tamestacks},
we can associate the rigidification of $\cX$ with respect to $\I^{\dom}\cX$. 
\begin{defn}
By the \emph{rigidification} of $\cX$, we mean the rigidification
of $\cX$ with respect to $\I^{\dom}\cX$; we denote it by $\cX^{\rig}$
and the natural morphism $\cX\to\cX^{\rig}$ is denoted by $\rho$. 
\end{defn}

The stack $\cX^{\rig}$ is a DM stack having the trivial generic stabilizer.
The coarse moduli morphism $\pi\colon\cX\to X$ factors as 
\[
\cX\xrightarrow{\rho}\cX^{\rig}\xrightarrow{\pi^{\rig}}X,
\]
where $\pi^{\rig}$ is the coarse moduli morphism for $\cX^{\rig}$.
In particular, $\cX$ and $\cX^{\rig}$ share the coarse moduli space
$X$. The morphism $\rho$ is an étale gerbe and the morphism $\pi^{\rig}$
is birational (see also \cite[Proposition 2.1]{olsson2007aboundedness}).
\begin{prop}
\label{prop:neutral}Let $k'/k$ be an algebraically closed field
and let $\Spec k'\tpars\to\cX^{\rig}$ be a $k'\tpars$-point. Then,
the fiber product $\Spec k'\tpars\times_{\cX^{\rig}}\cX$ is isomorphic
to the classifying stack $\B\cH$ over $k'\tpars$, where $\cH$ is
a finite étale group scheme over $k'\tpars$.
\end{prop}

\begin{proof}
Take an arbitrary $k'\tpars$-point $u\colon\Spec k'\tpars\to\cX^{\rig}$
and put $\cY:=\cX\times_{\cX^{\rig},u}\Spec k'\tpars$. Then, $\cY$
is a finite étale gerbe over $\Spec k'\tpars$. It suffices to show
that $\cY$ is neutral. From \cite[Theorem 2]{greenberg1966rational},
the field $k'\tpars$ is $C_{1}$ and hence has cohomological dimension
$\le1$. From \cite[Lemma 4.4]{bresciani2024fieldsof}, a gerbe over
$k'\tpars$ is neutral, which implies the proposition.
\end{proof}
As a special case of map (\ref{eq:f_infty}), we get a map 
\[
\rho_{\infty}\colon|\cJ_{\infty}\cX|\to|\cJ_{\infty}\cX^{\rig}|.
\]
In fact, (for suitable choices of $\Gamma_{\cX}$; see Remark \ref{rem:not unique})
there exists a ccft morphism 
\[
\rho_{\infty}\colon\cJ_{\infty}\cX\to\cJ_{\infty}\cX^{\rig}
\]
again denoted by $\rho_{\infty}$ which induces the above map of point
sets. 
\begin{rem}
From Proposition \ref{prop:neutral}, for a point $x\colon\Spec k'\tpars\to\cX^{\rig}$,
the fiber of $\rho_{\infty}$ over $x$ is identified with $|\Delta_{\cH}|$.
The moduli stack $\Delta_{\cH}$ admits a P-moduli space represented
by the disjoint union of countably many $k$-varieties in the following
cases:
\end{rem}

\begin{enumerate}
\item $\cH$ is constant \cite{tonini2023moduliof}.
\item The base change $\cH_{k'\tpars^{\sep}}$ is of the form
\[
\prod_{i=1}^{n}(H_{i}\rtimes C_{i}),
\]
where $H_{i}$ are $p$-groups and $C_{i}$ are tame cyclic groups
(\cite{tonini2020moduliof} and Appendix \ref{sec:Moduli-of-formal}
of the present paper). 
\end{enumerate}

\section{Gorenstein weights and sectoroids\label{sec:Gorenstein-weights}}

\subsection{Gorenstein weights}

We keep the notation from the last section. In particular, $\cX$
is a quasi-nice stack, $\cX^{\rig}$ is its rigidification and $X$
is their common coarse moduli space. Thus, $X$ is a normal projective
variety. There exists a unique $\QQ$-divisor $B$ on $X$ such that
$\cX\to(X,B)$ is crepant. We now recall the Gorenstein motivic measure
on the arc space $\J_{\infty}X$ associated to the pair $(X,B)$,
which was introduced by Denef--Loeser \cite{denef2002motivic} in
the case of varieties with Gorenstein singularities. Let $d:=\dim\cX=\dim X$.
For a stable subset $C\subset|\J_{\infty}X|$ of level $n$, its motivic
measure is defined as 
\[
\mu_{X}(C):=\{\pi_{n}(C)\}\LL^{-nd}.
\]

\begin{example}
If $X$ is smooth, then $\mu_{X}(|\J_{\infty}X|)=\{X\}$. Some literature
uses a normalization of the motivic measure different by the factor
of $\LL^{d}$ so that $\mu_{X}(|\J_{\infty}X|)=\{X\}\LL^{-d}$.
\end{example}

For a positive integer $r$ as above, we define a fractional coherent
ideal $\cI_{(X,B),r}\subset K(X)$ by requiring the following equality
of subsheaves of the constant sheaf $\omega_{X}^{\otimes r}\otimes K(X)$:
\[
\cI_{(X,B),r}\cdot\cO_{X}(r(K_{X}+B))=\Image\left(\left(\bigwedge^{d}\Omega_{X}\right)^{\otimes r}\to\omega_{X}^{\otimes r}\otimes K(X)\right).
\]

\begin{defn}
A subset $C\subset|\J_{\infty}X|$ is said to be \emph{Gorenstein
measurable} (with respect to the log pair $(X,B)$) if the motivic
function $\LL^{\frac{1}{r}\ord(\cI_{(X,B),r})}$ is integrable on
$C$. For a Gorenstein measurable subset $C$, we define its \emph{Gorenstein
motivic measure} by
\[
\mu_{(X,B)}^{\Gor}(C):=\int_{C}\LL^{\frac{1}{r}\ord(\cI_{(X,B),r})}\,d\mu_{X}.
\]
\end{defn}

We can also define a motivic measure on $|\cJ_{\infty}\cX^{\rig}|$
as in \cite{yasuda2024motivic2}. The map $\pi_{\infty}^{\rig}\colon|\cJ_{\infty}\cX^{\rig}|\to|\J_{\infty}X|$
is bijective outside measure zero subsets. 
\begin{defn}
We say that a subset $C\subset|\cJ_{\infty}\cX^{\rig}|$ is \emph{Gorenstein
measurable }if $\pi_{\infty}^{\rig}(C)$ is Gorenstein measurable
(with respect to $(X,B)$). For a Gorenstein measurable subset $C\subset|\cJ_{\infty}\cX^{\rig}|$,
we define its \emph{Gorenstein measure }to be
\[
\mu_{\cX^{\rig}}^{\Gor}(C):=\mu_{(X,B)}^{\Gor}(\pi_{\infty}^{\rig}(C)).
\]
For a Gorenstein measurable subset $C$, we define its \emph{Gorenstein
weight }to be
\[
\gw(C):=\dim\mu_{\cX^{\rig}}^{\Gor}(C)-d.
\]
\end{defn}

\begin{example}
If $\cX$ is smooth, then the space $|\J_{\infty}\cX|$ of untwisted
arcs is Gorenstein measurable and has Gorenstein weight $0$. 
\end{example}

\begin{defn}
We call a subset $C\subset|\cJ_{\infty}\cX|$ \emph{Gorenstein pseudo-measurable
}if 
\begin{enumerate}
\item $\rho_{\infty}(C)$ is Gorenstein measurable, and 
\item there exists an integer $m_{C}\in\ZZ_{\ge0}$ such that for every
point $\gamma\in\rho_{\infty}(C)$, $C\cap(\rho_{\infty})^{-1}(\gamma)$
is a constructible subset of $(\rho_{\infty})^{-1}(\gamma)$ of dimension
$m_{C}$. 
\end{enumerate}
For a Gorenstein pseudo-measurable subset $C\subset|\cJ_{\infty}\cX|$,
we define its \emph{Gorenstein weight }to be 
\[
\gw(C):=\gw(\rho_{\infty}(C))+m_{C}.
\]
\end{defn}

\begin{example}
If $\cX=\B G$ for an étale finite group scheme $G$ over $k$ and
if we write $\cJ_{\infty}\cX$ as the coproduct $\coprod_{j=0}^{\infty}\cW_{j}$
of countably many DM stacks $\cW_{j}$ of finite type over $k$, then
a Gorenstein pseudo-measurable subset of $|\cJ_{\infty}\cX|$ is a
constructible subset of $\coprod_{j=0}^{m}|\cW_{j}|$ for $m\gg0$
and its Gorenstein weight is nothing but its dimension. 
\end{example}

\subsection{Sectoroids}
\begin{defn}
\label{def:equiv}Two points $|\cJ_{\infty}\cX|$ are\emph{ }said
to be\emph{ equivalent} if they are represented by twisted arcs $\gamma\colon\cE\to\cX$
and $\gamma'\colon\cE'\to\cX$ defined over the same algebraically
closed field $k'$ such that there exists an isomorphism $\alpha\colon\cE\xrightarrow{\sim}\cE'$
over $k'$ with $\gamma\cong\gamma'\circ\alpha$. When two points
$\beta$ and $\beta'$ are equivalent, we write $\beta\sim\beta'$. 
\end{defn}

\begin{defn}
\label{def:sectoroid}We call a subset $C\subset|\cJ_{\infty}\cX^{\rig}|$
a \emph{sectoroid }if the following conditions are satisfied:
\begin{enumerate}
\item $C$ is closed under the equivalence relation of Definition \ref{def:equiv},
\item $C$ is the union of two subsets $C_{0}$ and $C_{1}$ such that
\begin{enumerate}
\item $C_{0}$ is an irreducible stable subset and the function $\ord\cI_{(X,B),r}$
is constant on $C_{0}$,
\item $C_{1}$ is Gorenstein measurable with $\gw(C_{1})<\gw(C_{0})$.
\end{enumerate}
\end{enumerate}
\end{defn}

If $C\subset|\cJ_{\infty}\cX^{\rig}|$ is a sectoroid, then it is
Gorenstein measurable and $\gw(C)=\gw(C_{0})$ for a stable subset
$C_{0}$ as above. Moreover, the Gorenstein measure of $C$ is of
the form
\[
\{V\}\LL^{m}+\sum_{i=0}^{\infty}\{V_{i}\}\LL^{m_{i}}\quad(\dim V_{i}+m_{i}<\dim V+m),
\]
where $V$ is an irreducible $k$-variety and $V_{i}$ are $k$-varieties,
and $m,m_{i}\in\frac{1}{r}\ZZ$. 

\begin{defn}
\label{def:sectoroid-1}We call a subset $C\subset|\cJ_{\infty}\cX|$
a \emph{sectoroid }if
\begin{enumerate}
\item $C$ is Gorenstein pseudo-measurable,
\item $\rho_{\infty}(C)\subset|\cJ_{\infty}\cX^{\rig}|$ is a sectoroid,
\item there exists an integer $m_{C}\in\ZZ_{\ge0}$ such that for every
point $\gamma\in\rho_{\infty}(C)$, $C\cap(\rho_{\infty})^{-1}(\gamma)$
is an \emph{irreducible} constructible subset of $(\rho_{\infty})^{-1}(\gamma)$
of dimension $m_{C}$. 
\end{enumerate}
\end{defn}

\begin{defn}
\label{def:sectoroid-subdivision}Let $\cX^{\circ}\subset\cX$ be
the open substack obtained by removing the image of $(\I\cX)\setminus(\I^{\dom}\cX)$
from the smooth locus $\cX_{\sm}$ of $\cX$. A \emph{sectoroid subdivision
}of $|\cJ_{\infty}\cX|$ is a collection $\{A_{i}\}_{i\in I}$ of
countably many subsets $A_{i}\subset|\cJ_{\infty}\cX|$ satisfying
the following conditions:
\begin{enumerate}
\item the subsets $A_{i}$, $i\in I$ are mutually disjoint and $|\cJ_{\infty}\cY|=\bigsqcup_{i\in I}A_{i}$,
\item the index set $I$ contains two distinguished elements $0$ and $\infty$, 
\item the subsets $A_{i}$, $i\in I\setminus\{\infty\}$ are sectoroids, 
\item there exists an open dense substack $\cU\subset\cX$ such that $|\J_{\infty}\cU|\subset A_{0}$,
\item $\gw(A_{0})=0$ and $\gw(|\J_{\infty}\cX^{\circ}|\setminus A_{0})<-1$
,
\item $A_{\infty}$ is contained in $|\cJ_{\infty}\cY|$ for some proper
closed substack $\cY\subsetneq\cX$ (possibly $A_{\infty}=\emptyset$). 
\end{enumerate}
For a sectoroid subdivision $\{A_{i}\}_{i\in I}$, we call $A_{0}$
the \emph{non-twisted sectoroid, $A_{i}$, $i\in I\setminus\{0,\infty\}$
}the\emph{ twisted sectoroids, }and\emph{ }$A_{\infty}$ the \emph{exceptional
arc locus.} A closed substack $\cY$ as in (5) is called an \emph{exceptional
substack} for the sectoroid subdivision. 
\end{defn}

Note that since $A_{i}$, $i\ne\infty$ are closed under the equivalence,
so is $A_{\infty}$. Note also that $|\cJ_{\infty}\cY|$ is a measurable
subset of measure zero, which follows from \cite[Lemma 10.8]{yasuda2024motivic2}. 
\begin{rem}
We impose Condition (5) to avoid having one prime divisor contribute
doubly in the definition of augmented Néron--Severi space (Definition
\ref{def:augmented-NS}).
\end{rem}

\section{Heights associated to raised line bundles\label{sec:Heights}}
\begin{defn}
We say that a function $c\colon|\cJ_{\infty}\cX|\to\RR\cup\{\infty\}$
is a \emph{pseudo-raising function} if it is invariant under the equivalence
relation defined in Definition \ref{def:equiv}. We say that a pseudo-raising
function $c\colon|\cJ_{\infty}\cX|\to\RR\cup\{\infty\}$ is a \emph{raising
function} if there exists a sectoroid subdivision $\{A_{i}\}_{i\in I}$
of $|\cJ_{\infty}\cX|$ such that 
\begin{enumerate}
\item for each $i\in I$, the restriction $c|_{A_{i}}$ of $c$ to $A_{i}$
is constant, 
\item for $i\in I\setminus\{\infty\}$, $c(A_{i})<\infty$, and
\item $c(A_{0})=0$.
\end{enumerate}
We say that a raising function $c$ and a sectoroid subdivision $\{A_{i}\}_{i\in I}$
as above are \emph{compatible. }
\end{defn}

Let $F$ be a global field containing $k$ as the constant field.
For a place $v$ of $F$, let $F_{v}$ be the completion of $F$ at
$v$, let $F_{v}^{\nr}$ be its maximal unramified extension and let
$\widehat{F_{v}^{\nr}}$ be its completion. We fix an isomorphism
$\widehat{F_{v}^{\nr}}\cong\overline{k}\tpars$. For each $F$-point
$x\in\cX(F)$ and for each place $v$ of $F$, we get the induced
$\overline{k}\tpars$-point 
\[
x_{v}\colon\D_{\overline{k}}^{*}=\Spec\overline{k}\tpars=\Spec\widehat{F_{v}^{\nr}}\to\Spec F\xrightarrow{x}\cX.
\]
It has a unique extension to a twisted arc 
\[
\widetilde{x_{v}}\colon\cE\to\cX
\]
over $\overline{k}$. The equivalence class $[\widetilde{x_{v}}]\in|\cJ_{\infty}\cX|/{\sim}$
is independent of the choice of the isomorphism $\widehat{F_{v}^{\nr}}\cong\overline{k}\tpars$.Thus,
for a raising function $c$, the value $c([\widetilde{x_{v}}])$ is
well-defined and will be abbreviated as $c(\widetilde{x_{v}})$ or
simply as $c(x_{v})$. 
\begin{defn}
Let $c$ be a pseudo-raising function and let $x\in\cX(F)$. Let $M_{F}$
be the set of places of $F$ and let $q_{v}$ be the cardinality of
the residue field of $F_{v}$. We define the \emph{height of $x$
with respect to $c$} to be 
\[
H_{c}(x):=\prod_{v\in M_{F}}q_{v}^{c(x_{v})}\in\RR\cup\{\infty\}.
\]
\end{defn}

Note that if $\{A_{i}\}_{i\in I}$ is a sectoroid subdivision of $|\cJ_{\infty}\cX|$,
then for each $F$-point $x\in\cX(F)$ and for almost all places $v$,
the twisted arc $\widetilde{x_{v}}$ defines a point of the non-twisted
sectoroid $A_{0}$, thanks to condition (5) of Definition \ref{def:sectoroid-subdivision}.
This shows that we have $c(x_{v})=0$ except for finitely many places
$v$. In particular, if $c(\widetilde{x}_{v})<\infty$ for every $v$,
in particular, if $x$ does not lie in an exceptional closed subset
$\cY$ of a sectoroid subdivision compatible with $c$, then $H_{c}(x)<\infty$. 
\begin{example}[Discriminant exponents]
\label{exa:discriminant}Examples of raising functions are induced
from functions on $|\Lambda|$. For example, consider the discriminant
exponent function $\bd\colon|\Lambda|\to\ZZ_{\ge0}$, which maps a
point of $|\Lambda|$ to the discriminant exponent of the corresponding
Galois extension of $k'\tpars$. This is a locally constructible function
\cite[Theorem 9.8]{tonini2023moduliof}. This implies that he composite
function 
\[
\bd\circ\tau\colon|\cJ_{\infty}\cX|\xrightarrow{\tau}|\Lambda|\xrightarrow{\bd}\ZZ_{\ge0}
\]
is a raising function. Indeed, it is constantly zero on $|\J_{\infty}\cX|$
and we just need to decompose each fiber $(\bd\circ\tau)^{-1}(n)$
into sectoroids in order to find a compatible sectoroid subdivision. 
\end{example}

\begin{defn}
A \emph{raised (resp.~pseudo-raised) line bundle }on $\cX$ means
the pair $(\cL,c)$ of a line bundle $\cL$ on $\cX$ and a raising
(resp.~pseudo-raising) function $c\colon|\cJ_{\infty}\cX|\to\RR\cup\{\infty\}$.
We define the \emph{height of $x$ with respect to a pseudo-raised
line bundle $(\cL,c)$ }to be 
\[
H_{\cL,c}(x):=H_{\cL}(x)\cdot H_{c}(x)\in\RR\cup\{\infty\}.
\]
\end{defn}

\begin{prop}
Let $f\colon\cY\to\cX$ be a morphism of quasi-nice stacks over $k$.
Let $(\cL,c)$ be a pseudo-raised line bundle on $\cX$ and let $(f^{*}\cL,f^{*}c)$
be the induced pseudo-raised line bundle on $\cY$ with $f^{*}c:=c\circ f_{\infty}$.
Then, for $y\in\cY(F)$, we have
\[
H_{f^{*}\cL,f^{*}c}(y)=H_{\cL,c}(f(y)).
\]
\end{prop}

\begin{proof}
This follows from $H_{f^{*}\cL}(y)=H_{\cL}(f(y))$ and $H_{f^{*}c}(y)=H_{c}(f(y)).$
The former is a well-known property of the usual height function and
it is straightforward to generalize it to the stacky setting. The
latter is obvious from the definition as follows:
\[
H_{f^{*}c}(y)=\prod_{v}q_{v}^{f^{*}c(y_{v})}=\prod_{v}q_{v}^{c(f(x))}=H_{c}(x).
\]
\end{proof}

\section{The Northcott property\label{sec:The-Northcott-property}}

\subsection{Statement}
\begin{defn}
We say that a raising function $c\colon|\cJ_{\infty}\cX|\to\RR$ is
\emph{strongly positive} if there exists $\epsilon\in\RR_{>0}$ such
that for every $\gamma\in|\cJ_{\infty}\cX|$, we have $c(\gamma)\ge\epsilon\cdot(\bd\circ\tau)(\gamma)$,
where $\bd$ is the discriminant exponent function on $|\Lambda|$
(Example \ref{exa:discriminant}).
\end{defn}

In this section, we prove the following proposition. 
\begin{thm}[The Northcott property]
\label{thm:Northcott}Suppose that $\cL$ corresponds to an ample
$\QQ$-line bundle on the coarse moduli space $X$ and that $c$ is
a strongly positive raising function. Then, for every $B\in\RR$,
the set 
\[
S_{\cL,c}(B):=\{x\in\cX\langle F\rangle\mid H_{\cL,c}(x)\le B\}
\]
is finite.
\end{thm}

The proof of this theorem will be completed at the end of this section.

\subsection{Relating rational points with torsors}

First, we explain our strategy. It is relatively easy to see that
the image of $S_{\cL,c}(B)$ in $X(F)$ is finite. The difficult part
is to prove that each fiber of the map
\[
S_{\cL,c}(B)\to X(F)
\]
is finite. Consider an $F$-point $x\colon\Spec F\to X$ and the corresponding
morphism $C\to X$ from the irreducible smooth proper curve with the
function field $F$. To understand the preimage of the $F$-point
$x$ by the map $\cX\langle F\rangle\to X(F)$, we consider the normalization
$\cC$ of $(C\times_{X}\cX)_{\red}$. Let $\cC^{\rig}$ be its rigidification
and let $\overline{\cC}$ be the common coarse moduli space of $\cC$
and $\cC^{\rig}$. The natural morphism $\overline{\cC}\to C$ is
a universal homeomorphism. 

Let us now suppose that the $F$-point of $X$ lifts to an $F$-point
$\widetilde{x}\colon\Spec F\to\cX$. Then, the morphism $\overline{\cC}\to C$
admits a section and hence it is an isomorphism. Thus, $C$ is a coarse
moduli space of $\cC$ and $\cC^{\rig}$ and the morphism $\cC^{\rig}\to C$
is birational. In other words, the stack $\cC^{\rig}$ is the curve
$C$ with finitely many stacky points. Note that $\cC^{\rig}$ is
not generally isomorphic to $C\times_{X}\cX^{\rig}$.

The canonical $F$-point of $C$ uniquely lifts to $\cC^{\rig}$.
The fiber product $\Spec F\times_{\cC^{\rig}}\cC$ is a gerbe over
$\Spec F$. Since this gerbe has a section by the assumption, it is
a neutral gerbe and can be written as $\B\cH$ for some finite étale
group scheme $\cH$ over $\Spec F$. The $F$-point $\widetilde{x}\in\cX(F)$
corresponds to a morphism $\Spec F\to\B\cH$ given by an $\cH$-torsor
$T_{\widetilde{x}}\to\Spec F$. We summarize our construction so far
in the following diagram: 
\[
\xymatrix{\B\cH\ar[dd]\ar[r] & \cC\ar[d]\ar[r] & \cX\ar[dd]\\
 & \cC^{\rig}\ar[d]\\
\Spec F\ar[r]\ar@/_{1.5pc}/[rr]^{x}\ar@{-->}@/_{1.5pc}/[uurr]_{\widetilde{x}}\ar@{-->}@/^{1.5pc}/[uu]^{T_{\widetilde{x}}} & C\ar[r] & X
}
\]
Here the rectangles in this diagram are Cartesian. We have got the
one-to-one correspondence
\begin{equation}
\begin{aligned}\{\text{\ensuremath{F}-points of \ensuremath{\cX} mapping to \ensuremath{x}}\}/{\cong} & \leftrightarrow\{\text{\ensuremath{\cH}-torsors over \ensuremath{\Spec F}}\}/{\cong}\\
\widetilde{x} & \mapsto T_{\widetilde{x}}.
\end{aligned}
\label{eq:corr}
\end{equation}

\subsection{From the boundedness of heights to the boundedness of discriminants}

Let $v\in C$ be a closed point and let $\overline{v}$ be the induced
$\overline{k}$-point of $C$. We define 
\[
C_{\overline{v}}:=\Spec\widehat{\cO_{C_{\overline{k}},\overline{v}}}=\Spec\cO_{\widehat{F_{v}^{\nr}}}.
\]
Fixing an isomorphism $\widehat{F_{v}^{\nr}}\cong\overline{k}\tpars$,
we identify $C_{\overline{v}}$ with $\D=\Spec\overline{k}\tbrats$.
We then define $\cC_{\overline{v}}:=C_{\overline{v}}\times_{C}\cC$
and $\cC_{\overline{v}}^{\rig}:=C_{\overline{v}}\times_{C}\cC^{\rig}$.
Let us now fix a lift $\widetilde{x}\in\cX(F)$ of $x\in X(F)$, which
induces an $F$-point of $\cC$. We then get the induced $\widehat{F_{v}^{\nr}}$-point
$\widetilde{x}_{\cC}\colon\Spec\widehat{F_{v}^{\nr}}=\Spec\overline{k}\tpars\to\cC$,
which in turn extends to a twisted arc 
\[
\gamma_{\cC,x,v}\colon\cE=[E/G]\to\cC.
\]
Here $E$ is a regular connected Galois cover of the formal disk $\D$
with the Galois group $G$. The composition 
\[
\gamma_{x,v}\colon\cE\xrightarrow{\gamma_{\cC,x,v}}\cC\to\cX
\]
is the twisted arc of $\cX$ extending the induced $\widehat{F_{v}^{\nr}}$-point
of $\cX$. We will relate the discriminant exponent $\bd_{E/\D_{\overline{k}}}$
of $E\to\D$ with the discriminant exponent $\bd_{T_{\widetilde{x},\overline{v}}/\D^{*}}$
of the torsor 
\[
T_{\widetilde{x},\overline{v}}:=T_{\widetilde{x}}\otimes_{F}\widehat{F_{v}^{\nr}}\to\Spec\widehat{F_{v}^{\nr}}=\D^{*},
\]
which is equal to the discriminant exponent $\bd_{T_{\widetilde{x},v}/F_{v}}$
of the torsor $T_{\widetilde{x},v}:=T_{\widetilde{x}}\otimes_{F}F_{v}\to\Spec F_{v}$. 

 Let $\cH_{\D_{\overline{k}}^{*}}:=\cH\otimes_{F}\widehat{F_{v}^{\nr}}=\cH\times_{\Spec F}\D^{*}$.
Then $\D^{*}\times_{\cC^{\rig}}\cC\cong\B\cH_{\D^{*}}$ is an open
substack of $\cC_{\overline{v}}$. The twisted arc $\gamma_{\cC,x,v}\colon\cE\to\cC$
induces a morphism $\cE\to\cC_{\overline{v}}$ and the composition
\[
\D^{*}\hookrightarrow\cE\to\cC_{\overline{v}}
\]
maps into the open substack $\B\cH_{\D_{\overline{k}}^{*}}\subset\cC_{\overline{v}}$.
The obtained morphism $\D^{*}\to\B\cH_{\D^{*}}$ corresponds to the
torsor $T_{\widetilde{x},\overline{v}}\to\D^{*}$. We summarize the
obtained stacks and morphisms as follows:
\[
\xymatrix{\B\cH_{\D^{*}}\ar[rr]\ar[dd] &  & \cC_{\overline{v}}\ar[d]\ar[r] & \cC\ar[d]\ar[r] & \cX\ar[dd]\\
 & \cE=[E/G]\ar[ur]\ar[d] & \cC_{\overline{v}}^{\rig}\ar[d] & \cC^{\rig}\ar[d]\\
\D^{*}\ar@{^{(}->}[r]\ar@{^{(}->}[ur]\ar@{^{(}->}[urr]\ar@{-->}@/^{1.5pc}/[uu]^{T_{\widetilde{x},\overline{v}}} & \D\ar@{=}[r] & C_{\overline{v}}\ar[r] & C\ar[r] & X
}
\]
Here rectangles are Cartesian. In what follows, we use the following
fact several times:
\begin{lem}
For an algebraically closed field $k'$, if $f\colon T\to\D_{k'}$
is a finite étale morphism, then there exists an isomorphism $T\cong\coprod_{i=1}^{r}\D_{k'}$
for some $r\in\ZZ_{\ge0}$ compatible with $f$ and the projection
$\coprod_{i=1}^{r}\D_{k'}\to\D_{k'}$. 
\end{lem}

\begin{lem}
\label{lem:torsor-factor}Let $E^{*}:=E\times_{\D}\D^{*}$. The morphism
$E^{*}\to\D^{*}$ factors through the torsor $\widehat{T_{\widetilde{x}}}\to\D^{*}$. 
\end{lem}

\begin{proof}
Let us write $\cC_{\overline{v}}^{\rig}=[U/B]$, where $U=\Spec\overline{k}\sbrats$
and $B$ is a finite group effectively acting on $U$. We take the
base-changes $E_{U}$, $\cE_{U}$ and $\cC_{\overline{v},U}$ of $E$,
$\cE$ and $\cC_{\overline{v}}$ by the atlas $U\to\cC_{\overline{v}}^{\rig}$,
respectively. We claim that $\cC_{\overline{v}}\times_{\cC_{\overline{v}}^{\rig}}U$
is isomorphic to $\B H_{U}$ for a constant finite group scheme $H_{U}$
over $U$. Indeed, $\cC_{\overline{v}}\times_{\cC_{\overline{v}}^{\rig}}U\to U$
is a finite étale gerbe. The closed point $\Spec\overline{k}\hookrightarrow U$
lifts $\cC_{\overline{v}}\times_{\cC_{\overline{v}}^{\rig}}U$. From
the formal étaleness, we see that the morphism $\cC_{\overline{v}}\times_{\cC_{\overline{v}}^{\rig}}U\to U$
has a section. Thus, the gerbe is neutral and can be written as $\B\cG$
for some finite étale group scheme over $U$. Since $U=\Spec\overline{k}\sbrats$,
every finite étale group scheme over $U$ is constant. 

Since $E_{U}\to E$ is an étale $B$-torsor over $E=\Spec\overline{k}\llbracket u\rrbracket$,
it is a trivial torsor and can be written as $E\times B$. The scheme
$E_{U}$ also has the $G$-action induced from the action on $E$.
Let us fix a connected component $E_{0}$ of $E_{U}=E\times B$ and
let $S$ be its stabilizer for the $G$-action. We have
\[
\cE_{U}=[E\times B/G]=[E_{0}/S].
\]
Note that if $U^{*}\subset U$ is the punctured open subset, then
the natural morphism $\cE_{U}\to U$ restricts to an isomorphism from
an open substack of $\cE_{U}$ onto $U^{*}\subset U$. We get the
following commutative diagram:
\[
\xymatrix{E_{U}=E\times B\ar[r]^{B}\ar[d]^{G} & E\ar[d]^{G}\\
\cE_{U}=[E/S]\ar[r]^{B}\ar[d]^{\textrm{rep.}}\ar@(l,l)[dd]_{\text{bir.}} & \cE=[E/G]\ar[d]^{\textrm{rep.}} & \D^{*}\ar@{_{(}->}[l]\ar[d]^{T_{\widetilde{x},\overline{v}}}\ar@{=}@(r,r)[dd]\\
\B H_{U}\ar[r]^{B}\ar[d] & \cC_{\overline{v}}\ar[d] & \B\cH_{\D^{*}}\ar@{_{(}->}[l]\ar[d]\\
U\ar[r]^{B} & \cC_{\overline{v}}^{\rig}=[U/B] & \D^{*}\ar@{_{(}->}[l]
}
\]
Here labels ``rep.,'' ``bir.'' and ``$B$'' mean representable
morphisms, a birational morphism and $B$-torsors respectively. Let
$\pi\colon\D^{*}\to\B\cH_{\D^{*}}$ and $\pi_{U}\colon U\to\B H_{U}=[U/H]$
be the canonical atlases, where $H$ denotes the finite group corresponding
to the constant group scheme $H_{U}$. 

Consider the following diagram. 
\[
\xymatrix{E\times_{\B H_{U}}U=E\times H\ar[rr]\ar[d] &  & E\ar[d]\ar@{-->}[dll]^{\diamondsuit}\ar@{-->}@/_{1.0pc}/[ll]\ar@{..>}[dddr]^{\spadesuit}\\
\cE_{U}\times_{\B H_{U}}U\ar[d]\ar[rr]\ar@{..>}[ddr]\sp(0.4){\heartsuit} &  & \cE_{U}\ar[d]\ar@{..>}[ddr]\\
U\ar[rr]^{\pi_{U}}\ar@{..>}[ddr] &  & \B H_{U}\ar@{..>}[ddr]\\
 & T_{\widetilde{x},\overline{v}}\ar[rr]^{\clubsuit}\ar[d] &  & \D^{*}\ar[d]^{T_{\widetilde{x},\overline{v}}}\\
 & \D^{*}\ar[rr]^{\pi} &  & \B\cH_{\D^{*}}
}
\]
Here dotted arrows are morphisms defined only on open dense substacks.
Rectangles are Cartesian, while parallelograms are Cartesian when
restricted to suitable open substacks. We need to show that the arrow
$\spadesuit$ factorizes through the arrow $\clubsuit$. Since the
projection $E\times_{\B H_{U}}U\to E$ is an étale $H$-torsor, it
should be trivial and we have an isomorphism $E\times_{\B H_{U}}U\cong E\times H$
over $E$. In particular, we have a section $E\to E\times_{\B H_{U}}U$
and the induced morphism $E\to\cE_{U}\times_{\B H_{U}}U$ as expressed
by dashed arrows in the last diagram. The diagram shows the factorization
$\spadesuit=\clubsuit\circ\heartsuit\circ\diamondsuit$. We have proved
the lemma. 
\end{proof}
\begin{cor}
\label{cor:compare-d}With the above notation, we have
\[
\bd_{T_{\widetilde{x},v}/F_{v}}\le\ord(\cH)\cdot\bd_{E/\D}.
\]
\end{cor}

\begin{proof}
From Lemma \ref{lem:torsor-factor}, $T_{\widetilde{x},\overline{v}}$
has a connected component $(T_{\widetilde{x},\overline{v}})_{0}$
dominated by $E^{*}$ and hence 
\[
\bd_{(T_{\widetilde{x},\overline{v}})_{0}/\D^{*}}\le\bd_{E/\D}.
\]
Since $T_{\widetilde{x},\overline{v}}$ is the disjoint union of at
most $\ord(\cH)$ copies of $(T_{\widetilde{x},\overline{v}})_{0}$,
we have
\[
\bd_{T_{\widetilde{x},v}/F_{v}}=\bd_{T_{\widetilde{x},\overline{v}}/\D^{*}}\le\ord(\cH)\cdot\bd_{(T_{\widetilde{x},\overline{v}})_{0}/\D^{*}}\le\ord(\cH)\cdot\bd_{E/\D}.
\]
\end{proof}

\subsection{Completing the proof of Theorem \ref{thm:Northcott}}

If we bound the raised height $H_{\cL,c}$, then $H_{\cL}$ is also
bounded. Thus, the image of $S_{\cL,c}(B)$ in $X\langle F\rangle$
is a finite set. We need to show that fibers of $S_{\cL,c}(B)\to X(F)$
are finite. Let us fix a point $x\in X(F)$ as before. From Corollary
\ref{cor:compare-d}, if we denote the discriminant of $T_{\widetilde{x}}\to\Spec F$
by $\fD_{T_{\widetilde{x}}/F}$, then 
\begin{align*}
\fD_{T_{\widetilde{x}}/F} & =\prod_{v}q_{v}^{\bd_{T_{\widetilde{x},v}}}\\
 & \le\prod_{v}q_{v}^{\ord(\cH)\cdot(\bd\circ\tau(\gamma_{x,v}))}\\
 & \le\prod_{v}q_{v}^{\ord(\cH)\cdot\epsilon^{-1}\cdot c(\gamma_{x,v})}\\
 & \le H_{c}(\widetilde{x})^{\ord\cH\cdot\epsilon^{-1}}.
\end{align*}
Since $H_{c}$ is bounded on $S_{\cL,c}(B)$, the map $\widetilde{x}\mapsto\fD_{T_{\widetilde{x}}/F}$
is also bounded on $S_{\cL,c}(B)$. This shows that there are only
finitely many underlying $F$-schemes for torsors $T_{\widetilde{x}}$
corresponding to lifts $\widetilde{x}\in S_{\cL,c}(B)$ of $x$. From
\cite[Lemma 4.7]{darda2024thebatyrevtextendashmanin}, there are only
finitely many torsor-structures which can be given to each $F$-scheme.
Thus, there are only finite many $\cH$-torsors of the form $T_{\widetilde{x}}\to\Spec F$,
$\widetilde{x}\in S_{\cL,c}(B)$, which completes the proof of Theorem
\ref{thm:Northcott}. 

\section{Augmented pseudo-effective cones and the main conjecture\label{sec:Augmented-Main-Conj}}

We fix a sectoroid subdivision $\{A_{i}\}_{i\in I}$ of $|\cJ_{\infty}\cX|$
and let $I^{*}:=I\setminus\{0,\infty\}$. Let $\N^{1}(\cX)_{\RR}$
be the Néron--Severi space of $\cX$ and $\N_{1}(\cX)_{\RR}$ its
dual. Namely, they are the real vector spaces consisting of numerical
classes $\RR$-divisors and 1-cycles respectively. 
\begin{defn}
\label{def:augmented-NS}Let $[A_{i}]$, $i\in I^{*}$ and $[A_{i}]^{*}$,
$i\in I^{*}$ be indeterminants. We define the \emph{augmented} \emph{Néron--Severi
space} of $\cX$ to be 
\begin{align*}
\fN^{1}(\cX) & :=\N^{1}(\cX)_{\RR}\oplus\prod_{i\in I^{*}}\RR[A_{i}]^{*}.
\end{align*}
and the \emph{augmented 1-cycle space }of $\cX$ to be
\begin{align*}
\fN_{1}(\cX) & :=\N_{1}(\cX)_{\RR}\oplus\bigoplus_{i\in I^{*}}\RR[A_{i}].
\end{align*}
We then write an element of $\fN^{1}(\cX)$ as a formal sum
\[
D+\sum_{i\in I}c_{i}[A_{i}]^{*}\quad(D\in\N^{1}(\cX)_{\RR},c_{i}\in\RR).
\]
Similarly for an element of $\fN_{1}(\cX)$. 
\end{defn}

There exists the following natural pairing between $\fN_{1}(\cX)$
and $\fN^{1}(\cX)$: 
\begin{align*}
\fN_{1}(\cX)\times\fN^{1}(\cX) & \to\RR\\
\left(C+\sum_{i\in I^{*}}b_{i}[A_{i}],D+\sum_{i\in I^{*}}c_{i}[A_{i}]^{*}\right) & \mapsto(C,D)+\sum_{i\in I^{*}}b_{i}c_{i}
\end{align*}
This pairing makes $\fN^{1}(\cX)$ the algebraic dual of $\fN_{1}(\cX)$,
that is, the space of linear functions on $\fN_{1}(\cX)$. We give
the weak topologies to $\fN_{1}(\cX)$ and $\fN^{1}(\cX)$ relative
to this pairing; the weak topology on $\fN_{1}(\cX)$ is the weakest
topology such that for every $\beta\in\fN^{1}(\cX)$, $\alpha\mapsto(\alpha,\beta)$
is continuous. Similarly for the weak topology on $\fN^{1}(\cX)$.
See \cite[pages II.42 and II.51]{bourbaki1987topological} for more
details. Since the pairing between $\fN_{1}(\cX)$ and $\fN^{1}(\cX)$
are separating, from \cite[p. II.46]{bourbaki1987topological}, $\fN^{1}(\cX)$
and $\fN_{1}(\cX)$ are topological duals to each other; $\fN^{1}(\cX)$
consists of all the continuous linear functions of $\fN_{1}(\cX)$
and vice versa.
\begin{defn}[{\cite[Definition 8.4]{darda2024thebatyrevtextendashmanin}}]
Let $k'$ be an algebraically closed field over $k$. By a \emph{stacky
curve} on $\cX_{k'}$, we mean a representable morphism $\cC\to\cX_{k'}$
over $k'$, where $\cC$ is a one-dimensional, irreducible, proper,
and smooth DM stack over $k'$ which has trivial generic stabilizer.
A \emph{covering family of stacky curves} on $\cX_{k'}$ is the pair
$(\pi\colon\widetilde{\cC}\to T,\widetilde{f}\colon\widetilde{\cC}\to\cX_{k'})$
of $k'$-morphisms of DM stacks such that 
\begin{enumerate}
\item $\pi$ is smooth and surjective,
\item $T$ is an integral scheme of finite type over $k'$,
\item for each point $t\in T(k')$, the morphism
\[
\widetilde{f}|_{\pi^{-1}(t)}\colon\pi^{-1}(t)\to\cX_{k'}
\]
is a stacky curve on $\cX_{k'}$,
\item $\widetilde{f}$ is dominant. 
\end{enumerate}
\end{defn}

For a stacky curve $f\colon\cC\to\cX_{k'}$, the source stack $\cC$
has only finitely many stacky points. For a stacky point $v$, let
$\cC_{v}:=[\Spec\widehat{\cO_{\cC,v}}/\Stab(v)]$ be the formal completion
of $\cC$ at $v$, where $\Stab(v)$ denotes the stabilizer group
at $v$. The natural morphism 
\[
f_{v}\colon\cC_{v}\to\cX_{k'}\to\cX
\]
can be regarded as a twisted arc of $\cX$ over $k'$, provided that
we choose an isomorphism of a coarse moduli space with the formal
disk $\D_{k'}$. There exists a unique $i_{v}\in I\setminus\{0\}$
such that $[f_{v}]\in A_{i_{v}}$, which is independent of the choice
of an isomorphism. 

\begin{defn}
We define the \emph{augmented class} of a stacky curve $f\colon\cC\to\cX_{k'}$
to be 
\[
\llbracket f\rrbracket=\llbracket\cC\rrbracket:=[\cC]+\sum_{\nu}[A_{i_{v}}]\in\fN_{1}(\cX).
\]
Here $v$ runs over the stacky points of $\cC$ and we put $[A_{i_{v}}]=0$
when $A_{i_{v}}=A_{\infty}$. For a covering family $(\pi\colon\widetilde{\cC}\to T,\widetilde{f}\colon\widetilde{\cC}\to\cX_{\overline{k}})$,
we define its \emph{augmented class $\llbracket\widetilde{\cC}\rrbracket$
}to be $\llbracket\widetilde{f}|_{\pi^{-1}(\overline{\eta})}\rrbracket$,
where $\overline{\eta}$ denotes the geometric generic point of $T$. 
\end{defn}

\begin{defn}
We define the \emph{cone of augmented classes of moving stacky curves},
denoted by $\fMov_{1}(\cX)$, to be the closure (with respect to the
weak topology of $\fN_{1}(\cX)$) of 
\[
\sum_{\widetilde{\cC}}\RR_{\ge0}\llbracket\widetilde{\cC}\rrbracket,
\]
where $\overline{\cC}$ runs over covering families of stacky curves
on $\cX_{\overline{k}}$. 
\end{defn}

\begin{defn}
Using the pairing between $\fN^{1}(\cX)$ and $\fN_{1}(\cX)$, we
define the \emph{augmented pseudo-effective cone }$\fPEff(\cX)$ to
be the dual cone of $\fMov_{1}(\cX)$:
\[
\fPEff(\cX):=\{\alpha\in\fN^{1}(\cX)\mid\forall\beta\in\fMov_{1}(\cX),\,(\beta,\alpha)\ge0\}.
\]
\end{defn}

\begin{rem}
Note that for each $\beta\in\fN_{1}(\cX)$, the inequality $(\beta,\alpha)\ge0$
defines a closed subset of $\fN^{1}(\cX)$. It follows that $\fPEff(\cX)$
is also a closed subset of $\fN^{1}(\cX)$. We also see that the cone
$\fPEff(\cX)$ is identical to 
\[
Q:=\{\alpha\in\fN^{1}(\cX)\mid\text{for every moving stack curve \ensuremath{\cC}, }(\llbracket\cC\rrbracket,\alpha)\ge0\}.
\]
Indeed, we obviously have $Q\supset\fPEff(\cX)$. To see the opposite
inclusion, let $\alpha\in Q$ and $\beta\in\fMov_{1}(\cX)$. There
exists a sequence $\beta_{i}=\sum_{i=1}^{n_{i}}b_{i,j}\llbracket\cC_{i,j}\rrbracket$
with $\beta_{i}\to\beta$. Then, we have $(\alpha,\beta_{i})\to(\alpha,\beta)$.
Since $(\alpha,\beta_{i})\ge0$ for every $i$, we get $(\alpha,\beta)\ge0$.
This shows $\alpha\in\fPEff(\cX)$ and $Q\subset\fPEff(\cX)$. 
\end{rem}

\begin{defn}
For a raised line bundle $(\cL,c)$ with $c$ compatible with the
fixed sectoroid subdivision $\{A_{i}\}$, we can define its \emph{augmented
class} 
\[
\llbracket\cL,c\rrbracket:=[\cL]+\sum_{i\in I^{*}}c(A_{i})[A_{i}]^{*}.
\]
We define the\emph{ augmented canonical class} as follows:
\[
\fK_{\cX}:=[\omega_{\cX}]+\sum_{i\in I^{*}}(-\gw(A_{i})-1)[A_{i}]^{*}.
\]
\end{defn}

\begin{defn}
We say that a raised line bundle $(\cL,c)$ is \emph{nef} if for every
stacky curve $f\colon\cC\to\cX$, we have $(\llbracket\cL,c\rrbracket,\llbracket f\rrbracket)\ge0$.
We say that a raised line bundle $(\cL,c)$ is \emph{big} if $\llbracket\cL,c\rrbracket$
is in the interior of $\fPEff(\cX)$ and if for $a\gg0$, $a\llbracket\cL,c\rrbracket+\fK_{\cX}$
is contained in $\fPEff(\cX)$. Suppose that $\fK_{\cX}\notin\fPEff(\cX)$.
For a big raised line bundle $(\cL,c)$, we define 
\[
a(\cL,c):=\inf\{a'\in\RR_{\ge0}\mid a'\llbracket\cL,c\rrbracket+\fK_{\cX}\in\fPEff(\cX)\}\in\RR_{\ge0}.
\]
We define $b(\cL,c)$ to be the dimension of the linear subspace of
$\fN_{1}(\cX)$ spanned by 
\[
\fMov_{1}(\cX)\cap(a(\cL,c)\llbracket\cL,c\rrbracket+\fK_{\cX})^{\perp},
\]
which is either a non-negative integer or the infinity. 
\end{defn}

\begin{defn}
Let $\cX$ be a quasi-nice DM stack over $k$ and let $F$ be a global
field over $k$. A subset $T\subset\cX\langle F\rangle$ is said to
be \emph{thin }if there exist finitely many morphisms $f_{i}\colon\cY_{i}\to\cX_{F}$,
$1\le i\le n$, of DM stacks of finite type over $F$ such that $T\subset\bigcup_{i=1}^{n}f_{i}(\cY_{i}\langle F\rangle)$
and for each $i$, 
\end{defn}

\begin{enumerate}
\item $\cY_{i}$ is irreducible,
\item $f_{i}$ is representable and generically finite onto its image,
\item $f_{i}$ is not birational, that is, there is no open dense substack
$\cU_{i}\subset\cY$ such that $f_{i}|_{\cU_{i}}\colon\cU_{i}\to\cX_{F}$
is an open immersion. 
\end{enumerate}
We are now ready to formulate a Batyrev--Manin type conjecture for
quasi-nice DM stacks over $k$. 
\begin{conjecture}
\label{conj:main}Let $\cX$ be a quasi-nice DM stack over $k$ and
let $F$ be a global field over $k$. We fix a sectoroid subdivision
$\{A_{i}\}_{i\in I}$ of $|\cJ_{\infty}\cX|$. Let $(\cL,c)$ be a
raised line bundle which is compatible with $\{A_{i}\}_{i\in I}$
and big. Additionally, we assume:
\begin{enumerate}
\item $\fK_{\cX}\notin\fPEff(\cX)$,
\item $(\cL,c)$ is nef and big,
\item $b(\cL,c)<\infty$,
\item \label{enu:raising}$c$ belongs to a suitable large enough class
of functions, 
\item $X$ is geometrically rationally connected, and
\item $F$-points of $\cX$ are Zariski dense.
\end{enumerate}
Then, there exists a thin subset $T\subset\cX\langle F\rangle$ such
that the following asymptotic formula holds:
\[
\#\{x\in\cX\langle F\rangle\setminus T\mid H_{\cL,c}(x)\le B\}\asymp B^{a(\cL,c)}(\log B)^{b(\cL,c)-1}\quad(B\to\infty)
\]
\end{conjecture}

\begin{rem}
We impose Condition (\ref{enu:raising}) in the above conjecture to
avoid pathologies which might arise for artificially constructed bizarre
functions. For the case where $\cX$ is the classifying stack of an
abelian $p$-group, some classes of nice functions are proposed in
\cite{darda2025counting}. 
\end{rem}

\begin{rem}
It would be possible to take an approach free of the choice of a sectoroid
subdivision, for example, by replacing the factor $\bigoplus_{i\in I}\RR[A_{i}]^{*}$
of $\fN^{1}(\cX)$ with the space of real-valued functions on $|\cJ_{\infty}^{*}\cX|$
satisfying a certain condition. Such an approach would have some theoretical
advantages. However, we take the present more ad-hoc approach in this
paper, which would require less preparation. 
\end{rem}

\section{Tame smooth stacks\label{sec:Tame-smooth-stacks}}

Suppose that our quasi-nice stack $\cX$ is smooth and tame, that
is, the stabilizer of every geometric point has order prime to $p$.
It would be quite natural to expect that the framework of \cite[Definition 9.1]{darda2024thebatyrevtextendashmanin}
is valid also in this situation without any substantial change. We
now explain how that framework becomes a special case of the one considered
in the present paper.

Recall that the stack of twisted 0-jets (also known as the cyclotomic
inertia stack), denoted by $\cJ_{0}\cX$, is the stack parametrizing
representable morphisms $\B\mu_{l}\to\cX$ with $l$ varying over
positive integers. For an algebraically closed field $k'$ over $k$,
a twisted arc over $k'$ of $\cX$ is of the form $[\Spec k'\llbracket t^{1/l}\rrbracket/\mu_{l}]\to\cX$.
The composite morphism 
\[
\B\mu_{l,k'}\hookrightarrow[\Spec k'\llbracket t^{1/l}\rrbracket/\mu_{l}]\to\cX
\]
is then a $k'$-point of $\cJ_{0}\cX$. Thus, we have a natural map
\[
|\cJ_{\infty}\cX|\to|\cJ_{0}\cX|.
\]
Note that this is the $n=0$ case of truncation maps $|\cJ_{\infty}\cX|\to|\cJ_{n}\cX|$
to twisted $n$-jets. 
\begin{rem}
The term “twisted 0-jet” above is used in the sense of \cite{yasuda2006motivic,yasuda2004twisted},
but not in the sense of \cite{yasuda2024motivic2}. See \cite[Remark 15.6]{yasuda2024motivic2}
for details on the difference. 
\end{rem}

\begin{defn}
Connected components of $|\cJ_{0}\cX|$ are called \emph{sectors }and
the set of sectors is denoted by $\pi_{0}(\cJ_{0}\cX)$. There is
a special sector called the \emph{non-twisted sector}, which is naturally
isomorphic to $\cX$; we denote the non-twisted sector again by $\cX$.
For a sector $\cY$, we denote its preimage in $|\cJ_{\infty}\cX|$
by $\widetilde{\cY}$. Then $\widetilde{\cX}=|\J_{\infty}\cX|$. 
\end{defn}

We choose $\{\widetilde{\cY}\}_{\cY\in\pi_{0}(\cJ_{0}\cX)}$ as our
sectoroid subdivision. The index set $I$ is now $\pi_{0}(\cJ_{\infty}\cX)$
and we have $A_{0}=\widetilde{\cX}$ and $A_{\infty}=\emptyset$.
With the notation from \cite{yasuda2006motivic} and from computation
in \cite[page 753]{yasuda2006motivic}, we have
\begin{align*}
\gw(\widetilde{\cY}) & =\dim\cY-\mathrm{sht(\cY)}-\dim\cX\\
 & =-\age(\cY).
\end{align*}
Therefore, 
\[
\fK_{\cX}=[\omega_{\cX}]+\sum_{\cY\in\pi_{0}(\cJ_{0}\cX)\setminus\{\cX\}}(\age(\cY)-1)[\widetilde{\cY}]^{*}.
\]
This coincides with the orbifold canonical class defined in \cite[Definition 9.1]{darda2024thebatyrevtextendashmanin}
via the obvious isomorphism $\fN^{1}(\cX)\cong\N_{\orb}^{1}(\cX)_{\RR}$.

\section{Classifying stacks\label{sec:Classifying-stacks}}

\subsection{The general case}

Consider the case where $\cX=\B G$ for a finite étale group scheme
$G$ over $k$. By definition, for a field $k'$ over $k$, the $k'\tpars$-point
of $\cX$ corresponds to a $G$-torsor over $\D_{k'}^{*}$. If the
base-change $G_{\overline{k}}$ to an algebraic closure $\overline{k}$
of $k$ is the semi-direct product $H\rtimes C$ of a $p$-group and
a tame cyclic group $C$, then there exists a moduli stack defined
over $k$, denoted by $\Delta_{G}$, which is an inductive limit of
finite-type DM stacks \cite{tonini2020moduliof} (see Appendix \ref{sec:Moduli-of-formal}
for generalization to the case where $G$ is defined over $k\tpars$).
When $G$ is a general finite étale group $G$ over $k$, we have
a weaker result that there exists a P-moduli space of $G$-torsors
over $\D^{*}$, which is represented by the disjoint union of countably
many $k$-varieties \cite{tonini2023moduliof}. 

For our purpose, it is harmless to think that $\Delta_{G}$ \emph{is
}the disjoint union $\coprod_{j\in J}Z_{j}$ of countably many $k$-varieties
$Z_{j}$. We may identify $|\cJ_{\infty}\cX|$ with $|\Delta_{G}|$.
The set $|\J_{\infty}\cX|$ of \emph{untwisted} arcs is then identified
with $\{o\}$, the singleton corresponding to the trivial $G$-torsor.
In this situation, a sectoroid in $|\Delta_{G}|$ is a constructible
subset of $\coprod_{j\in J'}Z_{j}$ for a finite subset $J'\subset J$
which has only one irreducible component of the maximal dimension.
For a sectoroid $A\subset|\Delta_{G}|$, we have
\[
\gw(A)=\dim A.
\]

Let us now fix a sectoroid subdivision $\{A_{i}\}_{i\in I}$ with
$A_{0}=\{o\}$ and $A_{\infty}=\emptyset$. Since $\N_{1}(X)_{\RR}=\N^{1}(X)_{\RR}=0$,
we have
\[
\fN_{1}(\cX)=\bigoplus_{i\in I^{*}}\RR[A_{i}]\text{ and }\fN^{1}(\cX)=\prod_{i\in I^{*}}\RR[A_{i}]^{*}.
\]
The canonical augmented class is given by
\[
\fK_{\cX}=\sum_{i\in I^{*}}(-\dim A_{i}-1)[A_{i}]^{*}.
\]

\begin{prop}
\label{prop:mov-BG}We have 
\[
\fMov_{1}(\cX)=\sum_{i\in I^{*}}\RR_{\ge0}[A_{i}]\text{ and }\fPEff(\cX)=\prod_{i\in I^{*}}\RR_{\ge0}[A_{i}]^{*}.
\]
\end{prop}

\begin{proof}
Since 
\[
\sum_{i\in I^{*}}\RR_{\ge0}[A_{i}]=\bigcap_{i\in I^{*}}\{\alpha\mid(\alpha,[A_{i}]^{*})\ge0\},
\]
the cone $\sum_{i\in I^{*}}\RR_{\ge0}[A_{i}]$ is closed relative
to the weak topology. Since the augmented class of every stacky curve
lies in $\sum_{i\in I^{*}}\RR_{\ge0}[A_{i}]$, we have $\fMov_{1}(\cX)\subset\sum_{i\in I^{*}}\RR_{\ge0}[A_{i}]$.
To show the opposite inclusion, it is enough to show that for every
$i\in I^{*}$, $[A_{i}]\in\fMov_{1}(\cX)$. Let $\D_{\overline{k},0}^{*}$
and $\D_{\overline{k},\infty}^{*}$ be two copies of formal disks
over $\overline{k}$ which are obtained by taking the formal completions
of $\PP_{\overline{k}}^{1}$ at $0$ and $\infty$, respectively.
Let us take an arbitrary $G$-torsor $T_{\infty}\to\D_{\overline{k},\infty}^{*}$.
From Harbater \cite{harbater1980moduliof2} and Katz-Gabber \cite[Theorem 2.1.6]{katz1986localtoglobal},
the given $G$-torsor $T_{\infty}\to\D_{\overline{k},\infty}^{*}$
can be lifted to a torsor $T\to\PP_{\overline{k}}^{1}\setminus\{0,\infty\}.$
We have the corresponding morphism $\PP_{\overline{k}}^{1}\setminus\{0,\infty\}\to\cX_{\overline{k}}$,
which uniquely extends to a stacky curve $\cC\to\cX_{\overline{k}}$.
Note that since $\cX$ is zero-dimensional, every stacky curve is
automatically a covering family of stacky curves. The only potential
stacky points of $\cC$ are $0$ and $\infty$. The twisted arc $\cC_{\infty}\to\cX_{\overline{k}}$
at $\infty$ corresponds to the point of $|\Delta_{G}|$ given by
the original torsor $T_{\infty}\to\D_{\overline{k}}^{*}$, while the
twisted arc $\cC_{\infty}\to\cX_{\overline{k}}$ at $0$ corresponds
to the point given by the torsor $T_{0}\to\D_{\overline{k},0}^{*}$
obtained by localizing $T$ around $0$. Let $A(T_{0})$ denote the
sectoroid $A_{i}$ with $[T_{0}]\in A_{i}$, and similarly for $A(T_{\infty})$.
Then, we have
\[
\llbracket\cC\rrbracket=[A(T_{0})]+[A(T_{\infty})].
\]
Thus, we get $[A(T_{0})]+[A(T_{\infty})]\in\fMov_{1}(\cX)$, while
what we want is $[A(T_{\infty})]\in\fMov_{1}(\cX)$. 

Let $S_{0}\to\D_{\overline{k},0}^{*}$ be a totally ramified $H$-torsor
for some finite group $H$ which trivialize the $G$-torsor $T_{0}\to\D_{\overline{k},0}^{*}$
and let $S\to\PP_{\overline{k}}^{1}$ be an $H$-torsor extending
$S_{0}$. By a suitable coordinate change of $\PP_{\overline{k}}^{1}$,
we may assume that $S\to\PP_{\overline{k}}^{1}$ is étale over $\infty$.
Let $\cS\to\cX_{\overline{k}}$ to be the stacky curve obtained by
extending the composite rational map $S\to\PP_{\overline{k}}^{1}\dasharrow\cX_{\overline{k}}$.
From the construction, the $|H|$ points of $\cS$ over $\infty$
are the only stacky points of $\cS$. Moreover, all the twisted arcs
of $\cX_{\overline{k}}$ derived from these points give the same sectoroid
class $[A(T_{\infty})]$. Thus, $\llbracket\cS\rrbracket=|H|\cdot[A(T_{\infty})]\in\fMov_{1}(\cX)$.
We conclude that for every $i\in I^{*}$, $[A_{i}]\in\fMov_{1}(\cX)$,
as desired. 
\end{proof}
The $a$- and $b$-invariants can be computed as follows. Let $c$
be a raising function compatible with $\{A_{i}\}_{i\in I}$ and let
$\cL$ be a line bundle on $\cX$. Note that since $\N^{1}(\cX)_{\RR}=0$,
the choice of a line bundle does not cause any difference. From Proposition
\ref{prop:mov-BG},
\begin{equation}
\begin{aligned}a(\cL,c) & =\inf\{\alpha\in\RR\mid\forall i\in I^{*},\,\alpha c_{i}-\dim A_{i}-1\ge0\}\\
 & =\sup_{i\in I^{*}}\frac{\dim A_{i}+1}{c_{i}}.
\end{aligned}
\label{eq:a}
\end{equation}
\begin{align*}
a(\cL,c) & =\inf\{\alpha\in\RR\mid\forall i\in I^{*},\,\alpha c_{i}-\dim A_{i}-1\ge0\}\\
 & =\sup_{i\in I^{*}}\frac{\dim A_{i}+1}{c_{i}}.
\end{align*}
Let 
\[
I_{0}:=\left\{ i\in I^{*}\bigmid\frac{\dim A_{i}+1}{c_{i}}=a(L,c).\right\} 
\]
Then, 
\[
(a(L,c)\llbracket L,c\rrbracket+\fK_{\cX})^{\perp}\cap\fMov_{1}(\cX)=\sum_{i\in I_{0}}\RR_{\ge0}[A_{i}]
\]
and hence 
\begin{equation}
b(L,c)=\#I_{0}.\label{eq:b}
\end{equation}

\begin{rem}
\label{rem:well-def}We also have the following expressions which
are independent of the chosen sectoroid subdivision: 
\begin{align*}
a(\cL,c) & =\sup_{o\ne x\in|\Delta_{G}|}\frac{\dim\overline{\{x\}}+1}{c(x)},\\
b(\cL,c) & =\#x\left\{ \in|\Delta_{G}|\setminus\{o\}\bigmid\frac{\dim\overline{\{x\}}+1}{c(x)}=\alpha(\cL,c)\right\} .
\end{align*}
In particular, the $a$- and $b$-invariants are independent of the
sectoroid subdivision in the case of classifying stacks.
\end{rem}

\begin{rem}
\label{rem:pathological}If $a(\cL,c)$ is not equal to $(\dim A_{i}+1)/c_{i}$
for any $i\in I^{*}$, equivalently if the supremum is not maximum,
then Conjecture \ref{conj:main} together with Formula (\ref{eq:b})
predicts the asymptotic ``$\asymp B^{a(\cL,c)}(\log B)^{-1}$.''
This appears strange, since the log factor usually has non-negative
exponents in arithmetic statistics. Maybe we should exclude such a
case by carefully choosing the class of functions in Condition (\ref{enu:raising})
of Conjecture \ref{conj:main}.
\end{rem}

\begin{rem}
A good candidate for a sectoroid subdivision of $|\Delta_{G}|$ is
the decomposition by the ramification filtration. However, there is
a raising function of geometric origin which is not compatible with
that subdivision \cite[Section 5]{yamamoto2021pathological}. 
\end{rem}

\begin{rem}
An interesting class of raising functions on $|\Delta_{G}|$ is the
weight function $\bw$ and a relative of it denoted by $\bv$ associated
to a representation of $G$; the former was introduced in \cite{yasuda2017towardmotivic}
and the latter in \cite{wood2015massformulas}. As discussed in these
cited papers, these functions are closely related to quotient singularities
and hence geometrically important. Thanks to \cite[Theorem 9.8]{tonini2023moduliof},
they are indeed raising functions. 
\end{rem}

\subsection{Counting with conductors and Artin--Schreier conductors; the case
of abelian $p$-groups}

Let $G$ be an abelian $p$-group with exponent $p^{e}$. We show
that results of Lagemann \cite{lagemann2012distribution,lagemann2015distribution}
and Gundlach \cite{gundlach2024counting} on counting $G$-extensions
or $G$-torsors by conductors and Artin--Schreier conductors match
with Conjecture \ref{conj:main} and formulas (\ref{eq:a}) and (\ref{eq:b}). 

Let $r_{i}(G)$ be its $p^{i}$-rank, which is defined by 
\[
r_{i}(G):=\log_{p}(p^{i-1}G:p^{i}G).
\]
If we write $G=\bigoplus_{j=1}^{l}C_{p^{n_{j}}}$ with $C_{l}$ denoting
the cyclic group of order $l$, then 
\[
r_{i}(G)=\#\{j\mid n_{j}\ge i\}.
\]
These numbers form a descending sequence
\[
r_{1}(G)\ge r_{2}(G)\ge\cdots\ge r_{e}(G)>r_{e+1}(G)=0.
\]

From the local class field theory, for a finite extension $\FF_{q}/k$,
a $G$-torsor over $\D_{\FF_{q}}^{*}$ corresponds to a continuous
homomorphism $\FF_{q}\tpars^{*}\to G$. 
\begin{defn}
The \emph{conductor exponent} $\bc(x)$ of a $G$-torsor $x$ over
$\D_{\FF_{q}}^{*}$ is the least integer $n\ge0$ such that the corresponding
map $\FF_{q}\tpars^{*}\to G$ kills $1+\fm^{n}$ with $\fm$ the maximal
ideal of $\FF_{q}\tbrats$.
\end{defn}

\begin{lem}
There exists a unique function $\bc\colon|\Delta_{G}|\to\ZZ_{\ge0}$
which is compatible with the functions $\bc\colon\Delta_{G}\langle\FF_{q}\rangle\to\ZZ_{\ge0}$
defined above.
\end{lem}

\begin{proof}
First consider the case where $G$ is cyclic. From \cite[Proposition 2.7]{tanno2021thewild},
for each algebraically closed field $k'/k$, the set $\Delta_{G}\langle k'\rangle$
is decomposed into countably many $k'$-varieties (denoted by $G\textrm{-}\mathrm{Cov}(D^{*};\mathbf{j})$
in the cited paper). Since the decompositions for all algebraically
closed fields $k'$ are compatible, we get a decomposition of $|\Delta_{G}|$.
From \cite[Theorem 3.9]{tanno2021thewild}, the conductor exponent
is constant on each component $G\textrm{-}\mathrm{Cov}(D^{*};\mathbf{j})$,
which shows the lemma in this case. 

In the general case, given a homomorphism $\chi\colon G\to H$ of
abelian $p$-groups, we have a morphism of moduli stacks $\Delta_{\chi}\colon\Delta_{G}\to\Delta_{H}$.
The desired function $\bc$ on $|\Delta_{G}|$ can be defined by
\[
\bc(\gamma):=\max_{\substack{\chi\colon G\to H\\
H\colon\text{ cyclic}
}
}\bc_{H}\circ\Delta_{\chi}(\gamma).
\]
Here $\bc_{H}$ is the conductor exponent on $|\Delta_{H}|$. This
shows that for each $n$, $\bc^{-1}(n)$ is locally constructible.
To show that this is in fact constructible, equivalently, quasi-compact,
we only need to observe that if we writ $G$ as the product $G=H_{1}\times\cdots\times H_{l}$
of cyclic groups $H_{i}$, then $\bc^{-1}(n)$ is embedded into $\prod_{i=1}^{l}\bc_{H_{i}}^{-1}(\{0,\dots,n\})$. 
\end{proof}
\begin{defn}
The height function $H_{\bc}$ on $\cX\langle F\rangle$ associated
to the conductor exponent $\bc$ is called \emph{conductors} of $G$-torsors
over $F$.
\end{defn}

\begin{lem}[{\cite[Lemma 4]{kluners2020theconductor}}]
\label{lem:kluners}For a positive integer $n$, let $U_{n,q}:=(1+\fm)/(1+\fm^{n})$
and let $f:=\log_{p}(q)$. We have 
\[
r_{i}(U_{n,q})=f\left(\left\lfloor \frac{n-1}{p^{i-1}}\right\rfloor -\left\lfloor \frac{n-1}{p^{i}}\right\rfloor \right).
\]
\end{lem}

\begin{lem}[{\cite[Equation (10)]{delsarte1948fonctions}, \cite[p.~310]{lagemann2015distribution}}]
\label{lem:=000023Hom}Let $A$ be a finitely generated abelian group.
Then 
\[
\#\Hom(A,G)=p^{\sum_{i\ge1}r_{i}(A)r_{i}(G)}.
\]
\end{lem}

\begin{cor}
\label{cor:num-tors}For an integer $n\ge2$, the number of $G$-torsors
over $\D_{\FF_{q}}^{*}$with conductor exponent $n$ is 
\[
\#G\cdot\left(q^{\sum_{i\ge1}\left(\left\lfloor \frac{n-1}{p^{i-1}}\right\rfloor -\left\lfloor \frac{n-1}{p^{i}}\right\rfloor \right)r_{i}(G)}-q{}^{\sum_{i\ge1}\left(\left\lfloor \frac{n-2}{p^{i-1}}\right\rfloor -\left\lfloor \frac{n-2}{p^{i}}\right\rfloor \right)r_{i}(G)}\right).
\]
\end{cor}

\begin{proof}
Recall that we have 
\[
\FF_{q}((t))^{*}=\ZZ\times\FF_{q}^{*}\times(1+\fm)
\]
(for example, see \cite[Chapter II, Proposition 5.3]{neukirch1999algebraic}).
Giving a torsor of conductor exponent $\le n$ corresponds to giving
a map
\[
\ZZ\times\FF_{q}^{*}\times U_{n,q}\to G
\]
and hence corresponds to a triple $(\ZZ\to G,\FF_{q}^{*}\to G,U_{n,q}\to G)$.
Because of orders, there are no nontrivial maps $\FF_{q}^{*}\to G$.
There are $\#G$ maps $\ZZ\to G$. From Lemma \ref{lem:=000023Hom},
the desired number of torsors is 
\[
\#G\cdot(\#\Hom(U_{n},G)-\#\Hom(U_{n-1},G))=\#G\cdot\left(p^{\sum_{i}r_{i}(U_{n,q})r_{i}(G)}-p^{\sum_{i}r_{i}(U_{n-1,q})r_{i}(G)}\right).
\]
The formula of the lemma follows from Lemma \ref{lem:kluners}.
\end{proof}

Let $h\in\ZZ_{\ge0}$ be the minimal integer such that $p^{h}G$ is
cyclic, which is characterized by 
\[
r_{i}(G)\begin{cases}
\le1 & (i>h)\\
>1 & (i\le h).
\end{cases}
\]

\begin{lem}
We have $\bc^{-1}(n)=\emptyset$ if and only if $p^{e}\mid(n-1)$.
\end{lem}

\begin{proof}
We need to show
\[
\sum_{i=1}^{e}r_{i}(U_{n})r_{i}(G)-\sum_{i=1}^{e}r_{i}(U_{n-1})r_{i}(G)=0\Leftrightarrow p^{e}\mid(n-1).
\]
We easily see that the left equality is equivalent to 
\[
r_{i}(U_{n})=r_{i}(U_{n-1})\quad(i=1,\dots,e)
\]
and also to 
\begin{equation}
\left\lfloor \frac{n}{p^{i-1}}\right\rfloor -\left\lfloor \frac{n-1}{p^{i-1}}\right\rfloor =\left\lfloor \frac{n}{p^{i}}\right\rfloor -\left\lfloor \frac{n-1}{p^{i}}\right\rfloor \quad(i=1,\dots,e).\label{eq:cond}
\end{equation}
Now we observe that 
\[
\left\lfloor \frac{n}{p^{i}}\right\rfloor -\left\lfloor \frac{n-1}{p^{i}}\right\rfloor =\begin{cases}
0 & (p^{i}\mid n)\\
1 & (\text{otherwise}).
\end{cases}
\]
Thus, condition (\ref{eq:cond}) is equivalent to ``for every $1\le i\le e$,
if $p^{i-1}\mid n$, then $p^{i}\mid n$.'' Inductively this shows
$p^{e}\mid n$. Conversely, if $p^{e}\mid n$, then (\ref{eq:cond})
clearly holds. We have completed the proof. 
\end{proof}
\begin{cor}
For an integer $n>0$ with $p^{e}\nmid(n-1)$, the constructible subset
$\bc^{-1}(n)\subset|\Delta_{G}|$ has dimension
\[
\sum_{i\ge1}\left(\left\lfloor \frac{n-1}{p^{i-1}}\right\rfloor -\left\lfloor \frac{n-1}{p^{i}}\right\rfloor \right)r_{i}(G).
\]
Moreover, there exists only one geometric irreducible component of
maximal dimension, that is, $\bc^{-1}(n)$ is a sectoroid.
\end{cor}

\begin{proof}
Corollary \ref{cor:num-tors} shows that the number of $\FF_{q}$-points
in $\bc^{-1}(n)$ is 
\[
\left(q^{\sum_{i\ge1}\left(\left\lfloor \frac{n-1}{p^{i-1}}\right\rfloor -\left\lfloor \frac{n-1}{p^{i}}\right\rfloor \right)r_{i}(G)}-q^{\sum_{i\ge1}\left(\left\lfloor \frac{n-2}{p^{i-1}}\right\rfloor -\left\lfloor \frac{n-2}{p^{i}}\right\rfloor \right)r_{i}(G)}\right).
\]
This is a monic polynomial in $q$ of degree $\sum_{i\ge1}\left(\left\lfloor \frac{n-1}{p^{i-1}}\right\rfloor -\left\lfloor \frac{n-1}{p^{i}}\right\rfloor \right)r_{i}(G)$.
The corollary follows from the Lang--Weil estimate. 
\end{proof}
\begin{lem}
\label{lem:<=00003D-1}For $n>0$ with $p^{e}\nmid(n-1)$, we have
\[
\dim\bc^{-1}(n)-\frac{1+(p-1)\sum_{i=1}^{e}p^{e-i}r_{i}(G)}{p^{e}}n\le-1.
\]
Moreover, the equality holds exactly when $p^{h}\mid n$ and $n\le p^{e}$.
\end{lem}

\begin{proof}
We denote the left hand side by $D(n)$. We have
\begin{align*}
 & D(n)\\
 & =-\frac{n}{p^{e}}+\sum_{i=1}^{e}r_{i}(G)\left(\left\lfloor \frac{n-1}{p^{i-1}}\right\rfloor -\left\lfloor \frac{n-1}{p^{i}}\right\rfloor -\frac{(p-1)n}{p^{i}}\right)\\
 & =-\frac{n}{p^{e}}+\sum_{i=1}^{e}r_{i}(G)\left(\left(\frac{n-1}{p^{i}}-\left\lfloor \frac{n-1}{p^{i}}\right\rfloor +\frac{1}{p^{i}}\right)-\left(\frac{n-1}{p^{i-1}}-\left\lfloor \frac{n-1}{p^{i-1}}\right\rfloor +\frac{1}{p^{i-1}}\right)\right)\\
 & =-\frac{n}{p^{e}}+\sum_{i=1}^{e}r_{i}(G)\left(\left(\left\{ \frac{n-1}{p^{i}}\right\} +\frac{1}{p^{i}}\right)-\left(\left\{ \frac{n-1}{p^{i-1}}\right\} +\frac{1}{p^{i-1}}\right)\right).
\end{align*}
In particular, we have 
\[
D(n+mp^{e})=D(n)-m.
\]
Thus, the function $D(n)$ on the set $\{n\in\ZZ_{>0}\mid p^{e}\nmid(n-1)\}$
can attain the maximum only in the domain $\{2,\dots,p^{e}\}$. If
we denote the fractional part of a rational number $r$ by $\{r\}$,
then we have 
\[
\left\{ \frac{n-1}{p^{i}}\right\} +\frac{1}{p^{i}}=\begin{cases}
1 & (p^{i}\mid n)\\
\left\{ \frac{n}{p^{i}}\right\}  & (\text{otherwise}).
\end{cases}
\]
This shows $D(p^{e})=-1$. Now let $n\le p^{e}$ and let $v\le e$
be the $p$-adic valuation $v_{p}(n)$ of $n$. We have
\begin{align*}
 & D(n)\\
 & =-\frac{n}{p^{e}}+\sum_{i=1}^{v}r_{i}(G)\left(1-1\right)+r_{v+1}(G)\left(\left\{ \frac{n}{p^{v+1}}\right\} -1\right)\\
 & \qquad+\sum_{i=v+2}^{e}r_{i}(G)\left(\left\{ \frac{n}{p^{i}}\right\} -\left\{ \frac{n}{p^{i-1}}\right\} \right)\\
 & =-\frac{n}{p^{e}}-r_{v+1}(G)+r_{e}(G)\left\{ \frac{n}{p^{e}}\right\} +\sum_{i=v+1}^{e}\left(r_{i}(G)-r_{i+1}(G)\right)\left\{ \frac{n}{p^{i}}\right\} \\
 & \le-\frac{n}{p^{e}}-r_{v+1}(G)+r_{e}(G)\frac{n}{p^{e}}\\
 & \le-r_{v+1}(G)+(r_{e}(G)-1)\frac{n}{p^{e}}\\
 & \le-r_{v+1}(G)+(r_{e}(G)-1)\\
 & \le-1.
\end{align*}
Here equalities hold in all the last four inequalities if and only
if
\[
r_{v+1}(G)=r_{v+2}(G)=\cdots=r_{e}(G)=1.
\]
The last condition is equivalent to that $h\le v$. This proves the
lemma. 
\end{proof}
From (\ref{eq:a}) and (\ref{eq:b}) and from Lemma \ref{lem:<=00003D-1},
\begin{align*}
a(\cL,\bc) & =\frac{1+(p-1)\sum_{i=1}^{e}p^{e-i}r_{i}(G)}{p^{e}}
\end{align*}
and $b(\cL,\bc)$ is equal to the cardinality of the set
\begin{align*}
 & \{m\in\ZZ_{>0}\mid p^{e}\nmid(m-1),p^{h}\mid m\text{ and }m\le p^{e}\}\\
 & =\begin{cases}
\{2,3,\dots,p^{e}\} & (h=0)\\
\{p^{h}i\mid i=1,\dots,p^{e-h}\} & (h>0).
\end{cases}
\end{align*}
Thus,
\begin{align*}
b(\cL,\bc) & =\begin{cases}
p^{e}-1 & (h=0)\\
p^{e-h} & (h>0).
\end{cases}
\end{align*}
These values of the $a$- and $b$-invariants complies with Lagemann's
result \cite[Theorem 1.2]{lagemann2015distribution} on the number
of $G$-extensions with bounded conductor. 
\begin{defn}
We define a raising function $\bc'$ on $|\Delta_{G}|$ by 
\[
\bc'(x):=\begin{cases}
0 & (x=o)\\
\bc(x)-1 & (x\ne o).
\end{cases}
\]
The associated height is called the \emph{Artin--Schreier conductor
}in \cite{potthast2025onthe,gundlach2024counting}.
\end{defn}

\begin{lem}
When $n$ runs over positive integers with $p^{e}\nmid(n-1)$, the
fractions
\[
\frac{1+\sum_{i\ge1}(\lfloor(n-1)/p^{i-1}\rfloor-\lfloor(n-1)/p^{i}\rfloor)r_{i}(G)}{n-1}
\]
have the maximum and it is attained only when $n=2$. 
\end{lem}

\begin{proof}
When $n=2$, the fraction is equal to $1+r_{1}(G)$. For every $n\ge3$,
we have
\begin{align*}
 & (n-1)(1+r_{1}(G))-\left(1+\sum_{i\ge1}\left(\left\lfloor \frac{n-1}{p^{i-1}}\right\rfloor -\left\lfloor \frac{n-1}{p^{i}}\right\rfloor \right)r_{i}(G)\right)\\
 & =n-2+\sum_{i=1}^{e-1}\left\lfloor \frac{n-1}{p^{i}}\right\rfloor (r_{i}(G)-r_{i+1}(G))+\left\lfloor \frac{n-1}{p^{e}}\right\rfloor r_{e}(G)\\
 & >0.
\end{align*}
This shows that the fraction doesn't take the maximum for $n\ge0$. 
\end{proof}
This lemma shows 
\[
a(\cL,\bc')=1+r_{1}(G)\text{ and }b(\cL,\bc')=1.
\]
This complies with Gundlach's result \cite[Corollary 4.5(b)]{gundlach2024counting}. 

\subsection{Counting with discriminants; the case of elementary abelian $p$-groups}

We now assume that $G$ is an elementary abelian $p$-group of rank
$r\ge2$. We identify $G$ with $(\FF_{p})^{r}$, adopting the additive
notation. For simplicity, we fix an algebraically closed field $L/k$
and consider only $L$-points of $\Delta_{G}$. We then have
\[
\Delta_{G}\langle L\rangle=(L\tpars/\wp(L\tpars))^{r}=\left(\bigoplus_{j>0;p\nmid j}L\cdot t^{-j}\right)^{r}.
\]
When $r=1$, for an element $h\in\Delta_{G}\langle L\rangle$ identified
with an element $\bigoplus_{j>0;p\nmid j}L\cdot t^{-j}$, its conductor
exponent $\bc(h)$ is given by 
\[
\bc(h)=\begin{cases}
-\ord(h)+1 & (h\ne0)\\
0 & (h=0).
\end{cases}
\]

Let $G^{*}:=\Hom(G,\FF_{p})$, which we again identify with $(\FF_{p})^{r}$.
For each $\chi=(\chi_{1},\dots,\chi_{r})\in G^{*}$, we have the induced
map
\[
\Delta_{G}\langle L\rangle\to\Delta_{\FF_{p}}\langle L\rangle,\quad h=(h_{1},\dots,h_{r})\mapsto\chi(h)=\sum_{i}\chi_{i}h_{i}.
\]

\begin{prop}
Let $h=(h_{1},\dots,h_{r})\in\Delta_{G}\langle L\rangle$. Suppose
that $\langle h_{1},\dots,h_{s}\rangle_{\FF_{p}}$ has dimension $u$,
equivalently that the corresponding $G$-torsor over $\D_{L}^{*}$
consists of $p^{r-u}$ copies of a $\FF_{p}^{u}$-torsor. Then, there
exists a sequence of $\FF_{p}$-linear subspaces
\[
V_{\bullet}^{h}\colon G^{*}=V_{1}^{h}\supsetneq V_{2}^{h}\supsetneq\cdots\supsetneq V_{s+1}^{h}\cong\FF_{p}^{r-u}
\]
and integers $j_{1}>j_{2}>\cdots>j_{s}$ such that 
\[
\bc(\chi(h))=\begin{cases}
j_{i}+1 & (\chi\in V_{i}^{h}\setminus V_{i+1}^{h}\text{ for }i\le s)\\
0 & (\chi\in V_{s+1}^{h}).
\end{cases}
\]
\end{prop}

\begin{proof}
For each $j\in\ZZ_{\ge0}$, we define the $\FF_{p}$-linear subspace
\[
W_{j}:=\{\chi\in G^{*}\mid\bc_{\chi}(h)\le j\}.
\]
The sequence $(W_{j})_{j}$ is increasing and $W_{j}=G^{*}$ for $j\gg0$.
We put $V_{1}^{h}$ to be $G^{*}$, $V_{2}^{h}$ to be the second
largest subspace among $\{W_{j}\}$ and so on. This construction ends
at $V_{s+1}^{h}=W_{0}$, which has dimension $r-u$. The value of
$\bc_{\chi}(h)$ is constant on $V_{i}^{h}\setminus V_{i+1}^{h}$
for each $i\le s$, which we set to be $j_{i}+1$. 
\end{proof}
Let $\bd\colon|\Delta_{G}|\to\ZZ$ be the discriminant exponent. From
the conductor-discriminant formula \cite[p. 104]{serre1979localfields},
we obtain:
\begin{cor}
We keep the notation of the last proposition. Let $j_{1}>j_{2}>\cdots>j_{s}$
be the indices such that $V_{j_{i}-1}\ne V_{j_{i}}$. Let $r_{i}:=\dim_{\FF_{p}}V_{j_{i}}$.
Then, 
\begin{align*}
\bd(h) & =\sum_{i=1}^{s}(j_{i}+1)(p^{r_{i}}-p^{r_{i+1}}).
\end{align*}
\end{cor}

\begin{proof}
The case $u=r$ is a direct consequence of the conductor-discriminant
formula \cite[p. 104]{serre1979localfields}. If $u<r$, the relevant
$G$-torsor is the disjoint union of $p^{r-u}$ copies of a connected
$H$-torsor for a subgroup $\FF_{p}^{u}\cong H\subset G$. The filtration
of $H^{*}=\Hom(H,\FF_{p})=G^{*}$ induced by this $H$-torsor is identical
to 
\[
V_{1}^{h}/V_{s+1}^{h}\supset V_{2}^{h}/V_{s+1}^{h}\supset\cdots\supset V_{s+1}^{h}/V_{s+1}^{h}=0.
\]
Thus, the discriminant exponent of the $H$-torsor is 
\[
\sum_{i=1}^{s}(j_{i}+1)(p^{r_{i}-(r-u)}-p^{r_{i+1}-(r-u)})=p^{u-r}\sum_{i=1}^{s}(j_{i}+1)(p^{r_{i}}-p^{r_{i+1}}).
\]
The discriminant exponent of the original $G$-torsor is $p^{r-u}$
times the discriminant exponent of the $H$-torsor, which shows the
corollary. 
\end{proof}

Let us now fix a flag
\[
G^{*}=\FF_{p}^{r}=V_{1}\supsetneq V_{2}\supsetneq\cdots\supsetneq V_{s}\supsetneq V_{s+1}
\]
with $\dim V_{i}=r_{i}$. (For the number of such flags, see \cite[Section 1.5]{morrison2006integer}.)
By a suitable choice of bases of $G$ and $G^{*}$ which are dual
to each other, we may assume that 
\[
V_{j_{i}}=\{\chi=(\chi_{1},\dots,\chi_{r})\in\FF_{p}^{r}\mid\chi_{l}=0\text{ for \ensuremath{l>r_{i}}}\}.
\]
Let us also fix $j_{1}>j_{2}>\cdots>j_{s}$. An element $h=(h_{1},\dots,h_{r})\in\Delta_{G}\langle L\rangle$
satisfies 
\[
\bc(\chi(h))=\begin{cases}
j_{i}+1 & (\chi\in V_{i}\setminus V_{i+1}\text{ for }i\le s)\\
0 & (\chi\in V_{s+1}).
\end{cases}
\]
if and only if for each $i$, $\ord h_{\lambda}=-j_{i}$ for $r_{i+1}<\lambda\le r_{i}$
and their leading coefficients of $h_{\lambda}$, $r_{i+1}<\lambda\le r_{i}$
are linearly independent over $\FF_{p}$. Such $h$'s are parametrized
by 
\begin{equation}
\prod_{i=1}^{s}\left(L^{r_{i}-r_{i+1}}\setminus W\right)\times\left(L^{j_{i}-\lfloor j_{i}/p\rfloor-1}\right)^{r_{i}-r_{i+1}},\label{eq:sect}
\end{equation}
where 
\[
W=\bigcup_{(a_{m})\in(\FF_{p})^{r_{i}-r_{i+1}}}\left\{ (z_{m})_{1\le m\le r_{i-}r_{i-1}}\in L^{r_{i}-r_{i+1}}\bigmid\sum_{m}a_{m}z_{m}=0\right\} .
\]
In particular, the dimension of the parameter space is 
\begin{align*}
 & \sum_{i=1}^{s}(r_{i}-r_{i+1})+\sum_{i=1}^{s}(r_{i}-r_{i+1})\left(j_{i}-\left\lfloor \frac{j_{i}}{p}\right\rfloor -1\right)\\
 & =\sum_{i=1}^{s}(r_{i}-r_{i+1})j_{i}-\sum_{i=1}^{s}(r_{i}-r_{i+1})\left\lfloor \frac{j_{i}}{p}\right\rfloor .
\end{align*}

\begin{prop}
\label{prop:disc-max}Consider the rational number
\[
\alpha(s,j_{\bullet},r_{\bullet}):=\frac{1+\sum_{i=1}^{s}(r_{i}-r_{i+1})j_{i}-\sum_{i=1}^{s}(r_{i}-r_{i+1})\left\lfloor \frac{j_{i}}{p}\right\rfloor }{\sum_{i=1}^{s}(j_{i}+1)(p^{r_{i}}-p^{r_{i+1}})}
\]
where $s$, $j_{\bullet}$ and $r_{\bullet}$ varies under the conditions
as above.
\begin{enumerate}
\item Under the extra condition $u\ge2$, $\alpha(s,j_{\bullet},r_{\bullet})$
has the maximum value $\frac{1+r(p-1)}{p(p^{r}-1)}$, exactly when
$u=r$, $s=1$ and $j_{1}=p-1$.
\item Under the extra condition $u=1$ (and hence $s=1$), $\alpha(s,j_{\bullet},r_{\bullet})$
has the maximum value $\frac{1}{p^{r}-p^{r-1}}$, exactly when $j_{1}=1,2,\dots,p-1$.
\end{enumerate}
\end{prop}

\begin{proof}
(1) We put $a:=\frac{1+r(p-1)}{p(p^{r}-1)}$ and define
\[
f(s,j_{\bullet},r_{\bullet}):=\sum_{i=1}^{s}(r_{i}-r_{i+1})j_{i}-\sum_{i=1}^{s}(r_{i}-r_{i+1})\left\lfloor \frac{j_{i}}{p}\right\rfloor -a\left(\sum_{i=1}^{s}(j_{i}+1)(p^{r_{i}}-p^{r_{i+1}})\right).
\]
Note that 
\[
\frac{1+\sum_{i=1}^{s}(r_{i}-r_{i+1})j_{i}-\sum_{i=1}^{s}(r_{i}-r_{i+1})\left\lfloor \frac{j_{i}}{p}\right\rfloor }{\sum_{i=1}^{s}(j_{i}+1)(p^{r_{i}}-p^{r_{i+1}})}\le a\Leftrightarrow f(s,j_{\bullet},r_{\bullet})\le-1.
\]
Moreover, the left inequality is strict if and only if so is the right
one. We need to show that the inequality $f(s,j_{\bullet},r_{\bullet})\le-1$
always holds and that the equality in the last inequality holds exactly
when $s=1$ and $j_{1}=p-1$. 

Consider $r$ variables $k_{1},\dots,k_{r}$ running over positive
integers satisfying $k_{1}\ge k_{2}\ge\cdots\ge k_{r}$ and $p\nmid k_{\lambda}$.
To each such tuple $(k_{1},\dots,k_{r})\in\ZZ^{r}$, we associate
integers $s>0$, $j_{1}>\cdots>j_{s}$, $r_{1}>\cdots>r_{s}$ so that
for $r_{i+1}<\lambda\le r_{i}$, $k_{\lambda}=j_{i}$. We define $f(k_{1},\dots,k_{r}):=f(s,j_{\bullet},r_{\bullet})$.
We also associate $(l_{1},\dots,l_{r})\in(\ZZ_{\ge0})^{r}$ given
by 
\begin{align*}
k_{\lambda} & =\sum_{u=\lambda}^{r}l_{u}.
\end{align*}
Thus, $\{k_{1},\dots,k_{r}\}=\{j_{1},\dots,j_{s}\}$. We then consider
the following change of variables
\begin{align*}
k_{i} & =\sum_{u=i}^{r}l_{u}.
\end{align*}
Then
\begin{align*}
f & =\sum_{i=1}^{r}k_{i}-\sum_{i=1}^{r}\left\lfloor \frac{k_{i}}{p}\right\rfloor -a\left(\sum_{i=1}^{r}(k_{i}+1)(p^{r-i+1}-p^{r-i})\right)\\
 & =\frac{p-1}{p}\sum_{i=1}^{r}k_{i}+\sum_{i=1}^{r}\left\{ \frac{k_{i}}{p}\right\} -a\left(\sum_{i=1}^{r}(k_{i}+1)(p^{r-i+1}-p^{r-i})\right)\\
 & =\frac{p-1}{p}\sum_{i=1}^{r}il_{i}-a\left(p^{r}\sum_{i=1}^{r}l_{i}-\sum_{i=1}^{r}l_{i}p^{r-i}\right)+\sum_{i=1}^{r}\left\{ \frac{k_{i}}{p}\right\} \\
 & =\sum_{i=1}^{r}\left(\frac{i(p-1)}{p}-a(p^{r}-p^{r-i})\right)l_{i}+\sum_{i=1}^{r}\left\{ \frac{k_{i}}{p}\right\} .
\end{align*}

We claim that 
\[
\frac{i(p-1)}{p}-a(p^{r}-p^{r-i})<0\quad(1\le i\le r).
\]
The claim is equivalent to
\[
a>\frac{p-1}{p^{r+1}}\max\left\{ \frac{i}{1-p^{-i}}\mid i\in\{1,\dots,r\}\right\} .
\]
It is easy to check that the function $\frac{x}{1-p^{-x}}$ is monotonically
increasing for $x\ge0$. The claim follows from
\[
a=\frac{1+r(p-1)}{p^{r+1}-p}>\frac{r(p-1)}{p^{r+1}-p}=\frac{p-1}{p^{r+1}}\cdot\frac{r}{1-p^{-r}}.
\]

From the claim, when we fix residues $\overline{k_{i}}\in\{1,\dots,p-1\}$
of $k_{i}$'s modulo $p$, then $f$ is a strictly decreasing function
and the maximum is attained when $k_{i}$'s satisfies the following
condition:
\begin{itemize}
\item $\lfloor k_{r}/p\rfloor=0$ and
\item if $\overline{k_{i}}<\overline{k_{i-1}}$, then$\lfloor k_{i}/p\rfloor=\lfloor k_{i-1}/p\rfloor+1$. 
\end{itemize}
We can write
\[
\left(\left\lfloor \frac{k_{1}}{p}\right\rfloor ,\dots,\left\lfloor \frac{k_{r}}{p}\right\rfloor \right)=(\overset{b_{1}}{\overbrace{\lambda-1,\dots,\lambda-1}},\dots,\overset{b_{\lambda-1}}{\overbrace{1,\dots,1}},\overset{b_{\lambda}}{\overbrace{0,\dots,0}}).
\]
Let $B_{\nu}:=\sum_{\mu=1}^{\nu}b_{\mu}.$ Since $k_{i}-\lfloor k_{i}/p\rfloor=\overline{k_{i}}+(p-1)\lfloor k_{i}/p\rfloor$,
\begin{align*}
f & =\sum_{i=1}^{r}\overline{k_{i}}+(p-1)\sum_{i=1}^{r}\left\lfloor \frac{k_{i}}{p}\right\rfloor -a\left(\sum_{i=1}^{r}\left(\overline{k_{i}}+p\left\lfloor \frac{k_{i}}{p}\right\rfloor +1\right)(p^{r-i+1}-p^{r-i})\right)\\
 & =\sum_{i=1}^{r}\overline{k_{i}}+(p-1)\sum_{\mu=1}^{\lambda-1}(\lambda-\mu)b_{\mu}\\
 & \qquad-a\left((p-1)\sum_{i=1}^{r}(\overline{k_{i}}+1)p^{r-i}+(p-1)\sum_{i=1}^{r}\left\lfloor \frac{k_{i}}{p}\right\rfloor p^{r+1-i}\right)\\
 & =\sum_{i=1}^{r}\overline{k_{i}}+(p-1)\sum_{\mu=1}^{\lambda-1}(\lambda-\mu)b_{\mu}\\
 & \qquad-a\left((p-1)\sum_{i=1}^{r}(\overline{k_{i}}+1)p^{r-i}+(p-1)\sum_{\mu=1}^{\lambda-1}(\lambda-\mu)\sum_{\xi=1+B_{\nu-1}}^{B_{\nu}}p^{r+1-\xi}\right).
\end{align*}
From \cite[Lemma  5.30]{potthast2025onthe}, if $\lambda\ge2$, then
$f<-1$. If $\lambda=1$, then 
\[
p-1\ge k_{1}\ge\cdots\ge k_{r}\ge1.
\]
From \cite[Lemma  5.29]{potthast2025onthe}, we have $f\le-1$, where
the equality holds exactly when $k_{1}=\cdots=k_{r}=p-1$, equivalently
when $u=r$, $s=1$ and $j_{1}=p-1$. 

(2) In this case, $s=1$ and $r_{1}=r$ and $r_{2}=r-1$. Thus, $\alpha(s,j_{\bullet},r_{\bullet})$
reduces to 
\[
\frac{1+j_{1}-\left\lfloor \frac{j_{1}}{p}\right\rfloor }{(j_{1}+1)(p^{r}-p^{r-1})}=\frac{1}{p^{r}-p^{r-1}}-\frac{\left\lfloor \frac{j_{1}}{p}\right\rfloor }{(j_{1}+1)(p^{r}-p^{r-1})},
\]
which shows the assertion.
\end{proof}
\begin{lem}[{{]}\cite[Lemma  5.24]{potthast2025onthe}}]
\label{lem:r-r'}For $1\le r'<r$, we have 
\[
\frac{1+r'(p-1)}{p^{r-r'+1}(p^{r'}-1)}\le\frac{1+r(p-1)}{p(p^{r}-1)}.
\]
Moreover, the equality holds only when $p=r=2$ (and hence $r'=1$). 
\end{lem}

We can get the following corollary, which complies with \cite[Theorem 1.1]{potthast2025onthe}:
\begin{cor}
For any line bundle $\cL$ on $\cX=\B G$ and any sectoroid subdivision
of $|\cJ_{\infty}\cX|=|\Delta_{G}|$, we have
\[
a(\cL,\bd)=\frac{1+r(p-1)}{p(p^{r}-1)}\text{ and }b(\cL,\bd)=\begin{cases}
p-1 & (r=1)\\
4 & (p=r=2)\\
1 & (\text{otherwise}).
\end{cases}
\]
\end{cor}

\begin{proof}
From Remark \ref{rem:well-def}, we may choose any sectoroid subdivision.
We choose the subdivision into loci of $h$'s determined by given
$s$, $j_{\bullet}$, $r_{\bullet}$ and a flag $V_{\bullet}$ of
$G^{*}$; each locus is isomorphic to the irreducible variety given
in (\ref{eq:sect}) after the base change to an algebraically closed
field $L$. Then, the $a$-invariant is the maximum of the rational
numbers $\alpha(s,j_{\bullet},r_{\bullet})$ in Proposition \ref{prop:disc-max}.
Note that $\frac{1}{p^{r}-p^{r-1}}=\frac{1+r'(p-1)}{p^{r-r'+1}(p^{r'}-1)}$
with $r'=1$. The assertion for the $a$-invariant now follows from
Proposition \ref{prop:disc-max} and Lemma \ref{lem:r-r'}. 

To show the assertion for the $b$-invariant, we need to find how
many sectoroids give the maximum of $\alpha(s,j_{\bullet},r_{\bullet})$.
If $r=1$, then $s$, $r_{\bullet}$ and $V_{\bullet}$ are uniquely
determined as $s=1$, $r_{1}=1>r_{2}=0$ and $V_{\bullet}=\{G^{*}=V_{1}\supset V_{2}=0\}$.
Thus, there are exactly $p-1$ sectoroids corresponding to $j_{1}=1,\dots,p-1$
and hence $b(\cL,\bd)=p-1$ in this case. If $r\ge2$ and $(p,r)\ne2$,
then Proposition \ref{prop:disc-max} and Lemma \ref{lem:r-r'} show
that the maximum of $\alpha(s,j_{\bullet},r_{\bullet})$ is attained
exactly when $u=r$, $s=1$ and $j_{1}=p-1$. Then, the only possible
flag $V_{\bullet}$ is $\{G^{*}=V_{1}\supset V_{2}=0\}$. Thus, $b(\cL,\bd)=1$.
Finally, if $(p,r)=(2,2)$, then the maximum of $\alpha(s,j_{\bullet},r_{\bullet})$
is attained when $u=2$, $s=1$ and $j_{1}=1$ and when $u=1$, $s=1$
and $j_{1}=1$. In the former case, the only possible flag $V_{\bullet}$
is again $\{G^{*}=V_{1}\supset V_{2}=0\}$. In the latter case, the
possible flags $V_{\bullet}$ are ones of the form $\{(\FF_{2})^{2}\cong G^{*}=V_{1}\supset V_{2}\cong\FF_{2}\}$;
there are three of them. In total, four sectoroids give the maximum
of $\alpha(s,j_{\bullet},r_{\bullet})$ and hence $b(\cL,\bd)=4$.
\end{proof}

\section{Fano varieties with canonical singularities\label{sec:Fano-varieties}}

In this section, we consider the case where $-K_{X}$ is ample and
$X$ has only canonical singularities. For simplicity, we assume that
there exists a resolution of singularities, $f\colon Y\to X$. Then,
there are only finitely many crepant divisors over $X$ and denote
its number by $\gamma(X)$. A version of the Manin conjecture for
Fano varieties with canonical singularities \cite[Conjecture 2.3]{yasuda2014densities}
(applied to the positive characteristic and isotrivial case without
change) says that for some thin subset $T\subset X(F)$ and a positive
constant $C>0$, we have
\[
\#\{x\in X(F)\setminus T\mid H_{-K_{X}}(x)\le B\}\sim CB(\log X)^{\rho(X)+\gamma(X)-1}\quad(B\to\infty).
\]
The anti-canonical height function $H_{-K_{X}}$ is regarded as the
height function associated to the trivially raised canonical line
bundle $(\omega_{X}^{-1},0)$. To see that our main ``conjecture''
is compatible with this conjecture, we will show $a(\omega_{X}^{-1},0)=1$
and $b(\omega_{X}^{-1},0)$ under certain extra assumptions in this
section.

We now choose an arbitrary sectoroid subdivision $\{A_{i}\}_{i\in I}$
of $|\cJ_{\infty}\cX|$. Then, for every $i\in I^{*}$, $\gw(A_{i})\le-1$.
Let $I_{0}:=\{i\in I^{*}\mid\gw(A_{i})=-1\}$. 
\begin{prop}
\label{prop:a=00003D1-1}We have $a(\omega_{X}^{-1},0)=1$.
\end{prop}

\begin{proof}
Let $\widetilde{C}\to X_{\overline{k}}$ be a covering family of curves.
Then, for $a'<1$
\[
(\llbracket\widetilde{C}\rrbracket,a'\llbracket\omega_{X}^{-1},0\rrbracket+\fK_{X})=([\widetilde{C}],(1-a')[\omega_{X}])<0
\]
and hence $a'\llbracket\omega_{X}^{-1},0\rrbracket+\fK_{X}\notin\fPEff(X)$.
On the other hand, 
\begin{equation}
\llbracket\omega_{\cX}^{-1},0\rrbracket+\fK_{\cX}\in\sum_{i\in I^{*}\setminus I_{0}}\RR_{>0}[A_{i}]^{*}.\label{eq:expr}
\end{equation}
This class intersects non-negatively with the augmented class of any
covering family of curves. Hence $\llbracket\omega_{\cX}^{-1},0\rrbracket+\fK_{\cX}\in\fPEff(\cX)$.
Thus, we get $a(\omega_{X}^{-1},0)=1$. 
\end{proof}
\begin{prop}
We have $\gamma(X)=\#I_{0}$. Moreover, there exists a natural one-to-one
correspondence between the set of crepant divisors over $X$ and $I_{0}$. 
\end{prop}

\begin{proof}
Let $f\colon Y\to X$ be a resolution and let $E_{1},\dots,E_{\gamma(X)}$
be those prime divisors on $Y$ which are crepant over $X$. Let $(\J_{\infty}X)_{X_{\sing}}$
be the preimage of the singular locus $X_{\sing}\subset X$ by the
projection $\J_{\infty}X\to X$. From the change of variables formula,
there exist open dense subsets $E_{j}^{\circ}\subset E_{j}$, $1\le j\le\gamma(X)$
such that 
\[
\mu_{X}^{\Gor}(|(\J_{\infty}X)_{X_{\sing}}|)=\sum_{j=1}^{\gamma(X)}\{E_{j}^{\circ}\}\LL^{-\dim X}+(\text{terms of dimension \ensuremath{<-1}}).
\]
On the other hand, we can get another expression of $\mu_{X}^{\Gor}(|(\J_{\infty}X)_{X_{\sing}}|)$
by using the sectoroid subdivision as follows:
\begin{align*}
\mu_{X}^{\Gor}(|(\J_{\infty}X)_{X_{\sing}}|) & =\mu_{\cX}^{\Gor}(A_{0}\cap|(\J_{\infty}\cX)_{X_{\sing}}|)+\sum_{i\in I^{*}}\mu_{\cX}^{\Gor}(A_{i}\cap|(\J_{\infty}\cX)_{X_{\sing}}|)\\
 & =\sum_{i\in I^{*}}\mu_{\cX}^{\Gor}(A_{i})+(\text{terms of dimension \ensuremath{<-1}})\\
 & =\sum_{i\in I_{0}}\mu_{\cX}^{\Gor}(A_{i})+(\text{terms of dimension \ensuremath{<-1}}).
\end{align*}
Comparing the two obtained expressions of $\mu_{X}^{\Gor}(|(\J_{\infty}X)_{X_{\sing}}|)$
shows the first assertion of the proposition. To show the second assertion,
we see that there exists a unique $i\in I_{0}$ such that $\gw(f_{\infty}((\J_{\infty}X)_{E_{j}^{\circ}})\cap A_{i})=-1$
and for every $i'\in I_{0}\setminus\{i\}$, $\gw(f_{\infty}((\J_{\infty}X)_{E_{j}^{\circ}})\cap A_{i'})<-1$.
Associating this $i$ to $E_{j}$ defines the desired one-to-one correspondence. 
\end{proof}
Let $\theta\colon\{1,\dots,\gamma(X)\}\to I_{0}$ be the bijection
obtained in the last proposition. Its proof shows that for each $j\in\{1,\dots,\gamma(X)\}$,
$A_{\theta(j)}$ and $f_{\infty}((\J_{\infty}X)_{E_{j}^{\circ}})$
have the intersection with large volume. We now show $b(\omega_{X}^{-1},0)=\rho(X)+\gamma(X)$
under extra conditions:
\begin{prop}
\label{prop:b-rho-gamma-1}Assume that for each $j\in\{1,\dots,\gamma(X)\}$,
there exists an open dense subset $E_{j}^{\circ}$ such that $A_{\theta(j)}$
contains $f_{\infty}((\J_{\infty}Y)_{E_{j}^{\circ}})$. Assume also
that $I$ is finite. Then, we have $b(\omega_{X}^{-1},0)=\rho(X)+\gamma(X)$. 
\end{prop}

\begin{proof}
From the definition we have $\fMov_{1}(X)\subset\Mov_{1}(\cX)+\sum_{i\in I^{*}}\RR_{\ge0}[A_{i}]$.
From (\ref{eq:expr}) in the proof of Proposition \ref{prop:a=00003D1-1},
we see that 
\[
(\llbracket\omega_{X}^{-1},0\rrbracket+\fK_{X})^{\perp}\cap\fMov_{1}(X)\subset\Mov_{1}(X)+\sum_{i\in I_{0}}\RR_{\ge0}[A_{i}].
\]
In particular, we have $b(\omega_{\cX}^{-1},0)\le\rho(\cX)+\gamma(\cX)$. 

We now show the opposite inequality. We fix $j\in\{1,\dots,\gamma(X)\}$.
As the intersection of general members of a very ample complete linear
system, we can find a curve $C\subset Y_{\overline{k}}$ that meets
$E_{i}^{\circ}$ but does not meet $Y\setminus\bigcup_{j}E_{j}$.
Let $C'$ be its normalization and let $v_{0}\in C'$ be a point mapping
to $E_{i}^{\circ}$ and let $v_{1},\dots,v_{m}\in C'\setminus\{v_{0}\}$
be the points mapping to $Y\setminus f^{-1}(U)$, where $U$ is an
open dense subset of $X$ such that $|\J_{\infty}U|\subset A_{0}$.
For a positive integer $n$ with $p\nmid n$, we can take a cover
$C_{n}\to C'$ from a smooth irreducible curve $C_{n}$ which is étale
over $e_{i}$ and totally ramified over $v_{1},\dots,v_{m}$ \cite[Lemma 8.12]{darda2024thebatyrevtextendashmanin}.
Then,
\[
\frac{1}{n}\llbracket C_{n}\rrbracket=[C]+[A_{i}]+\frac{1}{n}\sum_{j=1}^{m}[A_{i_{j,n}}].
\]
Here $i_{j,n}$ are some elements of $I$. As $n$ increases, this
element approaches $[C]+[A_{i}]$ (we use the finiteness of $I$ here),
hence $[C]+[A_{i}]\in\fMov_{1}(X)$. We can do this construction for
every $i\in\{1,\dots,\gamma(X)\}$ by using the same very ample complete
linear system and get $[C]+[A_{i}]\in\fMov_{1}(X)$ for every $i$
with the same class $[C]$. 

Let $[C_{1}],\dots,[C_{\rho(X)}]$ be the moving curve classes that
are linearly independent. A similar covering-and-limit argument shows
that $[C_{1}],\dots,[C_{\rho(X)}]$ are also in $\fMov_{1}(X)$. Now,
$\rho(X)+\gamma(X)$ elements 
\[
[C_{1}],\dots,[C_{\rho(X)}],[C]+[A_{1}],\dots,[C]+[A_{\gamma(X)}]\in(\llbracket\omega_{X}^{-1},0\rrbracket+\fK_{X})^{\perp}\cap\fMov_{1}(X)
\]
are linearly independent. This shows $b(\omega_{\cX}^{-1},0)\ge\rho(\cX)+\gamma(\cX)$. 
\end{proof}
Suppose that the coarse moduli space $\pi\colon\cX\to X$ is an isomorphism
in codimension one. Let $\cU\subset\cX$ and $U\subset X$ be the
largest open substack/subscheme such that $\pi$ restricts to an isomorphism
$\cU\xrightarrow{\sim}U$. In this situation, the canonical line bundle
$\omega_{\cX}$ of $\cX$ corresponds to the canonical line bundle
$\omega_{X}$ of $X$ via the identification $\N^{1}(\cX)_{\RR}=\N^{1}(X)_{\RR}$.
Hence, the height function $H_{\omega_{\cX}^{-1}}$ coincides with
the anti-canonical height $H_{\omega_{X}^{-1}}$ via the identification
$\cU(F)=U(F)$. Let us choose the sectoroid subdivision of $|\cJ_{\infty}\cX|$
that corresponds to the sectoroid subdivision $\{A_{i}\}_{i\in I}$
of $|\J_{\infty}X|$ that was chosen above. The same argument as above
shows that $a(\omega_{\cX}^{-1},0)=1$ and $b(\omega_{\cX}^{-1},0)=\rho(X)+\gamma(X)$. 
\begin{example}
Suppose that $k$ has characteristic 3. Let $\cX$ and $X$ be the
quotient stack and the quotient variety associated to the cyclic permutation
action of the cyclic group $C_{3}$ of order 3 on $\PP^{1}\times_{k}\PP^{1}\times_{k}\PP^{1}$,
respectively. The singular locus $X_{\sing}$ of $X$ is the image
of the small diagonal $\PP_{k}^{1}\hookrightarrow\PP^{1}\times_{k}\PP^{1}\times_{k}\PP^{1}$,
which is again isomorphic to $\PP_{k}^{1}$. The blowup $Y=\Bl_{X_{\sing}}(X)\to X$
of $X$ along $X_{\sing}$ is known to be a crepant resolution having
two exceptional prime (necessarily crepant) divisors. Note also that
$\rho(X)=1$. Thus, the Manin conjecture for $X$ \cite[Conjecture 1.1]{yasuda2015maninsconjecture}
says that for a suitable thin $T$, we have 
\[
\#\{x\in X(F)\setminus T\mid H_{-K_{X}}(x)\le B\}\sim CB(\log B)^{2}.
\]
Under the assumption in Proposition \ref{prop:b-rho-gamma-1}, we
have $a(\omega_{X}^{-1},0)=a(\omega_{\cX}^{-1},0)=1$ and $b(\omega_{X}^{-1},0)=b(\omega_{\cX}^{-1},0)=3$.
Thus, the above prediction is compatible with Conjecture \ref{conj:main}. 
\end{example}

\begin{example}
We have a similar example in characteristic 2. Let $S$ be a del Pezzo
surface over $k$ with Picard number $\rho=\rho(S)$. Let $X=(S\times S)/C_{2}$
and $\cX=[(S\times S)/C_{2}]$ be the quotient variety and the quotient
stack associated to the transposition action. Then, $X$ is singular
along the image of the diagonal. The blowup $Y\to X$ along it is
a crepant resolution and has only one exceptional prime divisor. Therefore,
the Manin conjecture for $X$ predicts an asymptotic formula of the
form $\#\{\dots\}\sim CB(\log B)^{\rho(X)}$. Again, Conjecture \ref{conj:main}
predicts a formula of the same form. 
\end{example}

\section{Compatibility with products\label{sec:Compatibility-with-products}}

Let $\cX$ and $\cY$ be quasi-nice DM stacks over $k$. Then the
product $\cZ:=\cX\times_{k}\cY$ is also a quasi-nice DM stack. Let
$\{A_{i}\}_{i\in I}$ and $\{B_{j}\}_{j\in J}$ be sectoroid subdivisions
of $|\cJ_{\infty}\cX|$ and $|\cJ_{\infty}\cY|$, respectively. Using
them, we construct a sectoroid subdivision of $|\cJ_{\infty}\cZ|$.
For each $(i,j)\in I\times J$, identifying $|\cJ_{\infty}\cZ|$ with
$|\cJ_{\infty}\cX|\times|\cJ_{\infty}\cY|$, we put 
\[
C_{ij}:=A_{i}\times B_{j}\subset|\cJ_{\infty}\cZ|.
\]
We then put $C_{0}:=C_{00}$ and 
\[
C_{\infty}:=(A_{\infty}\times|\cJ_{\infty}\cY|)\cup(|\cJ_{\infty}\cX|\times B_{\infty}).
\]
If $i\ne\infty$ and $j\ne\infty$ and if either $i$ or $j$ is zero,
then $C_{ij}$ is a sectoroid. For $(i,j)\in I^{*}\times J^{*}$,
we decompose $C_{ij}$ into the disjoint union 
\[
C_{ij}=\bigsqcup_{n\in N_{ij}}C_{ijn}
\]
of countably many sectoroids $C_{ijn}$, $n\in N_{ij}$. Note that
such a decomposition is necessary, because the product of two irreducible
schemes/stacks over $k$ may not be irreducible. We choose the collection
\[
\{C_{0},C_{\infty}\}\cup\{C_{0j}\mid j\in J^{*}\}\cup\{C_{i0}\mid i\in I^{*}\}\cup\{C_{ijn}\mid(i,j)\in I^{*}\times J^{*},n\in N_{ij}\}
\]
as our sectoroid subdivision of $|\cJ_{\infty}\cZ|$. Thus, its index
set is identified with 
\[
K:=\{0,\infty\}\cup J^{*}\cup I^{*}\cup(I^{*}\times J^{*})
\]
and we have
\[
K^{*}=J^{*}\cup I^{*}\cup(I^{*}\times J^{*}).
\]

\begin{lem}
\label{lem:gw-sum}If $i\ne\infty$ and $j\ne\infty$, then $\gw(C_{ij})=\gw(A_{i})+\gw(B_{j})$. 
\end{lem}

\begin{proof}
First, we consider the case where $\cX$ and $\cY$ have trivial generic
stabilizers and hence so does $\cZ$. It is easy to see that for measurable
subsets $A\subset|\J_{\infty}X|$ and $B\subset|\J_{\infty}Y|$, we
have 
\[
\mu_{X\times_{k}Y}(A\times B)=\mu_{X}(A)\cdot\mu_{Y}(B).
\]
For simplicity, we first consider the case where $X$ and $Y$ are
1-Gorenstein (that is, $\omega_{X}$ and $\omega_{Y}$ are invertible)
and $\cX^{\rig}\to X$ and $\cY^{\rig}\to Y$ are étale in codimension
one. Then, the $\QQ$-divisors induced on $X$ and $Y$ as in Section
\ref{sec:Gorenstein-weights} are both the zero divisors. Define an
ideal sheaf $\cI_{X}\subset\cO_{X}$ by 
\[
\left(\bigwedge^{\dim X}\Omega_{X/k}\right)/\tors=\cI_{X}\cdot\omega_{X}.
\]
Similarly for ideal sheaves $\cI_{Y}\subset\cO_{Y}$ and $\cI_{Z}\subset\cO_{Z}$.
Since $\Omega_{Z/k}=\pr_{1}^{*}\Omega_{X/k}\oplus\pr_{2}^{*}\Omega_{Y/k}$,
we have
\[
\bigwedge^{\dim Z}\Omega_{Z/k}=\bigoplus_{i+j=Z}\bigwedge^{i}\pr_{1}^{*}\Omega_{X/k}\otimes\bigwedge^{j}\pr_{2}^{*}\Omega_{Y/k}.
\]
The summand $\bigwedge^{i}\pr_{1}^{*}\Omega_{X/k}\otimes\bigwedge^{j}\pr_{2}^{*}\Omega_{Y/k}$
for $(i,j)\ne(\dim X,\dim Y)$ is a torsion sheaf. Thus, 
\[
\left(\bigwedge^{\dim Z}\Omega_{Z/k}\right)/\tors=\left(\left(\bigwedge^{\dim X}\pr_{1}^{*}\Omega_{X}\right)/\tors\otimes\left(\bigwedge^{\dim Y}\pr_{2}^{*}\Omega_{Y}\right)/\tors\right)/\tors.
\]
This shows that 
\[
\cI_{Z}=\pr_{1}^{-1}\cI_{X}\cdot\pr_{2}^{-1}\cI_{Y}.
\]
It follows that for Gorenstein measurable subsets $A\subset|\cJ_{\infty}\cX|$
and $B\subset|\cJ_{\infty}\cY|$, 
\begin{align*}
\mu_{Z}^{\Gor}(A\times B) & =\int_{A\times B}\LL^{\ord_{\cI_{X}}\circ(\pr_{1})_{\infty}+\ord_{\cI_{Y}}\circ(\pr_{2})_{\infty}}\,d\mu_{Z}\\
 & =\left(\int_{A}\LL^{\ord_{\cI_{X}}\circ(\pr_{1})_{\infty}}\,d\mu_{X}\right)\cdot\left(\int_{B}\LL^{\ord_{\cI_{Y}}\circ(\pr_{2})_{\infty}}\,d\mu_{Y}\right)\\
 & =\mu_{X}^{\Gor}(A)\cdot\mu_{Y}^{\Gor}(B).
\end{align*}
and $\gw(A\times B)=\gw(A)+\gw(B)$. It is straightforward to generalize
this argument and show the lemma by dropping the assumptions that
$\cX$ and $\cY$ have the trivial generic stabilizers, that $\cX^{\rig}\to X$
and $\cY^{\rig}\to Y$ are étale in codimension one, and that the
induced log canonical divisors on $X$ and $Y$ are not 1-Gorenstein
but only $\QQ$-Gorenstein. 
\end{proof}
Let $(\cL,c_{\cX})$ and $(\cM,c_{\cY})$ be raised line bundles on
$\cX$ and $\cY$ which are compatible with $\{A_{i}\}$ and $\{B_{j}\}$,
respectively. We define a raised line bundle $(\cN,c_{\cZ})$ on $\cZ$
by $\cN:=\pr_{1}^{*}\cL\otimes\pr_{2}^{*}\cM$ and $c_{\cZ}:=c_{\cX}\circ(\pr_{1})_{\infty}+c_{\cY}\circ(\pr_{2})_{\infty}$.
The raising function $c_{\cZ}$ is compatible with the sectoroid subdivision
of $|\cJ_{\infty}\cZ|$ constructed above. 

Let $\cC_{\cZ}\to\cZ_{\overline{k}}$ be a stacky curve. Let $\cC_{\cX}\to\cX_{\overline{k}}$
be the induced stacky curve obtained by the canonical factorization
of the composition $\cC_{\cZ}\to\cZ_{\overline{k}}\to\cX_{\overline{k}}$
(see \cite[Lemma 25]{yasuda2006motivic}). Similarly, we define a
stacky curve $\cC_{\cY}\to\cY_{\overline{k}}$. 
\begin{lem}
\label{lem:ineq-intersection}Suppose that the images of these stacky
curves are not contained in exceptional substacks for chosen sectoroid
subdivisions, respectively. 
\begin{enumerate}
\item We have
\[
((\cN,c_{\cZ}),\cC_{\cZ})=((\cL,c_{\cX}),\cC_{\cX})+((\cM,c_{\cY}),\cC_{\cY}).
\]
\item For a real number $a$, we have 
\[
(a\llbracket\cN,c_{\cZ}\rrbracket+\fK_{\cZ},\cC_{\cZ})\ge(a\llbracket\cL,c_{\cX}\rrbracket+\fK_{\cX},\cC_{\cX})+(a\llbracket\cM,c_{\cY}\rrbracket+\fK_{\cY},\cC_{\cY}).
\]
Moreover, the equality holds if and only if the augmented class $\llbracket\cC_{\cZ}\rrbracket$
has zero coefficients at the terms $[C_{ijn}]$ for $(i,j)\in I^{*}\times J^{*}$
and $n\in N_{i,j}$. 
\end{enumerate}
\end{lem}

\begin{proof}
(1) We identify the point sets $|\cC_{\cZ}|$, $|\cC_{\cY}|$ and
$|\cC_{\cX}|$ via the canonical bijections. Let $z\in|\cC_{1}|=|\cC_{2}|=|\cC_{3}|$
be a closed point. Let $A_{i_{z}}$ and $B_{j_{z}}$ be the sectoroids
in $|\cJ_{\infty}\cX|$ and $|\cJ_{\infty}\cY|$ derived from $z$,
respectively. Note that from the assumption, for every $z$, we have
$A_{i_{z}}\ne A_{\infty}$ and $B_{j_{z}}\ne B_{\infty}$. Then, the
sectoroid $C_{z}$ in $|\cJ_{\infty}\cZ|$ derived from $z$ is given
by
\[
C_{z}=\begin{cases}
C_{0} & (i_{z}=0\text{ and }j_{z}=0)\\
C_{0j_{z}} & (i_{z}=0\text{ and }j_{z}\ne0)\\
C_{i0} & (i_{z}\ne0\text{ and }j_{z}=0)\\
C_{i_{z}j_{z}n_{z}} & (i_{z}\ne0\text{ and }j_{z}\ne0).
\end{cases}
\]
Here $n_{z}$ is some element of $N_{i_{z}j_{z}}$. In each case,
we have 
\[
c_{\cZ}(C_{z})=c_{\cX}(A_{i_{z}})+c_{\cY}(B_{j_{z}}).
\]
Thus, we get
\begin{align*}
(\llbracket\cL_{3},c_{3}\rrbracket,\cC_{3}) & =(\cL_{3},\cC_{3})+\sum_{z}c_{\cZ}(C_{z})\\
 & =(\cL_{1},\cC_{1})+(\cL_{2},\cC_{2})+\sum_{z}(c_{\cX}(A_{i_{z}})+c_{\cY}(B_{j_{z}}))\\
 & =(\cL_{1},\cC_{1})+\sum_{z}c_{1}(A_{i_{z}})+(\cL_{2},\cC_{2})+\sum_{z}c_{2}(A_{i_{z}})\\
 & =(\llbracket\cL_{1},c_{1}\rrbracket,\cC_{1})+(\llbracket\cL_{2},c_{2}\rrbracket,\cC_{2}).
\end{align*}

(2) Since 
\[
\gw(C_{z})\le\gw(C_{i_{z}j_{z}})=\gw(A_{i_{z}})+\gw(B_{j_{z}}),
\]
we have
\begin{align*}
 & (a\llbracket\cN,c_{\cZ}\rrbracket+\fK_{\cZ},\cC_{\cZ})\\
 & =a([\cN]+[\omega_{\cZ}],\cC_{\cZ})+\sum_{z:C_{z}\ne C_{0}}(-\gw(C_{z})-1)\\
 & \ge a([\cN]+[\omega_{\cX}],\cC_{\cX})+a([\cN]+[\omega_{\cY}],\cC_{\cY})+\sum_{z:C_{z}\ne C_{0}}(-\gw(A_{i_{z}})-\gw(B_{j_{z}})-1).
\end{align*}
Moreover, in the last inequality, the equality holds if for every
$z$ with $C_{z}\ne C_{0}$, either $i_{z}$ or $j_{z}$ is zero.
We also have
\begin{multline*}
\sum_{z:C_{z}\ne C_{0}}(-\gw(A_{i_{z}})-\gw(B_{j_{z}})-1)=\\
\sum_{z:i_{z}\ne0}(-\gw(A_{i_{z}})-1)+\sum_{z:j_{z}\ne0}(-\gw(B_{j_{z}})-1)+\#\{z\mid i_{z}\ne0\text{ and }j_{z}\ne0\}.
\end{multline*}
Combining these shows the assertion.
\end{proof}
\begin{lem}
\label{lem:equiv-PEff}Let $a$ be a real number. The following are
equivalent:
\begin{enumerate}
\item $a\llbracket\cL,c_{\cX}\rrbracket+\fK_{\cX}\in\fPEff(\cX)$ and $a\llbracket\cM,c_{\cY}\rrbracket+\fK_{\cY}\in\fPEff(\cY)$.
\item $a\llbracket\cN,c_{\cZ}\rrbracket+\fK_{\cZ}\in\fPEff(\cZ)$.
\end{enumerate}
\end{lem}

\begin{proof}
(1) $\Rightarrow$ (2). Suppose that $\cC_{\cZ}$ is a general member
of a covering family of stacky curves and hence so are $\cC_{\cY}$
and $\cC_{\cX}$. Thus, $(a\llbracket\cL,c_{\cX}\rrbracket+\fK_{\cX},\cC_{\cX})\ge0$
and $(a\llbracket\cM,c_{\cY}\rrbracket+\fK_{\cY},\cC_{\cY})\ge0$.
From Lemma \ref{lem:ineq-intersection}, we have $(a\llbracket\cN,c_{\cZ}\rrbracket+\fK_{\cZ},\cC_{\cZ})\ge0$,
which shows (2).

(2) $\Rightarrow$ (1). We prove the contraposition. Suppose that
$a\llbracket\cL,c_{\cX}\rrbracket+\fK_{\cX}\notin\fPEff(\cX)$ and
hence that $(a\llbracket\cL,c_{\cX}\rrbracket+\fK_{\cX},\cC_{\cX})<0$
for a general member $\cC_{\cX}$ for some covering family of stacky
curves on $\cX_{\overline{k}}$. Let $\cC_{\cY}\to\cY_{\overline{k}}$
be a general member of a covering family of stacky curves. Let $s_{1},\dots,s_{m}\in|\cC_{\cX}|$
be the points $s$ such that the corresponding sectoroid $A_{i_{s}}$
is not $A_{0}$. Similarly for $t_{1},\dots,t_{n}\in|\cC_{\cY}|$.
Let 
\[
C'\subset\overline{\cC_{\cX}\times_{\overline{k}}\cC_{\cY}}=\overline{\cC_{\cX}}\times_{\overline{k}}\overline{\cC_{\cY}}
\]
be a general member of a very ample complete linear system of bi-degree
$(d_{\cX},d_{\cY})$ which does not meet any of the points $(s_{i},t_{j})$
and does meet transversally with all of $\{s_{i}\}\times\overline{\cC_{\cY}}$
and $\overline{\cC_{\cX}}\times\{t_{j}\}$. Let $\cC'\to\cZ_{\overline{k}}$
be the stacky curve induced by the rational map $C'\dasharrow\cZ_{\overline{k}}$,
which is again a general member of a covering family. If $u\in|\cC'|$
lies over $s_{i}$ (then does not lie over any of $t_{1},\dots,t_{n}$),
then it gives the locus $C_{i_{s_{i}}0}$. Moreover, for each $s_{i}$
(resp.~$t_{j}$), there are exactly $d_{\cX}$ (resp.~$d_{\cY}$)
points of $|\cC'|$ over it. Therefore, 
\[
(a\llbracket\cM,c_{\cZ}\rrbracket+\fK_{\cZ},\cC')=d_{\cX}(a\llbracket\cL,c_{\cX}\rrbracket+\fK_{\cX},\cC_{\cX})+d_{\cY}(a\llbracket\cM,c_{\cY}\rrbracket+\fK_{\cY},\cC_{\cY}).
\]
It follows that if we choose the bi-degree $(d_{\cX},d_{\cY})$ with
$d_{\cX}\gg d_{\cY}>0$, then the right hand side is negative. Thus,
$a\llbracket\cN,c_{\cZ}\rrbracket+\fK_{\cZ}\notin\fPEff(\cZ)$. Similarly
for the case $a\llbracket\cM,c_{\cY}\rrbracket+\fK_{\cY}\notin\fPEff(\cY)$.
\end{proof}
The formulas for the $a$- and $b$-invariants associated to a product
height is compatible with the asymptotic formula for a product height
obtained in Theorem \ref{thm:product-lemma} of Appendix \ref{sec:A-product-lemma}.
\begin{prop}
\label{prop:a-b-prod}We have
\[
a(\cN,c_{\cZ})=\max\{a(\cL,c_{\cX}),a(\cM,c_{\cY})\}
\]
and
\[
b(\cN,c_{\cZ})=\sum_{\substack{(\cF,e)\in\{(\cL,c_{\cX}),(\cM,c_{\cY})\}\\
a(\cF,e)=a(\cN,c_{\cZ})
}
}b(\cF,e).
\]
\end{prop}

\begin{proof}
The first equality is a direct consequence of Lemma \ref{lem:equiv-PEff}.
Let $a=a(\cN,c_{\cZ})$. To show the second equality, we first observe
that there exists a natural map $\fN_{1}(\cZ)\to\fN_{1}(\cX)\times\fN_{2}(\cY)$
sending $[C_{ijn}]$ to $([A_{i}],[B_{j}])$. We claim that if $a=a(\cL,c_{\cX})>a(\cM,c_{\cY})$,
then this map gives the identification
\[
\fMov_{1}(\cZ)\cap(a\llbracket\cN,c_{\cZ}\rrbracket+\fK_{\cZ})^{\perp}=\fMov_{1}(\cX)\cap(a\llbracket\cL,c_{\cX}\rrbracket+\fK_{\cX})^{\perp}
\]
and if $a=a(\cL,c_{\cX})=a(\cM,c_{\cY})$, then the identification
\begin{gather*}
\fMov_{1}(\cZ)\cap(a\llbracket\cN,c_{\cZ}\rrbracket+\fK_{\cZ})^{\perp}=\\
\left(\fMov_{1}(\cX)\cap(a\llbracket\cL,c_{\cX}\rrbracket+\fK_{\cX})^{\perp}\right)\times\left(\fMov_{1}(\cY)\cap(a\llbracket\cM,c_{\cY}\rrbracket+\fK_{\cY})^{\perp}\right).
\end{gather*}

First, we consider the case $a(\cL,c_{\cX})>a(\cM,c_{\cY})$. In this
case, $a\llbracket\cM,c_{\cY}\rrbracket+\fK_{\cY}$ is in the interior
of $\fPEff(\cY)$. If $\theta\in\fMov_{1}(\cZ)$ maps to a nonzero
element of $\fMov_{1}(\cY)$, then $(a\llbracket\cN,c_{\cZ}\rrbracket+\fK_{\cZ})\cdot\theta>0$.
Thus, for an element $\theta\in\fMov_{1}(\cZ)$ to have zero intersection
with $a\llbracket\cN,c_{\cZ}\rrbracket+\fK_{\cZ}$, it needs to map
to zero in $\fMov_{1}(\cY)$. This shows the first equality of the
claim. 

Next, we consider the case $a(\cL,c_{\cX})=a(\cM,c_{\cY})$. From
(2) of Lemma \ref{lem:ineq-intersection}, for an element $\theta\in\fMov_{1}(\cZ)$
to have zero intersection with $a\llbracket\cN,c_{\cZ}\rrbracket+\fK_{\cZ}$,
it needs to have zero coefficients at the terms $[C_{ijn}]$ for $(i,j)\in I^{*}\times J^{*}$
and $n\in N_{i,j}$. Thus, the cone $\fMov_{1}(\cZ)\cap(a\llbracket\cN,c_{\cZ}\rrbracket+\fK_{\cZ})^{\perp}$
is contained in the linear subspace
\[
N_{1}(X)\oplus N_{1}(Y)\oplus\bigoplus_{i\in I^{*}}\RR[C_{i0}]\oplus\bigoplus_{j\in J^{*}}\RR[C_{0j}]\subset\fN_{1}(\cZ),
\]
which is isomorphic to $\fN_{1}(\cX)\times\fN_{1}(\cY)$. Now, the
second equality of the claim follows again from (2) of Lemma \ref{lem:ineq-intersection}. 
\end{proof}

\section{Elliptic curves in characteristic 3\label{sec:Elliptic-curves}}

\subsection{Setup\label{subsec:Setup}}

Suppose that $k$ has characteristic 3. Let $\cV:=\cM_{1,1}$ be the
moduli stack of elliptic curves over $k$ and let $\cX:=\overline{\cM_{1,1}}$
be its Deligne--Mumford compactification. Let $V:=M_{1,1}$ and $X:=\overline{M_{1,1,}}$
be their coarse moduli spaces, respectively. From \cite{igusa1959fibresystems}
(see also \cite[(8.2.1)]{katz1985arithmetic}), $V=\Spec k[j]=\AA_{k}^{1}$
and $X=\PP_{k}^{1}=V\cup\{\infty\}$. Here the variable $j$ corresponds
to the $j$-invariant of elliptic curves. The automorphism group of
a geometric point of $\cM_{1,1}$ is of order 2 except when $j=0$
\cite[Chapter IV, Corollary 4.7]{hartshorne1977algebraic}. The automorphism
group of a geometric point with $j=0$ is the dicyclic group of order
12. Note that the automorphism group of a nodal elliptic curve (given
with a marked point) is also of order two. Indeed, the group is identified
with the group of automorphisms of $\PP^{1}$ that fixes a point $x_{1}$
(the marked point) and a pair $\{x_{2},x_{3}\}$ of two points (the
two points over the node), where $x_{1},x_{2},x_{3}$ are three distinct
points. Thus, the group is identical to $S_{2}=C_{2}$. Thus, the
rigidification $\cX^{\rig}$ has only one stacky point at $j=0$,
which has an automorphism group isomorphic to $S_{3}$. 

The rigidification $\cX^{\rig}$ contains $U:=\Spec k[j^{\pm}]$ as
an open substack. Let $\cU\to U$ be the restriction of $\cX\to\cX^{\rig}$,
which is a $\mu_{2}$-gerbe. The equation
\[
z^{2}=w^{3}+w^{2}-1/j
\]
defines a family of elliptic curves over $U$ which give a section
of this gerbe (from \cite[p. 42]{silverman2009thearithmetic}, the
Elliptic curve defined by the above equation has $j$-invariant $j$).
Thus, the gerbe $\cU\to U$ is neutral and we get:
\begin{prop}
The open substack $\cU\subset\cX$ is isomorphic to the trivial $\mu_{2}$-gerbe
$\B\mu_{2}\times U$. In particular, 
\[
\cU\langle F\rangle=(\B\mu_{2})\langle F\rangle\times U(F).
\]
\end{prop}

To study the ramification of $\cX\to X$ at $j=0$, consider the family
of elliptic curves given by the Legendre form
\[
z^{2}=w(w-1)(w-\lambda).
\]
The equation defines a family of (possibly nodal) elliptic curves
over $A:=\AA^{1}$ (with parameter $\lambda$) and defines a morphism
$A\to\cX.$
\begin{prop}
\label{prop:LegEt}The morphism $A\to\cX$ as well as the composition
$A\to\cX\to\cX^{\rig}$ is étale.
\end{prop}

\begin{proof}
Since $\cX\to\cX^{\rig}$ is étale, we only need to show that $A\to\cX^{\rig}$
is étale. Over a point $x\in X$ with $j\ne0$, both morphisms $A\to X$
and $\cX^{\rig}\to X$ are étale (see \cite[Proposition 1.7, page 49]{silverman2009thearithmetic}).
It follows that $A\to\cX^{\rig}$ is étale over such a point $x$.
Let $\cY=[V/S_{3}]$ be the completion of $\cX^{\rig}$ at the point
$j=0$. Similarly, we define $\widehat{A}$ and $\widehat{X}$. Since
$V\to\cY$ is étale, we have a unique dashed arrow in:
\[
\xymatrix{\Spec k\ar[r]\ar[d] & V\ar[d]\\
\widehat{A}\ar[r]\ar@{-->}[ur] & \cY
}
\]
Since $\widehat{A}\to\cY$ and $V\to\cY$ are of degree 6, the morphism
$\widehat{A}\to V$ is a degree-1 finite morphism of normal integral
schemes. Therefore, it is an isomorphism. It follows that $\widehat{A}\to\cY$
is étale. 
\end{proof}
The composition 
\[
\AA^{1}=\Spec k[\lambda]\to\overline{\cM_{1,1}}\to\overline{M_{1,1}}=\Spec k[j]
\]
is given by the formula
\[
j=\frac{2^{8}(\lambda^{2}-\lambda+1)^{3}}{\lambda^{2}(\lambda-1)^{2}}=\frac{(\lambda+1)^{6}}{\lambda^{2}(\lambda-1)^{2}}.
\]
The only point over $j=0$ is the point $\lambda=-1$. 

From \cite[Theorem 6]{igusa1959fibresystems}, the higher ramification
groups of $\cX^{\rig}\to X$ at $j=0$ are given by
\[
G_{0}=S_{3},\,G_{1}=C_{3},\,G_{i}=1\,(i\ge2).
\]
From \cite[Proposition 4,  p.64]{serre1979localfields}, the different
exponent of the corresponding $S_{3}$-extension of $\overline{k}\tpars$
is 
\[
(6-1)+(3-1)=7.
\]
Thus, 
\[
[\omega_{\cX^{\rig}}]=[K_{X}]+\frac{7}{6}[\pt]=-\frac{5}{6}[\pt].
\]
Here $[\pt]$ is the class of a degree-1 point of $X$ and we use
identifications $\N^{1}(\cX)_{\RR}=\N^{1}(X)_{\RR}=\RR$ through which
we have $[K_{X}]=-2$ and $[\pt]=1$. Therefore $\cX^{\rig}$ is a
Fano stack in the sense that the canonical line bundle $\omega_{\cX^{\rig}}$
corresponds to an ample $\QQ$-line bundle on the coarse moduli space
$X$ \cite[Definition 5.1]{darda2024thebatyrevtextendashmanin}. Let
$x\in X$ be the $j=0$ point. Then, the $\QQ$-divisor $D$ such
that $\cX^{\rig}\to(X,D)$ is crepant is given by $D=(7/6)x$. In
particular, $(X,D)$ is not log canonical. 

\subsection{\label{subsec:working on X^=00007Brig=00007D}Working on $\protect\cX^{\protect\rig}$}

Let $A_{i}'\subset\J_{\infty}X$ be the locus of arcs with order $i$
at $j=0$ and let 
\[
A_{i}:=(\pi_{\infty}^{\rig})^{-1}(A_{i}')\setminus|\J_{\infty}\cX|\subset|\cJ_{\infty}\cX^{\rig}|.
\]
We have
\begin{align*}
\mu_{\cX^{\rig}}^{\Gor}(A_{i}) & =\mu_{(X,D)}^{\Gor}(A_{i}')-\mu_{\cX^{\rig}}^{\Gor}((\pi_{\infty}^{\rig})^{-1}(A_{i}')\cap|\J_{\infty}\cX|)\\
 & =(\LL-1)\LL^{-i-1+(7/6)i}-(\LL-1)\LL^{-6i-1+(7/6)i}\\
 & =(\LL-1)(\LL^{(1/6)i-1}-\LL^{-(29/6)i-1}).
\end{align*}
In particular, $\gw(A_{i})=i/6$ and $A_{i}$'s are sectoroids. Thus,
$\{A_{i}\}_{i\in I}$ is a sectoroid subdivision, which is our choice
of the sectoroid subdivision. Then, 
\begin{align*}
\fN_{1}(\cX^{\rig}) & =\RR[X]\oplus\bigoplus_{i\in\ZZ_{>0}}\RR[A_{i}],\\
\fN^{1}(\cX^{\rig}) & =\RR[\pt]\oplus\bigoplus_{i\in\ZZ_{>0}}\RR[A_{i}]^{*}.
\end{align*}
The augmented  canonical class is computed as
\begin{align*}
\fK_{\cX^{\rig}} & =-\frac{5}{6}[\pt]+\sum_{i\in\ZZ_{>0}}\left(-\frac{i}{6}-1\right)[A_{i}]^{*}.
\end{align*}

\begin{lem}
\label{lem:gens-Mov}The cone $\fMov_{1}(\cX^{\rig})$ is generated
by
\[
[X]\text{ and }i[X]+[A_{i}]\quad(i\in\ZZ_{>0}).
\]
\end{lem}

\begin{proof}
For each $i\in\ZZ_{>0}$, we choose a degree-$i$ cover $C_{i}\to X_{\overline{k}}$
from a smooth projective curve $C_{i}$ which is totally ramified
at $x\in X_{\overline{k}}$. When $6$ divides $i$, we also assume
that the induced twisted arc of $\cX^{\rig}$ at $x$ is in fact untwisted,
equivalently, that the formal completion $\widehat{C_{i}}\to\widehat{X_{\overline{k}}}$
at $x$ does not factor through the affine line $\AA_{\overline{k}}^{1}=\Spec k[\lambda]$
with the Legendre parameter $\lambda$. Then, the augmented class
of $C_{i}$ is $i[X]+[A_{i}]$. Thus, we see that $[X]$ and $i[X]+[A_{i}]$,
$i\in\ZZ_{>0}$ are indeed in $\fMov_{1}(\cX^{\rig})$. 

To show that every element of $\fMov_{1}(\cX^{\rig})$ is generated
by these elements, let us take a degree-$m$ covering $Y\to X_{\overline{k}}$
and let $y_{1},\dots,y_{n}\in Y$ be the points over the $j=0$ point.
For each $s\in\{1,\dots,n\}$, let $\widehat{Y}_{s}\to X$ be the
induced arc and let $i_{s}$ be its degree. We have $m=\sum_{s=1}^{n}i_{s}$.
Then, the induced twisted arc at $y_{s}$ either is untwisted or belongs
to the sectoroid $A_{i_{s}}$. Let $[A_{i_{s}}]'$ be zero in the
former case and $[A_{i_{s}}]$ in the latter case. We have
\[
[Y]=m[X]+\sum_{s=1}^{n}[A_{i_{s}}]'=\sum_{s=1}^{n}(i_{s}[X]+[A_{i_{s}}]').
\]
In particular, $[Y]$ is generated by $[X]$ and $i[X]+[A_{i}]$,
$i\in\ZZ_{>0}$. 

It remains to show that the cone generated by these elements is closed
for the weak topology. We easily see that these elements form a basis
of $\fN_{1}(\cX)$. If $\alpha_{j}\in\fN^{1}(\cX^{\rig})$, $j\in J$
is its dual basis, then the cone is given by 
\[
\bigcap_{j\in J}\{\beta\in\fN_{1}(\alpha)\mid(\alpha,\beta)\ge0\},
\]
which is clearly closed. 
\end{proof}
\begin{prop}
The cone $\fPEff(\cX^{\rig})$ is described as
\[
\{a[\pt]+\sum_{i>0}b_{i}[A_{i}]^{*}\mid a\ge0\text{ and }ai+b_{i}\ge0\,(\forall i>0)\}.
\]
\end{prop}

\begin{proof}
This is a direct consequence of Lemma \ref{lem:gens-Mov}.
\end{proof}
\begin{prop}
Let $L$ be a line bundle on $X$ of degree $d>0$ and let $\cL$
be its pullback to $\cX^{\rig}$. Then, $a(\cL,0)=2/d$ and $b(\cL,0)=1$.
\end{prop}

\begin{proof}
We first compute the $a$-invariant. For a real number $a'$, we have
\begin{align*}
 & a'\llbracket\cL,0\rrbracket+\fK_{\cX^{\rig}}\in\fPEff(\cX)\\
 & \Leftrightarrow a'd-\frac{5}{6}\ge0\text{ and }\left(a'd-\frac{5}{6}\right)i+\left(-\frac{i}{6}-1\right)\ge0\,(\forall i>0)\\
 & \Leftrightarrow a'\ge\frac{5}{6d}\text{ and }a'\ge\frac{i+1}{id}\,(\forall i>0)\\
 & \Leftrightarrow a'\ge\frac{2}{d}.
\end{align*}
Hence $a(\cL,0)=2/d$. Next, we compute the $b$-invariant. We have
\[
a(\cL,0)\llbracket L,0\rrbracket+\fK_{\cX^{\rig}}=\frac{7}{6}[\pt]+\sum_{i>0}\left(-\frac{i}{6}-1\right)[A_{i}]^{*}.
\]
Thus, 
\begin{align*}
 & \fMov_{1}(\cX^{\rig})\cap(a(\cL,0)\llbracket L,0\rrbracket+\fK_{\cX^{\rig}})^{\perp}\\
 & =\left\{ a[X]+\sum_{i>0}b_{i}(i[X]+[A_{i}])\bigmid a,b_{i}\ge0\right\} \cap(a(\cL,0)[L,0]+\fK_{\cX^{\rig}})^{\perp}\\
 & =\left\{ a[X]+\sum_{i>0}b_{i}(i[X]+[A_{i}])\bigmid a,b_{i}\ge0,\,\frac{7}{6}a+\frac{7}{6}b_{i}i+\sum_{i>0}b_{i}\left(-\frac{i}{6}-1\right)=0\right\} \\
 & =\left\{ a[X]+\sum_{i>0}b_{i}(i[X]+[A_{i}])\bigmid a,b_{i}\ge0,\,\frac{7}{6}a+\sum_{i>0}b_{i}(i-1)=0\right\} \\
 & =\left\{ a[X]+\sum_{i>0}b_{i}(i[X]+[A_{i}])\bigmid a=0,\,b_{1}\ge0,\,b_{i}=0\,(i>1)\right\} \\
 & =\RR_{\ge0}([X]+[A_{1}]).
\end{align*}
This shows $b(\cL,0)=1$. 
\end{proof}
This proposition shows that our conjecture applied to $\cX^{\rig}$
and the pullback of $L$ is compatible with the conjecture for $X$
and $L$. 

\subsection{Working on $\protect\cX$}

We saw that the rigidification $\cX\to\cX^{\rig}$ is a neutral $\mu_{2}$-gerbe
over the open substack $U\subset\cX^{\rig}$. From Proposition \ref{prop:neutral},
for every point $\Spec k'\tpars\to\cX^{\rig}$, the fiber product
$\Spec k'\tpars\times_{\cX^{\rig}}\cX$ is isomorphic to $\B\mu_{2}$.
Here we used the fact that the group $\mu_{2}$ has no nontrivial
automorphism and hence has no nontrivial twisted form. Thus, every
fiber of the map $\rho_{\infty}\colon|\cJ_{\infty}\cX|\to|\cJ_{\infty}\cX^{\rig}|$
is identified with the space $|\Delta_{\mu_{2},k'}|$ of $\mu_{2}$-torsors
over $\D_{k'}^{*}$ for an algebraically closed field $k'$. The identification
is unique, since $|\Delta_{\mu_{2}}|$ has only two points say $o$
and $\xi$ corresponding to the trivial torsor and the nontrivial
torsor. We can identify $|\cJ_{\infty}\cX|$ with $|\cJ_{\infty}\cX^{\rig}|\times\{o,\xi\}$.
Thus, the stack $\cX$ almost looks like the product $\cX^{\rig}\times_{k}\B\mu_{2}$.
As we will see below, we can apply arguments in the case of products
from Section \ref{sec:Compatibility-with-products} to the current
``pseudo-product'' case without much change. 

For a sectoroid subdivision $\{A_{i}\}_{i\in I}$ of $|\cJ_{\infty}\cX^{\rig}|$
as Section \ref{subsec:working on X^=00007Brig=00007D} and for $j\in\{0,1\}$,
we define 
\[
C_{ij}:=\begin{cases}
A_{i}\times\{o\} & (j=0)\\
A_{i}\times\{\xi\} & (j=1).
\end{cases}
\]
We see that $C_{ij}$ are sectoroids unless $i\ne\infty$ with Gorenstein
weight $\gw(C_{ij})=\gw(A_{i})$ and $\{C_{ij}\}_{i\in I,j\in\{0,1\}}$
is a sectoroid subdivision of $|\cJ_{\infty}\cX|$. Let $c\colon\{o,\xi\}=|\Delta_{\mu_{2}}|\to\RR$
be a raising function with $c(\xi)=c_{0}$. We define a raising function
$\widetilde{c}\colon|\cJ_{\infty}\cX|\to\RR$ by 
\[
\widetilde{c}(\gamma):=\begin{cases}
0 & (\gamma\in|\cJ_{\infty}\cX^{\rig}|\times\{o\})\\
c_{0} & (\gamma\in|\cJ_{\infty}\cX^{\rig}|\times\{\xi\}).
\end{cases}
\]
This is an analog of $0\boxplus c$ considered in the product case.
Let $\cL'$ be the pullback of $\cL$ to $\cX$. We consider the height
function $H_{\cL',\widetilde{c}}$ associated to the raised line bundle
$(\cL',\widetilde{c})$. If we restrict it to $\cU\langle F\rangle$,
it is nothing but the product 
\[
H_{\cL}\times H_{c}\colon U\langle F\rangle\times(\B\mu_{2})\langle F\rangle\to\RR,\,(a,b)\mapsto H_{\cL}(a)\cdot H_{c}(b).
\]

Let $\cC_{\cX}\to\cX_{\overline{k}}$ be a dominating stacky curve
and let $\cC_{\cX^{\rig}}\to\cX_{\overline{k}}^{\rig}$ be the stacky
curve obtained by the canonical factorization of the composition $\cC_{\cX}\to\cX_{\overline{k}}\to\cX_{\overline{k}}^{\rig}$.
Denoting the function field of $\cC_{\cX}$ by $F$, let $\cC_{\B\mu_{2}}\to\B\mu_{2}$
be the stacky curve obtained by extending the point 
\[
\Spec F\to\cU=U\times_{k}\B\mu_{2}\xrightarrow{\pr_{2}}\B\mu_{2}.
\]

\begin{lem}
\label{lem:ineq-intersection-2}
\begin{enumerate}
\item We have
\[
(\cL',\widetilde{c})\cdot\cC_{\cX}=(\cL,0)\cdot\cC_{\cX^{\rig}}+(\cO,c)\cdot\cC_{\B\mu_{2}}.
\]
\item For a real number $a$, we have 
\[
(a\llbracket\cL',\widetilde{c}\rrbracket+\fK_{\cX})\cdot\cC_{\cX}\ge(a\llbracket\cL,0\rrbracket+\fK_{\cX})\cdot\cC_{\cX}+(a\llbracket\cO,c\rrbracket+\fK_{\B\mu_{2}})\cdot\cC_{\B\mu_{2}}.
\]
Moreover, the equality holds if and only if the augmented class $\llbracket\cC_{\cX}\rrbracket$
has zero coefficients at the terms $[C_{ij}]$ for $(i,j)\in I^{*}\times\{1\}$. 
\end{enumerate}
\end{lem}

\begin{proof}
The proof of Lemma \ref{lem:ineq-intersection} applies to this situation. 
\end{proof}
\begin{prop}
We have
\[
a(\cL',\widetilde{c})=\min\{a(\cL,0),a(\cO,c)\}
\]
and
\[
b(\cL',\widetilde{c})=\sum_{\substack{(\cF,e)\in\{(\cL,0),(\cO,c)\}\\
a(\cF,e)=a(\cL',\widetilde{c})
}
}b(\cF,e).
\]
\end{prop}

\begin{proof}
We can prove the version of Lemma \ref{lem:equiv-PEff} in our situation.
Using it, we can prove this proposition as an analog of Proposition
\ref{prop:a-b-prod} by the same argument. 
\end{proof}
We can compute 

\[
a(\cO,c)=1/c_{0}\text{ and }b(\cO,c)=1
\]
and 
\[
a(\cL,0)=2/d\text{ and }b(\cL,0)=1.
\]
Thus,
\[
a(\cL',\widetilde{c})=\begin{cases}
2/d & (2/d\ge1/c_{0})\\
1/e & (2/d<1/c_{0}).
\end{cases}
\]
and 
\[
b(\cL',\widetilde{c})=\begin{cases}
2 & (2/d=1/c_{0})\\
1 & (2/d\ne1/c_{0}).
\end{cases}
\]

\begin{thm}
The asymptotic formula in Conjecture \ref{conj:main} holds true for
the thin subset $T:=\cU\langle F\rangle\setminus\cX\langle F\rangle$
and for the height function $H_{\cL',\widetilde{c}}$, that is, we
have
\[
\#\{x\in\cU\langle F\rangle\mid H_{\cL',\widetilde{c}}(x)\le B\}\asymp B^{a(\cL',\widetilde{c})}(\log B)^{b(\cL',\widetilde{c})-1}.
\]
\end{thm}

\begin{proof}
The theorem can be proved by applying Theorem \ref{thm:product-lemma}
to a result for counting rational points of $\B\mu_{2}$ \cite[Theorems 1.3.2]{darda2023torsors}
and one for a projective line \cite{dipippo1990spacesof,wan1992heights}
(see also \cite{phillips2024pointsof}).
\end{proof}

\subsection{Conductors and minimal discriminants}

For an elliptic curve $E$ over $k'\tpars$, the \emph{minimal discriminant
exponent} $\fd(E)\in\ZZ_{\ge0}$ is defined to be the minimal $t$-order
of the discriminants for all the possible Weierstrass equations defining
$E$ with coefficients in $k'\tbrats$. The elliptic curve $E$ has
good reduction if and only if $\fd(E)=0$. We can also define the
\emph{conductor exponent} $\ff(E)$ as in \cite[p. 380]{silverman1994advanced}.
It is easy to see that $\fd$ and $\ff$ defines functions on $|\cJ_{\infty}\cX|$.
We have invariants for elliptic curves over a global field corresponding
to $\fd$ and $\ff$ respectively. 
\begin{defn}
Let $E$ be an elliptic curve over a global field $F$. For each place
$v\in M_{F}$, let $E_{v}:=E\otimes_{F}F_{v}$, which is an elliptic
curve over $F_{v}\cong k'\tpars$ for some finite extension $k'/k$.
The \emph{minimal discriminant $\fD(E)$ }and the \emph{conductor
$\fF(E)$ }of $E$ are defined respectively by
\[
\fD(E):=\prod_{v\in M_{F}}q_{v}^{\fd(E_{v})}\text{ and }\fF(E):=\prod_{v\in M_{F}}q_{v}^{\ff(E_{v})}.
\]
\end{defn}

The goal of this section is to show that $\fD$ and $\fF$ are the
height functions associated to raised line bundles. 
\begin{lem}
\label{lem:min-ext-ell}Let $K:=k'\tpars$ with $k'$ an algebraically
closed field and let $E$ be an elliptic curve over $K$ that has
potential good reduction. Let $\overline{K}$ be an algebraic closure
of $K$. Suppose that $E_{\overline{K}}=E\otimes_{K}\overline{K}$
can be given by an equation
\begin{equation}
z^{2}=(w-e_{1})(w-e_{2})(w-e_{3}),\label{eq:e123}
\end{equation}
where $e_{1}$, $e_{2}$ and $e_{3}$ are distinct elements of $\overline{K}$.
Then, 
\[
L:=\bigcap_{i\ne j}K(e_{1},e_{2},e_{3},\sqrt{e_{i}-e_{j}})\subset\overline{K}
\]
is the minimal extension $F/K$ such that $E_{F}=E\otimes_{K}F$ has
good reduction. 
\end{lem}

\begin{proof}
We first show that $E_{L}$ has good reduction. Relabeling $e_{1}$,
$e_{2}$ and $e_{3}$ if necessary, we may assume that $e_{2}-e_{1}$
has square roots in $L$. Over $L$, we can perform the coordinate
change
\begin{align*}
w & =(e_{2}-e_{1})w'+e_{1},\\
z & =(e_{2}-e_{1})^{3/2}z'
\end{align*}
so that we transform the equation to one of the Legendre form
\[
z^{2}=w(w-1)\left(w-\frac{e_{3}-e_{1}}{e_{2}-e_{1}}\right).
\]
From the assumption that $E$ has potential good reduction and from
an argument in \cite[Proof of (VII. 5.5), p.199]{silverman2009thearithmetic},
$\frac{e_{3}-e_{1}}{e_{2}-e_{1}}$ is an integral element and not
equal to $0$ or $1$ in the residue field of $L$. It follows that
$E_{L}$ has good reduction. 

To show the minimality, let $F/K$ be the minimal extension such that
$E_{F}$ has a good reduction. Note that from Lemma \ref{lem:twisted-arc-minimality},
such an extension is unique. Then, since the 2-torsion points are
$F$-rational, we have $K(e_{1},e_{2},e_{3})\subset F$. Since $E$
has  good reduction over $F$, we have a morphism $\Spec\cO_{F}\to\cX$,
where $\cO_{F}$ denotes the integral closure of $k'\tbrats$ in $F$.
Since $\cO_{F}$ has the algebraically closed residue field $k'$,
it lifts to $A=\AA_{k}^{1}$, where $A\to\cX$ is the étale morphism
induced from the Legendre form (Proposition \ref{prop:LegEt}). This
means that $E_{F}$ can be given by the Legendre form $z^{2}=w(w-1)(w-\lambda)$
for some $\lambda\in F$. From \cite[Proposition 3.(b), Chapter III]{silverman2009thearithmetic},
we have a linear coordinate change of the form
\begin{equation}
\begin{aligned}w & =u^{2}w'+r\\
z & =u^{3}z'+su^{2}z'+t
\end{aligned}
\quad(u\in F^{*},\,r,s,t\in F)\label{eq:sys-eqs}
\end{equation}
that transforms the Legendre equation to equation (\ref{eq:e123}).
By this transform, 2-torsion points map to 2-torsion points. If $(0,0)$
and $(1,0)$ map to $(e_{i},0)$ and $(e_{j},0)$ respectively, then
solving equations (\ref{eq:sys-eqs}) gives $u=\sqrt{e_{j}-e_{i}}$.
Thus, $K(e_{1},e_{2},e_{3},\sqrt{e_{j}-e_{i}})\subset F$ for some
two distinct $i,j\in\{1,2,3\}$. It follows that $L\subset F$. We
have proved the desired minimality.
\end{proof}
Let $\cY_{0},\cY_{\infty}\subset\cX$ be the reduced closed substacks
at $j=0$ and $j=\infty$, respectively. Let $x\in\cX(k'\tpars)$
be an elliptic curve given by a Weierstrass equation
\[
z^{2}=w^{3}+a_{2}w^{2}+a_{4}w+a_{6}\quad(a_{i}\in k'\tpars).
\]
To describe a formula for $\fd(x)$, we define several auxiliary invariants.
When $x^{*}(j)\ne0$, we have $a_{2}\ne0$ and define 
\[
\alpha_{2}(x):=\ord_{t}a_{2}\mod 2\in\ZZ/2\ZZ.
\]
When $x^{*}(j)=0$, we define 
\begin{align*}
\alpha_{4}(x) & :=\ord_{t}a_{4}\mod 4\in\ZZ/4\ZZ.
\end{align*}
For the twisted arc $\gamma$ associated to $x$, we also write these
values as $\alpha_{2}(\gamma)$ and $\alpha_{4}(\gamma)$, respectively.
Similarly, in what follows, we often identify $x$ with $\gamma$
and interchange the two symbols. Note that $\alpha_{2}(x)$ and $\alpha_{4}(x)$
are independent of the choice of the Weierstrass equation and that
these values are invariant under extensions of the coefficient field
$k'$. We get well-defined maps
\begin{gather*}
\alpha_{2}\colon|\cJ_{\infty}\cX|\setminus(|\cJ_{\infty}\cY_{0}|\cup|\cJ_{\infty}\cY_{\infty}|)\to\ZZ/2\ZZ,\\
\alpha_{4}\colon|\cJ_{\infty}\cY_{0}|\to\ZZ/4\ZZ.
\end{gather*}
Let 
\begin{align*}
\ord_{0}\colon|\cJ_{\infty}\cX|\setminus|\cJ_{\infty}\cY_{0}| & \to\ZZ_{\ge0},\\
\ord_{\infty}\colon|\cJ_{\infty}\cX|\setminus|\cJ_{\infty}\cY_{\infty}| & \to\ZZ_{\ge0}
\end{align*}
be the order functions at $0$ and $\infty$, respectively. If $\cI_{0}$
and $\cI_{\infty}$ are ideal sheaves in $\cO_{X}$ defining closed
points at $0$ and $\infty$ and if $\pi^{-1}\cI_{0}$ and $\pi^{-1}\cI_{\infty}$
are their pullbacks to $\cX$ as ideal sheaves, then $\ord_{0}$ and
$\ord_{\infty}$ are the order functions associated to $\pi^{-1}\cI_{0}$
and $\pi^{-1}\cI_{\infty}$, respectively \cite[Definition 8.2]{yasuda2024motivic2}.
More explicitly, the function $\ord_{0}$ is described as
\[
\ord_{0}(\gamma)=\begin{cases}
0 & (\gamma\in|\cJ_{\infty}\cY_{\infty}|\text{ or }\ord_{t}x^{*}(j)\le0)\\
\ord_{t}x^{*}(j) & (\text{otherwise}).
\end{cases}
\]
We have a similar description of $\ord_{\infty}$ by using $\ord_{t}x^{*}(1/j)$. 

When the elliptic curve $x\in\cX(k'\tpars)$ has potential good reduction,
following the notation from Lemma \ref{lem:min-ext-ell}, let $\delta=(e_{1}-e_{2})(e_{2}-e_{3})(e_{3}-e_{1})$.
Then, $\delta^{2}\in K=k'\tpars$ and we get the following tower of
field extensions:
\[
\xymatrix{L\ar@{-}[d]^{\le C_{2}}\\
K(e_{1},e_{2},e_{3})\ar@{-}[d]^{\text{\ensuremath{\le}\ensuremath{C_{3}}}}\\
K(\delta)\ar@{-}[d]^{\le C_{2}}\\
K
}
\]
Here the notation ``$\le C_{i}$'' means that the corresponding
field extension is Galois with Galois group being included in the
cyclic group $C_{i}$ of order $i$. We define $\bj(\gamma)$ be the
ramification jump of the extension $K(e_{1},e_{2},e_{3})/K(\delta)$.
If $K(e_{1},e_{2},e_{3})=K(\delta)$, we put $\bj(\gamma)=0$. Otherwise,
$\bj(\gamma)$ is a positive integer coprime to 3. From the first
displayed formulas in pages 435 and 436 of \cite{miyamoto2005reduction},
if we define $\xi\in K(\delta)$ by 
\[
\xi=\begin{cases}
1/\sqrt{x^{*}(j)} & (x^{*}(j)\ne0)\\
a_{6}/a_{4}\sqrt{-a_{4}} & (x^{*}(j)\ne0),
\end{cases}
\]
then $K(e_{1},e_{2},e_{3})/K(\delta)$ is obtained by adjoining a
root of the Artin--Schreier equation 
\[
T^{3}-T=\xi.
\]
As is well-known, the ramification jump is given by 
\[
\bj(\gamma)=-\min\left\{ 0,\sup_{c\in K(\delta)}v_{K(\delta)}(\xi-\wp c)\right\} .
\]
where $v_{K(\delta)}$ is the normalized valuation of $K(\delta)$
and $\wp$ is the Artin--Schreier map sending $c$ to $c^{p}-c$.
Using the notation by Miyamoto--Top \cite[pages 432 and 434]{miyamoto2005reduction},
we have 
\begin{equation}
\bj'(\gamma):=\frac{2\bj(\gamma)}{[K(\delta):K]}=\begin{cases}
-\widetilde{n} & (\widetilde{n}\ne\infty)\\
0 & (\widetilde{n}=\infty).
\end{cases}\label{eq:j'-tilde-n}
\end{equation}
(Note that the factor $[K(\delta):K]$ comes from the fact that the
definition of $v^{*K'}(-)$ in \cite[pages 432]{miyamoto2005reduction}
uses the valuation normalized for $K$ rather than for $K'$.)
\begin{prop}[{\cite[Theorem 2.1]{miyamoto2005reduction}}]
\label{prop:conductor}For $\gamma\in|\cJ_{\infty}\cX|\setminus|\cJ_{\infty}\cY_{\infty}|$,
we have
\[
\ff(\gamma)=\begin{cases}
1+\alpha_{2}(\gamma) & (\ord_{\infty}(\gamma)>0))\\
0 & (\ord_{\infty}(\gamma)=0,\,\gamma\in|\J_{\infty}\cX|)\\
2+\bj'(\gamma) & (\text{otherwise}).
\end{cases}
\]
\end{prop}

\begin{proof}
We explain how to rewrite \cite[Theorem 2.1]{miyamoto2005reduction}
using our notation to get this proposition. The case where $\ord_{\infty}(\gamma)>0$
corresponds to the case $n>0$ in the cited theorem and the parity
of $n_{2}$ is given by $\alpha_{2}$. Thus, the proposition holds
in this case. The case where $\ord_{\infty}(\gamma)=0$ and $\gamma\in|\J_{\infty}\cX|$
is exactly when the elliptic curve corresponding to $\gamma$ has
good reduction and hence Kodaira type $I_{0}$. In this case, $\ff(\gamma)=0$,
as given in the table in the cited theorem. In the remaining case,
the cited theorem together with (\ref{eq:j'-tilde-n}) gives $\ff(\gamma)=2+\bj'(\gamma)$. 
\end{proof}
\begin{prop}[{\cite[Theorem 2.1]{miyamoto2005reduction}}]
\label{prop:min-disc}We have
\[
\fd(\gamma)=\begin{cases}
\ord_{\infty}(\gamma)+6\alpha_{2} & (\gamma^{*}j\ne0,\,\ord_{\infty}(\gamma)>0)\\
12\left\lfloor \frac{10-\fe(\gamma)+\bj'(\gamma)}{12}\right\rfloor +\fe(\gamma) & (\text{otherwise}).
\end{cases}
\]
Here $\fe(\gamma)\in\{0,1,\dots,11\}$ is given by
\[
\fe(\gamma)=\begin{cases}
-\ord_{0}\gamma+6\alpha_{2}(\gamma)\mod{12} & (\gamma^{*}j\ne0)\\
3\alpha_{4}(\gamma) & (\gamma^{*}j=0).
\end{cases}
\]
\end{prop}

\begin{proof}
We translate \cite[Theorem 2.1]{miyamoto2005reduction} into our notation
as follows. The case $\ord_{\infty}(\gamma)>0$ corresponds to the
case $n>0$ according to the notation of Miyamoto--Top. Moreover,
we have $n=\ord_{\infty}(\gamma)$. Their theorem says that $\fd(\gamma)$
is either $n$ or $n+6$, depending on the parity of $\ord_{t}(a_{2})$.
Thus, the formula of our proposition holds in this case.

In the case $\ord_{\infty}(\gamma)=0$, $\fe(\gamma)$ is denoted
by $d'$ in the cited paper. If $\bj'(\gamma)>0$, then the formula
of the proposition follows from the cited theorem. In the remaining
case where $\bj'=\bj=0$, equivalently $\widetilde{n}\ge0$, their
theorem says that 
\[
\fd=\fe\in\{0,3,6,9\}.
\]
It follows that 
\[
\left\lfloor \frac{10-\fe(\gamma)+\bj'(\gamma)}{12}\right\rfloor =0.
\]
and hence the formula of the proposition holds also in this case. 
\end{proof}
\begin{lem}
\label{lem:const-alpha}There are at most countably many stable subsets
$B_{i}\subset|\cJ_{\infty}\cX|$, $i\in I$ such that $|\cJ_{\infty}\cX|\setminus(|\cJ_{\infty}\cY_{0}|\cup|\cJ_{\infty}\cY_{\infty}|)=\bigsqcup_{i\in I}B_{i}$
and such that $\alpha_{2}$ is constant on each $B_{i}$.
\end{lem}

\begin{proof}
Applying Propositions \ref{prop:versal-affine} and \ref{prop:factor-open},
we get 
\begin{enumerate}
\item an étale surjective morphism $\coprod_{s\in S}\Spec R_{s}\to\cJ_{\infty}\cX$
with $S$ being at most countable, 
\item for each $s\in S$, at most countably many locally closed, constructible
and affine subschemes $\Spec R_{s,i}\subset\Spec R_{s}$, $i\in I_{s}$,
\item for each $s\in S$ and $i\in I_{s}$, a morphism $\Spec R_{s,i}\tpars\to\cV$,
\end{enumerate}
such that the image of $\coprod_{s\in S}\coprod_{i\in I_{s}}\Spec R_{s,i}$
in $|\cJ_{\infty}\cX|$ is $|\cJ_{\infty}\cX|\setminus(|\cJ_{\infty}\cY_{0}|\cup|\cJ_{\infty}\cY_{\infty}|)$
and for every geometric point $\Spec k'\to\Spec R_{s,i}$, the induced
$k'\tpars$-point $\Spec k'\tpars\to\cV\to\cX$ of $\cX$ corresponds
to the induced $k'$-point of $\cJ_{\infty}\cX$,
\[
\Spec k'\to\Spec R_{s,i}\hookrightarrow\Spec R_{s}\to\cJ_{\infty}\cX.
\]

We now fix $s$ and $i$ and denote $R_{s,i}$ by $R$. Since $\Spec R\tpars$
is quasi-compact, from an argument in \cite[Section 13.1.6]{olsson2016algebraic},
we can find a finite Zariski cover $\Spec R\tpars=\bigcup_{j=1}^{n}\Spec R\tpars_{f_{j}}$
by distinguished open subsets $\Spec R\tpars_{f_{j}}$ such that for
each $j$, the induced elliptic curve over $R\tpars_{f_{j}}$ is defined
by a Weierstrass equation
\[
z^{2}=w^{3}+a_{j,2}w^{2}+a_{j,4}w+a_{j,6}\quad(a_{j,i}\in R\tpars_{f_{j}}).
\]
For each $j$, let us write $a_{j,2}=b_{j}/f_{j}^{m_{j}}$, where
$b_{j},f_{j}\in R\tpars$ and $m_{j}\in\ZZ$. For every point $x\in\Spec R$,
there exists $j$ such that $\ord_{b_{j}}(x)<\infty$ and $\ord_{f_{j}}(x)<\infty$.
Thus, every point of $\Spec R$ is contained in the subset
\[
\{\ord_{b_{j}}=n_{1}\}\cap\{\ord_{f_{j}}=n_{2}\}
\]
for some $j$ and for some integers $n_{1}$ and $n_{2}$, which is
constructible from Lemma \ref{lem:ord-constr}. On this constructible
subset, $\ord_{a_{j,2}}$ is constant. Since $\Spec R$ is quasi-compact
for the constructible topology \cite[Chapter I, Proposition 1.9.15]{grothendieck1964elements},
it is covered by finitely many of these constructible subsets. This
shows that $\Spec R$ decomposes into the two constructible subsets
\[
D_{s,i,e}:=\{\alpha_{2}=e\}\subset\Spec R=\Spec R_{s,i}\quad(e\in\{0,1\}).
\]
There images in $|\cJ_{\infty}\cX|$, which we denote by $C_{s,i,e}$
respectively, are constructible and cover $|\cJ_{\infty}\cX|\setminus(|\cJ_{\infty}\cY_{0}|\cup|\cJ_{\infty}\cY_{\infty}|)$.
There are at most countably many of them. On each of them, $\alpha_{2}$
is constant. Relabeling $C_{s,i,e}$ and denoting them by $C_{m}$,
$m\in\ZZ_{\ge0}$, we define $B_{m}:=C_{m}\setminus\bigcup_{v=0}^{m-1}C_{v}$
so that $|\cJ_{\infty}\cX|\setminus(|\cJ_{\infty}\cY_{0}|\cup|\cJ_{\infty}\cY_{\infty}|)=\bigsqcup_{m\in\ZZ_{\ge0}}B_{m}$. 

Since $\cX$ is smooth, from \cite[Lemma 6.22]{yasuda2024motivic2},
the Untwisting stack $\Utg_{\Gamma_{\cX}}(\cX)$ (see Section \ref{subsec:Constr-tw})
is rig-smooth (relative to the base $\Df_{\Gamma_{\cX}}$). Therefore,
for any quasi-compact substack $\Upsilon\subset\Gamma_{\cX}$, the
Jacobian order function $\fj_{\cX}$ \cite[Definition 8.4]{yasuda2024motivic2}
is bounded on the subset $|\cJ_{\Upsilon,\infty}\cX|\subset|\cJ_{\infty}\cX|$.
From \cite[Lemma 10.15]{yasuda2024motivic2}, $|\cJ_{\Upsilon,\infty}\cX|$
is a stable subset. From Corollary \ref{cor:constr-finite-level},
every $B_{m}$ is again stable. 
\end{proof}
\begin{lem}
\label{lem:const-j}For any choice of $\Gamma_{\cX}$ (see Remark
\ref{rem:not unique} and Section \ref{subsec:Constr-tw}), the functions
$\bj$ and $\bj'$ are constant on each connected component of $|\cJ_{\infty}\cV|$.
\end{lem}

\begin{proof}
Consider an $R$-point $\Spec R\to\cJ_{\infty}\cV$ with $\Spec R$
connected. For each geometric point $r\colon\Spec k'\to\Spec R$,
let us write the corresponding twisted arc as $[W_{r}/G_{r}]\to\cX$,
where $W_{r}$ is a connected regular Galois cover of $\D_{k'}$ with
Galois group $G_{r}$. The Galois extension associated to $W_{r}\to\D_{k'}$
is the minimal extension $L$ of $K=k'\tpars$ such that the elliptic
curve $E_{r}/k'\tpars$ corresponding to $r$ has good reduction after
the base change to $L$. The values of $\bj$ and $\bj'$ are determined
by the ramification jump of the intermediate extension $K(e_{1},e_{2},e_{3})/K(\delta)$
of $L/K$ and by whether the subextension $K(\delta)/K$ is trivial.
It is enough to show that these data are independent of the chosen
geometric point $r$.

From the constructions of $\Gamma_{\cX}$ and $\cJ_{\infty}\cX$ (Sections
\ref{subsec:Universal-twisted-formal} and \ref{subsec:Constr-tw}),
for any $R$-point $\Spec R\to\cJ_{\infty}\cV$ with $\Spec R$ connected,
the induced morphism
\[
\Spec R\to\cJ_{\infty}\cV\to\cJ_{\infty}\cX\to\Gamma_{\cX}\to\Lambda
\]
has image in $\Lambda_{[G]}^{[\br]}$ for some Galoisian group $G$
and some ramification data $\br$. In particular, the ramification
jump of $K(e_{1},e_{2},e_{3})/K(\delta)$ is independent of the geometric
point $r$. Also, whether the extension $K(\delta)/K$ is trivial
is independent of $r$. The lemma follows.
\end{proof}
\begin{defn}
We define functions $\ff'$ and $\fd'$ on $|\cJ_{\infty}\cX|$ by
$\ff':=\ff-\ord_{\infty}$ and $\fd':=\fd-\ord_{\infty}$, respectively.
Here we follow the convention $\ff'(\gamma)=\fd'(\gamma)=\infty$
for $\gamma\in|\cJ_{\infty}\cY_{\infty}|$. 
\end{defn}

\begin{lem}
The functions $\ff'$ and $\fd'$ are raising functions. 
\end{lem}

\begin{proof}
For each of these functions, we need to find a suitable sectoroid
subdivision of $|\cJ_{\infty}\cX|$. We can first put $A_{\infty}=|\cJ_{\infty}(\cY_{0}\cup\cY_{\infty})|$
so that the behaviors of $\ff'$ and $\fd'$ on this set do not matter,
in particular, they are allowed to take value $\infty$ at some points
of $A_{\infty}$. From Propositions \ref{prop:conductor} and \ref{prop:min-disc}
and Lemmas \ref{lem:const-alpha} and \ref{lem:const-j}, we can easily
find a desired sectoroid subdivision except condition (5) of Definition
\ref{def:sectoroid-subdivision}. The open substack denoted by $\cX^{\circ}$
in Definition \ref{def:sectoroid-subdivision} is $\cX\setminus\cY_{0}$
in the current situation. We need to show that $\ff'$ and $\fd'$
are zeroes on $|\J_{\infty}\cX^{\circ}|\setminus C$ for some Gorenstein
measurable subset $C$ with $\gw(C)<-1$. If $\gamma\in|\J_{\infty}\cX^{\circ}|$
and if $\ord_{\infty}(\gamma)=0$, then the corresponding elliptic
curve has good reduction and hence $\ff'(\gamma)=\fd'(\gamma)=0$.
By a standard computation of motivic measure for arcs on smooth varieties
and its straightforward generalization to untwisted arcs on smooth
stacks, we see that the locus in $|\J_{\infty}\cX^{\circ}|$ with
$\ord_{\infty}\ge2$ has Gorenstein weight $<-1$. Thus, it is enough
to show that if $\ord_{\infty}(\gamma)=1$, then $\ff'(\gamma)=\fd'(\gamma)=0$,
equivalently $\alpha_{2}(\gamma)=0$. For $\gamma\in|\J_{\infty}\cX|$
with $\ord_{\infty}(\gamma)>0$, since $\cX$ is a moduli stack of
semistable curves, the corresponding elliptic curve over $k'\tpars$
has multiplicative/semistable reduction. From formulas in page 42
and Proposition 5.1(b) of \cite{silverman2009thearithmetic}, for
a minimal Weierstrass equation, $\ord_{t}(a_{2})=0$, in particular,
we get $\alpha_{2}(\gamma)=0$ as desired.
\end{proof}
\begin{cor}
The conductor $\fF$ and the minimal discriminant $\fD$ are the height
functions associated to the raised line bundles $(\pi^{*}\cO(1),\ff')$
and $(\pi^{*}\cO(1),\fd')$, respectively. 
\end{cor}

\begin{proof}
The assertion for the conductor follows from the equalities 
\[
\fF(x)=\prod_{v}q_{v}^{\ff(x_{v})}=\prod_{v}q_{v}^{\ord_{\infty}(x_{v})}\cdot\prod_{v}q_{v}^{\ff'(x_{v})}.
\]
and
\[
\prod_{v}q_{v}^{\ord_{\infty}(x_{v})}=H_{\pi^{*}\cO(1)}(x).
\]
Similarly for the discriminant. 
\end{proof}

\appendix

\section{Moduli of formal torsors\label{sec:Moduli-of-formal}}

Let $k$ be a field of characteristic $p>0$. Let $G$ be a finite
étale group scheme over $k\tpars$ satisfying the following condition:
\begin{condition}
The base change $G_{k\tpars^{\sep}}$ of $G$ to $k\tpars^{\sep}$
is a constant group of the form
\[
\prod_{i=1}^{n}(H_{i}\rtimes C_{i}),
\]
where $H_{i}$ are $p$-groups and $C_{i}$ are tame cyclic groups. 
\end{condition}

\begin{rem}
In particular, if $G$ is commutative, then the above condition holds.
\end{rem}

We define the stack $\Delta_{G}$ of $G$-torsors over $\D_{k}^{*}$
as follows. This is a stack over $k$. For an affine $k$-scheme $\Spec R$,
the fiber $\Delta_{G}(\Spec R)$ over $\Spec R$ is the groupoid of
$G$-torsors over $\D_{R}^{*}=\Spec R\tpars$. 
\begin{thm}
The stack $\Delta_{G}$ is isomorphic to the direct limit $\injlim\cX_{n}$
of a direct system of DM stacks $\cX_{n}$, $n\in\ZZ_{\ge0}$ satisfying
the following conditions:
\begin{enumerate}
\item $\cX_{n}$ are separated and of finite type over $k$,
\item transition maps $\cX_{n}\to\cX_{n+1}$ are representable, finite and
universally injective,
\item there exists a direct system of finite and étale atlases $X_{n}\to\cX_{n}$
(see \cite[Definition 3.3]{tonini2020moduliof}).
\end{enumerate}
Note that when $G$ is defined over $k$, then the theorem is a direct
consequence of a main result of \cite{tonini2020moduliof}. 
\end{thm}

\begin{proof}
In this proof, we call the limit of DM stacks satisfying the above
conditions an \emph{ind-DM stack.} Let $S_{0}=\D_{k}^{*}$, a punctured
formal disk. There exists an étale finite cover $S_{1}\to S_{0}$
such that the base change $G_{S_{1}}$ of $G$ to $S_{1}$ is constant.
Let $S_{2}:=S_{1}\times_{S_{0}}S_{1}$ and $S_{3}:=S_{1}\times_{S_{0}}S_{1}\times_{S_{0}}S_{1}$.
There exists a finite Galois extension $k'/k$ such that the base
changes $S_{1,k'},S_{2,k'},S_{3,k'}$ are of the form $\D_{k'}^{*}$
\[
\D_{k'}^{*}\sqcup\cdots\sqcup\D_{k'}^{*}.
\]
From \cite[Proposition 3.5]{tonini2020moduliof}, it is enough to
show that the base change $(\Delta_{G})\otimes_{k}k'$ of $\Delta_{G}$
is an ind-DM stack. Thus, we may assume that there exists a finite
étale cover $S_{1}\to S_{0}$ such that $S_{1}$ as well as $S_{2}$
and $S_{3}$ defined as above is the disjoint union of finitely many
copies of $\D_{k}^{*}$. 

Giving a $G$-torsor over $S_{0}$ is equivalent to giving a $G$-torsor
$T_{1}\to S_{1}$ with an automorphism $\alpha$ of the $G$-torsor
\[
T_{2}:=T_{1}\times_{S_{1}}S_{2}\to S_{2}
\]
which satisfies the cocycle condition. Let $\cU_{1}$, $\cU_{2}$
and $\cU_{3}$ be the stacks of $G$-torsors over $S_{1}$, $S_{2}$
and $S_{3}$, respectively. From Lemma \ref{lem:coeff-ext}, if $S_{i}$
consists of $r$ copies of $\D_{k}^{*}$, then $\cU_{i}$ is isomorphic
to the product of $r$ copies of $\Delta_{G}$ (over the base field
$k$). From \cite[Proposition A.2]{tonini2020moduliof}, these stacks
are ind-DM stacks. 

Let $p_{i}\colon S_{2}\to S_{1}$, $i=1,2$ be the first and second
projections and let $q_{i}\colon S_{3}\to S_{1}$, $i=1,2,3$ be the
first to third projections. Let $q_{ij}\colon S_{3}\to S_{2}$ be
the projection to the $i$-th and the $j$-th factors. We have natural
pullback morphisms $p_{i}^{*}\colon\cU_{1}\to\cU_{2}$ and $q_{i}^{*}\colon\cU_{1}\to\cU_{3}$
and $q_{ij}^{*}\colon\cU_{2}\to\cU_{3}$. 

Consider the stack 
\[
\cV:=(\cU_{1}\times_{p_{1}^{*},\cU_{2},p_{2}^{*}}\cU_{1})\times_{\cU_{1}\times\cU_{1},\Delta_{\cU_{1}}}\cU_{1}.
\]
Again from \cite[Proposition A.2]{tonini2020moduliof}, this is an
ind-DM stack. This stack is equivalent to the stack $\cV_{1}'$ of
pairs $(T_{1},\alpha)$ of a $G$-torsor $T_{1}\to S_{1}$ and an
automorphism $\alpha$ of $T_{2}:=T_{1}\times_{S_{1}}S_{2}$ (possibly
without satisfying the cocycle condition). To incorporate the cocycle
condition, we consider the morphism 
\[
\cV'\to\I\cU_{3},\,(T_{1},\alpha)\mapsto(q_{1}^{*}T_{1},\alpha'),
\]
where $\I\cU_{3}$ is the inertia stack of $\cU_{3}$ and $\alpha'$
is the composite morphism 
\[
q_{1}^{*}T_{1}\xrightarrow{q_{12}^{*}\alpha}q_{2}^{*}T_{1}\xrightarrow{q_{23}^{*}\alpha}q_{3}^{*}T_{1}\xrightarrow{q_{31}^{*}\alpha}q_{1}^{*}T_{1}.
\]
Let $\iota\colon\cU_{3}\to\I\cU_{3}$ be the identity section. We
see that the desired stack $\Delta_{G}$ is realized as the fiber
product $\cV'\times_{\I\cU_{3}}\cU_{3}$. Since $\iota$ is an open
and closed immersion, so is the projection $\Delta_{G}\to\cV'$. This
shows that $\Delta_{G}$ is also an ind-DM stack.
\end{proof}
\begin{lem}
\label{lem:coeff-ext}Let $k\tpars\to k\spars$ be a finite étale
map of $k$-algebras. For a $k$-algebra $R$, the canonical map 
\[
k\spars\otimes_{k\tpars}R\tpars\to R\spars
\]
is an isomorphism.
\end{lem}

\begin{proof}
The ring $R\sbrats$ is the $s$-adic completion of $k\sbrats\otimes_{k}R$.
On the other hand, $k\sbrats\otimes_{k\tbrats}R\tbrats$ is the $t$-adic
completion of $k\sbrats\otimes_{k}R$. Since $t=s^{r}\cdot u$ for
some $r\in\ZZ_{>0}$ and a unit $u\in k\sbrats^{*}$, we see that
the $s$-adic topology and the $t$-adic topology on $k\sbrats\otimes_{k}R$
are the same. Thus, we have a canonical isomorphism $R\sbrats\cong k\sbrats\otimes_{k\tbrats}R\tbrats$.
We also have $k\spars=k\sbrats_{s}=k\sbrats_{t}$, where subscript
means localizations. Similarly $R\spars=R\sbrats_{t}$. We get 
\begin{align*}
k\spars\otimes_{k\tpars}R\tpars & =k\sbrats_{t}\otimes_{k\tbrats_{t}}R\tbrats_{t}\\
 & =(k\sbrats\otimes_{k\tbrats}R\tbrats)_{t}\\
 & =R\sbrats_{t}\\
 & =R\spars.
\end{align*}
\end{proof}

\section{Auxiliary results on stacks of twisted arcs\label{sec:Morphisms}}

We keep working over a field $k$ of characteristic $p>0$. We denote
by $\Df_{k}$ the formal spectrum $\Spf k\tbrats$. Following the
formalism in \cite{yasuda2024motivic2}, we work with formal DM stacks
over $\Df_{k}$ and ones over $\Df_{\Gamma}:=\Df_{k}\times_{k}\Gamma$
for some DM stack $\Gamma$. For a formal DM stack $\cX$ over $\Df_{k}$,
we denote its special fiber $\cX\times_{\Df_{k}}\Spec k$ by $\cX_{0}$.
Note that 

\subsection{Morphisms between inertia-type stacks\label{subsec:inertia stacks}}

Firstly, we recall variations of the inertia stack from \cite[Section 6.4]{yasuda2024motivic2}. 

\begin{defn}
We say that a finite group $G$ is \emph{Galoisian }if it is isomorphic
to the Galois group of a Galois extension of $k'\tpars$ for some
algebraically closed field $k'$ over $k$. We denote by $\GalGps$
the set of representatives of isomorphism classes of Galoisian groups.
We say that a finite étale group scheme $\cG\to S$ over a scheme
$S$ is \emph{Galoisian }if all its geometric fibers are Galoisian.
We denote by $\cA$ the moduli stack of Galoisian group schemes; for
a scheme $S$, $\cA(S)$ is the groupoid of Galoisian group schemes
over $S$. We define $\cA_{[G]}$ to be the substack of $\cA$ that
parametrizes Galoisian group schemes with geometric fibers isomorphic
to $G$.
\end{defn}

We have $\cA=\coprod_{G\in\GalGps}\cA_{[G]}$ and $\cA_{[G]}\cong\B(\Aut G)=[\Spec k/\Aut G]$.
For each $G\in\GalGps$, the canonical atlas 
\[
u_{G}\colon\Spec k\to\cA_{[G]}
\]
corresponds to the constant group scheme $G\times\Spec k\to\Spec k$.
For each $G\in\GalGps$, there exists the universal group scheme $\cG_{[G]}\to\cA_{[G]}$.
We define $\B\cG_{[G]}$ to be the stack over $\cA_{[G]}$ parametrizing
$\cG_{[G]}$-torsors \cite[Definition 6.14]{yasuda2024motivic2}.
Taking the coproducts over $G\in\GalGps$, we obtain the universal
Galoisian group scheme $\cG\to\cA$ and the stack $\B\cG$ over $\cA$.
\begin{defn}
Let $\cX_{0}$ be a DM stack of finite type over $k$. For a Galoisian
group $G$, we define 
\[
\I^{[G]}\cX_{0}:=\ulHom_{\cA_{[G]}}^{\rep}(\B\cG_{[G]},\cX_{0})\text{ and }\I^{(G)}\cX_{0}:=\ulHom_{\Spec k}^{\rep}(\B G,\cX_{0}).
\]
We also define 
\[
\I^{[\bullet]}\cX_{0}:=\coprod_{G\in\GalGps}\I^{[G]}\cX_{0}\text{ and }\I^{(\bullet)}\cX_{0}:=\coprod_{G\in\GalGps}\I^{(G)}\cX_{0}.
\]
\end{defn}

Note that $\I^{[G]}\cX_{0}$ and $\I^{(G)}\cX_{0}$ are empty for
all but finitely many $G$. There exists a natural morphism $\I^{(G)}\cX_{0}\to\I^{[G]}\cX_{0}$
fitting into the following 2-Cartesian diagram:
\[
\xymatrix{\I^{(G)}\cX_{0}\ar[r]\ar[d] & \I^{[G]}\cX_{0}\ar[d]\\
\Spec k\ar[r]^{u_{G}} & \cA_{[G]}
}
\]

Let $f_{0}\colon\cY_{0}\to\cX_{0}$ be a morphism of DM stacks of
finite type over $k$. We will construct a morphism 
\begin{equation}
\I^{[\bullet]}f_{0}\colon\I^{[\bullet]}\cY_{0}\to\I^{[\bullet]}\cX_{0}.\label{eq:If}
\end{equation}
Note that since we do not require $f_{0}$ to be representable, we
cannot construct a a natural morphism $\I^{[G]}\cY_{0}\to\I^{[G]}\cX_{0}$
for each $G$; the morphism $\I^{[\bullet]}f_{0}$ does not send $\I^{[G]}\cY_{0}$
into $\I^{[G]}\cX_{0}$ for the same $G$ in general. 

Let $\cG\to S$ be a Galoisian group scheme and let $\beta\colon\B\cG\to\cY_{0}$
be a representable $k$-morphism corresponding to an $S$-point of
$\I^{[\bullet]}\cY$. Composing it with the canonical atlas $b\colon S\to\B\cG$,
we get an $S$-point $y\colon S\to\cY_{0}$ and its image $x\colon S\to\cX_{0}$.
Since $\beta$ is representable, the homomorphism
\[
\cG=\ulAut(b)\to\ulAut(y)
\]
of group schemes over $S$ induced by $\beta$ is injective. Let 
\[
\cK:=\Ker(\cG\to\ulAut(y)\to\ulAut(x))
\]
and let $\cH:=\cG/\cK$, which is again a Galoisian group scheme over
$S$. The morphism $f\circ\beta\colon\B\cG\to\cY_{0}\to\cX_{0}$ uniquely
factors as $\B\cG\to\B\cH\to\cX_{0}$, where $\B\cG\to\B\cH$ is the
morphism associated to the quotient morphism $\cG\to\cH$ and $\B\cH\to\cX$
is a representable morphism. Sending the given $S$-point $(\beta\colon\B\cG\to\cY_{0})\in(\I^{[\bullet]}\cY_{0})(S)$
to the obtained $S$-point $(\B\cH\to\cX_{0})\in(\I^{[\bullet]}\cX_{0})(S)$
defines the desired morphism $\I^{[\bullet]}f_{0}$. 

\subsection{Universal twisted formal disks\label{subsec:Universal-twisted-formal}}
\begin{defn}[{\cite[Definitions 5.9, 5.10 and 5.12]{yasuda2024motivic2}}]
We define $\Lambda$ to be the stack of fiberwise connected étale
finite torsors over the punctured formal disk $\D_{k}^{*}=\Spec k\tpars$
with locally constant ramification.
\end{defn}

This stack is endowed with a morphism $\Lambda\to\cA$ and for the
morphism $\Spec R\to\cA$ corresponding to a Galoisian group scheme
$\cG\to\Spec R$, the fiber $\Lambda(\Spec R)$ over this $R$-point
of $\cA$ is the groupoid of fiberwise connected $\cG$-torsors over
$\D_{R}^{*}:=\Spec R\tpars$. We can write $\Lambda$ as the countable
coproduct 
\[
\Lambda=\coprod_{G\in\GalGps}\coprod_{[\br]\in\Ram(G)/\Aut(G)}\Lambda_{[G]}^{[\br]},
\]
where $[\br]$ runs over the $\Aut(G)$-orbits $[\br]$ of ramification
data $\br$ for $G$. 

Each $\Lambda_{[G]}^{[\br]}$ is the ind-perfection of a reduced DM
stack $\Theta_{[G]}^{[\br]}$ of finite type over $k$. We put $\Theta=\coprod_{[G]}\coprod_{[\br]}\Theta_{[G]}^{[\br]}$
so that $\Lambda$ is the ind-perfection of $\Theta$. Note that the
choice of $\Theta$ is unique only up to universal homeomorphisms;
if we have two choices $\Theta$ and $\Theta'$ as such stacks, then
there is a third choice $\Theta''$ together with canonical morphisms
$\Theta''\to\Theta$ and $\Theta''\to\Theta'$ which are representable
and universal homeomorphisms. Moreover, we may and shall choose $\Theta_{[G]}^{[\br]}$
in such a way that there exists a universal integral model $E_{\Theta}\to\Df_{\Theta}$;
for each point $\Spec k'\to\Theta$, the base change $E_{\Theta}\times_{\Df_{\Theta}}\Df_{k'}$
is the formal spectrum of the integral closure of $k'\tbrats$ in
the torsor over $\D_{k'}^{*}=\Spec k'\tpars$ corresponding to $\theta$. 
\begin{defn}[{\cite[Definition 5.22]{yasuda2024motivic2}}]
We define the \emph{universal formal disk }over $\Theta$ to be the
quotient stack $\cE:=[E_{\Theta}/\cG_{\Theta}]$, which is endowed
with a morphism $\cE_{\Theta}\to\Df_{\Theta}$.
\end{defn}

We now introduce a property of morphisms of stacks which is convenient
for us. 
\begin{defn}
We say that a morphism $g\colon\cW\to\cV$ of stacks is \emph{countably
component-wise of finite type (for short, ccft) }if it is DM \cite[tag 04YW]{stacksprojectauthors2022stacksproject}
and there exists a countable family of open and closed substacks $\cW_{i}\subset\cW$,
$i\in I$ such that $|\cW|=\bigsqcup_{i\in I}|\cW_{i}|$ and $g|_{\cW_{i}}\colon\cW_{i}\to\cV$
are of finite type. When this is the case, we also say that $\cW$
is \emph{ccft over $\cV$. }
\end{defn}

For example, $\Theta$ is ccft over $\Spec k$.

\subsection{Construction of the stack of twisted arcs\label{subsec:Constr-tw}}

Let $\cX$ be a formal DM stack of finite type over $\Df_{k}$ and
let $\cX_{0}:=\cX\times_{\Df_{k}}\Spec k$, the special fiber of $\cX\to\Df_{k}$.
For a DM stack $\varPhi$ over $k$, we denote $\cX\times_{k}\Theta=\cX\times_{\Df_{k}}\Df_{\Theta}$
by $\cX_{\Theta}$. 
\begin{defn}
We define the \emph{total untwisting stack} of $\cX$ to be the Hom
stack
\[
\Utg_{\Theta}(\cX):=\ulHom_{\Df_{\Theta}}^{\rep}(\cE_{\Theta},\cX_{\Theta}),
\]
which is a formal DM stack of finite type over $\Df_{\Theta}$. For
a morphism $\Xi\to\Theta$ of DM stacks, we define 
\[
\Utg_{\Xi}(\cX):=\Utg_{\Theta}(\cX)\times_{\Df_{\Theta}}\Df_{\Xi}=\Utg_{\Gamma}(\cX)\times_{\Theta}\Xi.
\]
 to be the base change. 
\end{defn}

To a formal DM stack $\cZ$ of finite type over $\Df_{k}$, we can
define a closed substack $\cX^{\pur}\subset\cX$ in a unique way by
killing $t$-torsions \cite[Definition 6.30]{yasuda2024motivic2}.
For each $[G]$ and $[\br]$, we can decompose $\Theta_{[G]}^{[\br]}$
into finitely many reduced locally closed substacks $\Theta_{[G],i}^{[\br]}$
such that $|\Theta_{[G]}^{[\br]}|=\bigsqcup_{i}|\Theta_{[G],i}^{[\br]}|$
and $\Utg_{\Theta_{[G],i}^{[\br]}}(\cX)^{\pur}$ are flat over $\Df_{\Theta_{[G],i}^{[\br]}}.$
We define $\Gamma_{\cX}$ to be $\coprod_{[G],[\br],i}\Theta_{[G],k}^{[\br]}$.
Thus, $\Utg_{\Gamma_{\cX}}(\cX)^{\pur}$ is flat over $\Df_{\Gamma_{\cX}}$.
The stack $\Gamma_{\cX}$ is unique up to universal bijections. More
precisely, if $\Gamma_{\cX}$ and $\Gamma_{\cX}'$ are two choices
as such stacks, then there exists a third one $\Gamma_{\cX}''$ given
with canonical morphisms $\Gamma_{\cX}''\to\Gamma_{\cX}$ and $\Gamma_{\cX}''\to\Gamma_{\cX}^{'}$
which are representable, universally bijective and immersions on each
connected component of $\Gamma_{\cX}''$.

For each $n\in\ZZ_{\ge0}$, let $\D_{\Gamma_{\cX},n}:=\Spec k[t]/(t^{n+1})\times_{\Spec k}\Gamma_{\cX}$.
This is a closed substack of $\Df_{\Gamma_{\cX}}$ and we have $\Df_{\Gamma_{\cX}}=\injlim\D_{\Gamma_{\cX},n}$.
We define 
\[
\J_{\Gamma_{\cX},n}\Utg_{\Gamma_{\cX}}(\cX)^{\pur}:=\ulHom_{\Df_{\Gamma_{\cX}}}(\D_{\Gamma_{\cX},n},\Utg_{\Gamma_{\cX}}(\cX)^{\pur}).
\]
This is a DM stack of finite type over $\Gamma_{\cX}$. Its base change
by a point $\Spec k'\to\Gamma_{\cX}$ is the $n$-jet stack of the
formal DM stack $\Utg_{\Spec k'}(\cX)^{\pur}$ flat and of finite
type over $\Df_{k'}$: 
\[
(\J_{\Gamma_{\cX},n}\Utg_{\Gamma_{\cX}}(\cX)^{\pur})\times_{\Gamma_{\cX}}\Spec k'=\ulHom_{\Df_{k'}}(\Spec k'[t]/(t^{n+1}),\Utg_{\Spec k'}(\cX)^{\pur}).
\]

\begin{defn}[{\cite[Definition 7.6]{yasuda2024motivic2}}]
The \emph{stack of twisted arcs} of $\cX$, denoted by $\cJ_{\infty}\cX=\cJ_{\Gamma_{\cX},\infty}\cX$,
is defined to be the arc stack 
\[
\J_{\Gamma_{\cX},\infty}\Utg_{\Gamma_{\cX}}(\cX)^{\pur}=\projlim\J_{\Gamma_{\cX},n}\Utg_{\Gamma_{\cX}}(\cX)^{\pur}.
\]
\end{defn}

\begin{rem}
The decomposition of $\Theta_{[G]}^{[\br]}$ into substacks $\Theta_{[G],i}^{[\br]}$
and construction of flat formal stacks $\Utg_{\Gamma_{\cX}}(\cX)^{\pur}$
were needed to construct a morphism $\Utg_{\Gamma_{\cX}}(\cX)\to X$.
We explained this decomposition above to follow the definition of
$\cJ_{\infty}\cX$, although the decomposition does not play any role
in what follows. 
\end{rem}

By construction, the fiber $(\cJ_{\infty}\cX)(S)$ of the projection
$\cJ_{\infty}\cX\to\Gamma_{\cX}$ over an $S$-point $S\to\Gamma_{\cX}$
with $S$ a scheme is given by 
\begin{align*}
(\cJ_{\infty}\cX)(S) & =\Utg_{\Gamma_{\cX}}(\cX)^{\pur}(\Df_{S})\\
 & =\Utg_{\Gamma_{\cX}}(\cX)(\Df_{S})\\
 & =\ulHom_{\Df_{S}}^{\rep}(\cE_{S},\cX_{\Df_{S}})\\
 & =\ulHom_{\Df_{k}}^{\rep}(\cE_{S},\cX).
\end{align*}
Note that the morphism $S\to\Gamma_{\cX}$ composed with $\Gamma_{\cX}\to\Lambda$
corresponds to a Galoisian group scheme $\cG\to S$ and we have a
canonical closed immersion $\B\cG\hookrightarrow\cE_{S}$. Therefore,
for an $S$-point of $\cJ_{\infty}\cX$ corresponding to a representable
morphism $\cE_{S}\to\cX$, the composition $\B\cG\hookrightarrow\cE_{S}\to\cX$
is again representable and corresponds to an $S$-point of $\I^{[\bullet]}\cX_{0}$.
This construction defines a natural morphism 
\[
\psi_{\cX}\colon\cJ_{\infty}\cX\to\I^{[\bullet]}\cX_{0}.
\]

\subsection{Morphisms between stacks of twisted arcs\label{subsec:mor-tw}}

Let us now consider a morphism $f\colon\cY\to\cX$ of formal DM stacks
of finite type over $\Df_{k}$. We first define a morphism 
\[
\xi_{f}\colon\Lambda\times_{\cA}\I^{[\bullet]}\cY_{0}\to\Lambda\times_{\cA}\I^{[\bullet]}\cX_{0}
\]
compatible with $\I^{[\bullet]}f_{0}$ (see (\ref{eq:If})) as follows.
For a $k$-algebra $R$, let us consider an $R$-point
\[
(T\to\D_{R}^{*},\B\cG\to\cY_{0})\in(\Lambda\times_{\cA}\I^{[\bullet]}\cY_{0})(R).
\]
Here $T\to\D_{R}^{*}$ is a fiberwise connected torsor for a Galoisian
group scheme $\cG\to\Spec R$ and $\B\cG\to\cY_{0}$ is a representable
morphism. Let 
\[
(\B(\cG/\cK)\to\cX_{0})\in(\I^{[\bullet]}\cX_{0})(R)
\]
be the image of $(\B\cG\to\cY_{0})\in(\I^{[\bullet]}\cY_{0})(\Spec R)$
by $\I^{[\bullet]}f_{0}$. The quotient scheme $T/\cK$ is then a
$\cG/\cK$-torsor over $\D_{R}^{*}$, which is again fiberwise connected
and has locally constant ramification. We define $\xi_{f}$ by sending
the given $R$-point of $\Lambda\times_{\cA}\I^{[\bullet]}\cY_{0}$
to 
\[
(T/\cK\to\Spec R\tpars,\B(\cG/\cK)\to\cX_{0})\in(\Lambda\times_{\cA}\I^{[\bullet]}\cX_{0})(R).
\]

For a suitable choice of DM stacks $\Theta'$ and $\Theta$ ccft over
$k$ whose ind-perfections are identified with $\Lambda$, we can
lift $\xi_{f}$ to get a morphism 
\[
\xi_{f}'\colon\Theta'\times_{\cA}\I^{[\bullet]}\cY_{0}\to\Theta\times_{\cA}\I^{[\bullet]}\cX_{0}.
\]
Let us construct $\Gamma_{\cX}$ by decomposing components of $\Theta$
as explained in subsection \ref{subsec:Constr-tw}. Next, we construct
$\Gamma_{\cY}$ by decomposing components of $\Theta'$ in such a
way that the morphism $\xi_{f}'$ induces a morphism 
\[
\zeta_{f}\colon\Gamma_{\cY}\times_{\cA}\I^{[\bullet]}\cY_{0}\to\Gamma_{\cX}\times_{\cA}\I^{[\bullet]}\cX_{0}.
\]

We are now ready to construct a morphism $\cJ_{\infty}\cY\to\cJ_{\infty}\cX$.
We first define stacks of twisted arcs, $\cJ_{\infty}\cY=\cJ_{\Gamma_{\cY},\infty}\cY$
and $\cJ_{\infty}\cX=\cJ_{\Gamma_{\cX},\infty}\cX$, by using $\Gamma_{\cY}$
and $\Gamma_{\cX}$ chosen as above respectively. Let $\gamma\colon S\to\cJ_{\infty}\cY$
be an $S$-point corresponding to a representable morphism
\[
\gamma'\colon\cE_{S}=[E_{S}/\cG_{S}]\to\cY.
\]
The composition $\psi_{\cY}\circ\gamma\colon S\to\cJ_{\infty}\cY\to\I^{[\bullet]}\cY_{0}$
corresponds to a representable morphism $\B\cG_{S}\to\cY_{0}$ and
the composition $(\I^{[\bullet]}f)\circ\psi_{\cY}\circ\gamma\colon S\to\I^{[\bullet]}\cX_{0}$
corresponds to a representable morphism $\B(\cG_{S}/\cK_{S})\to\cX$
for a uniquely determined subgroup scheme $\cK_{S}\subset\cG_{S}$.
Now, the morphism $f\circ\gamma'\colon\cE_{S}\to\cX$ canonically
factors as 
\[
\cE_{S}\to\cF_{S}:=[(E_{S}/\cK_{S})/(\cG_{S}/\cK_{S})]\to\cX,
\]
where the second morphism $\cF_{S}\to\cX$ is representable. By construction,
$\cF_{S}$ is the twisted formal disk corresponding to an $S$-point
of $\Gamma_{\cX}$. (Here we need the choice of $\Gamma_{\cY}$ as
above so that $\zeta_{f}$ exists.) Thus, the induced morphism $\cF_{S}\to\cX$
corresponds to an $S$-point of $\cJ_{\infty}\cX$. This construction
defines a morphism
\[
f_{\infty}\colon\cJ_{\infty}\cY\to\cJ_{\infty}\cX.
\]
From the construction, the following diagram is commutative:
\[
\xymatrix{\cJ_{\infty}\cY\ar[r]^{f_{\infty}}\ar[d] & \cJ_{\infty}\cX\ar[d]\\
\Gamma_{\cY}\times_{\cA}\I^{[\bullet]}\cY_{0}\ar[r]^{\zeta_{f}}\ar[d]^{\pr_{2}} & \Gamma_{\cX}\times_{\cA}\I^{[\bullet]}\cX_{0}\ar[d]^{\pr_{2}}\\
\I^{[\bullet]}\cY_{0}\ar[r]^{\I^{[\bullet]}f} & \I^{[\bullet]}\cX_{0}
}
\]

\subsection{Rigidification\label{subsec:Rigidification}}

We now consider the case where $\cX$ is the fiber product $\cX=\cX_{0}\times_{\Spec k}\Df_{k}$
for a DM stack $\cX_{0}$ of finite type over $k$. Let $(\cX_{0})^{\rig}$
be the rigidification of $\cX_{0}$ and let $\cX^{\rig}:=(\cX_{0})^{\rig}\times_{\Spec k}\Df_{k}$.
We denote by $\rho$ the induced morphism $\cX\to\cX^{\rig}$. 
\begin{prop}
\label{prop:ccft}The morphism $\rho_{\infty}\colon\cJ_{\infty}\cX\to\cJ_{\infty}\cX^{\rig}$
is ccft.
\end{prop}

\begin{proof}
Consider a 2-Cartesian diagram of formal DM stacks 
\[
\xymatrix{\cY'\ar[d]^{f'}\ar[r] & \cY\ar[d]^{f}\\
\cX'\ar[r] & \cX
}
\]
such that horizontal arrows are stabilizer-preserving and étale. Then,
we may suppose that $\Gamma_{\cY'}=\Gamma_{\cY}$ and $\Gamma_{\cX'}=\Gamma_{\cX}$.
Under this assumption, the diagram
\[
\xymatrix{\cJ_{\infty}\cY'\ar[d]^{f'_{\infty}}\ar[r] & \cJ_{\infty}\cY\ar[d]^{f_{\infty}}\\
\cJ_{\infty}\cX'\ar[r] & \cJ_{\infty}\cX
}
\]
is also 2-Cartesian. Using this fact, it is enough to consider the
case $\cX_{0}=[V_{0}/H]$, where $V_{0}$ is an irreducible variety
over $k$ and $H$ is a finite group. Let $V=V_{0}\times_{\Spec k}\Df_{k}$
and let $N$ be the kernel of $H\to\Aut(V)$ and let $Q:=H/N$. Then,
$\cX=[V/H]$ and $\cX^{\rig}=[V/Q]$. From \cite[Section 6.3]{yasuda2024motivic2},
we have
\[
\cJ_{\infty}\cX\cong\coprod_{G\in\GalGps}\left[\left(\coprod_{\iota\in\Emb(G,H)}\J_{\infty}\ulHom_{\Df_{\Gamma_{[V/H],G}}}^{\iota}(E_{\Gamma_{[V/H],G}},V_{\Gamma_{[V/H],G}})\right)/\Aut(G)\times H\right].
\]
Here $\ulHom^{\iota}$ means the Hom (formal) scheme of $\iota$-equivariant
morphisms and $\Gamma_{[V/H],G}$ denotes the fiber product $\Gamma_{[V/H]}\times_{\cA}\Spec k$
induced from the morphism $\Spec k\to\cA$ corresponding to the constant
group $G$. 

For a Galoisian group $G$, let $\cJ_{\infty}^{[G]}\cX:=\cJ_{\infty}\cX\times_{\I^{[\bullet]}\cX_{0}}\I^{[G]}\cX_{0}$
and $\cJ_{\infty}^{(G)}\cX:=\cJ_{\infty}\cX\times_{\I^{[\bullet]}\cX_{0}}\I^{(G)}\cX_{0}$.
The former is an open and closed substack of $\cJ_{\infty}\cX$, while
the latter is an $\Aut(G)$-torsor over the former. We can write 
\[
\cJ_{\infty}^{(G)}\cX=\left[\coprod_{\iota\in\Emb(G,H)}\J_{\infty}\ulHom_{\Df_{\Gamma_{\cX,G}}}^{\iota}(E_{\Gamma_{\cX,G}},V_{\Gamma_{\cX,G}})/H\right].
\]
Let us fix $\iota\in\Emb(G,H)$, let $K\subset G$ be the kernel of
$G\xrightarrow{\iota}H\to Q$ and let $\overline{\iota}\colon G':=G/K\hookrightarrow Q$
be the induced embedding. Then, we have a natural morphism $\Gamma_{\cX,G}\to\Gamma_{\cX',G'}$,
sending $G$-torsors to the induced $Q$-torsors. We have a natural
morphism
\[
\J_{\infty}\ulHom_{\Df_{\Gamma_{\cX,G}}}^{\iota}(E_{\Gamma_{\cX,G}},V_{\Gamma_{\cX,G}})\to\J_{\infty}\ulHom_{\Df_{\Gamma_{\cX^{\rig},G'}}}^{\overline{\iota}}(E_{\Gamma_{\cX^{\rig},G'}},V_{\Gamma_{\cX^{\rig},G'}})\times_{\Gamma_{\cX^{\rig},G'}}\Gamma_{\cX,G}.
\]
We claim that this is an isomorphism. Indeed, an $S$-point of the
fiber product of the target corresponds to a pair of an $\overline{\iota}$-equivariant
morphism $E_{S}\to V$ and a $G$-cover $\widetilde{E_{S}}\to\Df_{S}$
such that $\widetilde{E_{S}}/K=E_{S}$. The composition $\widetilde{E_{S}}\to E_{S}\to V$
is obviously $\iota$-equivariant, which give an $S$-point of the
left side. This construction gives the inverse morphism. Since $\Gamma_{\cX,G}\to\Gamma_{\cX^{\rig},G'}$
is ccft, so is the morphism 
\[
\J_{\infty}\ulHom_{\Df_{\Gamma_{\cX,G}}}^{\iota}(E_{\Gamma_{\cX,G}},V_{\Gamma_{\cX,G}})\to\J_{\infty}\ulHom_{\Df_{\Gamma_{\cX^{\rig},G'}}}^{\overline{\iota}}(E_{\Gamma_{\cX^{\rig},G'}},V_{\Gamma_{\cX^{\rig},G'}}).
\]

Finally, we consider the following commutative diagram:
\[
\xymatrix{\coprod_{G,\iota}\J_{\infty}\ulHom_{\Df_{\Gamma_{\cX,G}}}^{\iota}(E_{\Gamma_{\cX,G}},V_{\Gamma_{\cX,G}})\ar[d]\ar[r] & \cJ_{\infty}\cX\ar[d]^{\rho_{\infty}}\\
\coprod_{G',\overline{\iota}}\J_{\infty}\ulHom_{\Df_{\Gamma_{\cX^{\rig},G'}}}^{\overline{\iota}}(E_{\Gamma_{\cX^{\rig},G'}},V_{\Gamma_{\cX^{\rig},G'}})\ar[r] & \cJ_{\infty}\cX^{\rig}
}
\]
The horizontal arrows are finite, étale and surjective. Since the
left vertical arrow is ccft, so is $\rho_{\infty}$. 
\end{proof}

\subsection{Versal families of $k'\protect\tpars$-points}

We keep assuming that $\cX=\cX_{0}\times_{\Spec k}\Df_{k}$ with $\cX_{0}$
a DM stack of finite type over $k$. For each geometric point $\Spec k'\to\cJ_{\infty}\cX_{0}$,
there exists the corresponding $k'\tpars$-point of $\cX_{0}$. However,
for an $R$-point $\Spec R\to\cJ_{\infty}\cX_{0}$ with $R$ a general
ring, we cannot get a corresponding $R\tpars$-point of $\cX_{0}$.
This is because if $\cE\to\cX_{0}$ is the family of twisted arcs
over $\Spec R$ that corresponds to the given $R$-point of $\cJ_{\infty}\cX_{0}$,
then $\cE$ is defined as a formal DM stack and does not contain $\Spec R\tpars$
just like the formal spectrum $\Spf R\tbrats$ does not contain $\Spec R\tpars$.
That being said, we can still construct an étale cover $\coprod_{s\in S}\Spec R_{s}\to\cJ_{\infty}\cX_{0}$
from a countable disjoint union of affine schemes such that the corresponding
$R_{s}\tpars$-points of $\cX_{0}$ exist, as we will show in this
section.
\begin{prop}
\label{prop:versal-affine}There exists a collection of data, 
\begin{enumerate}
\item an étale, surjective and stabilizer-preserving morphism $\phi\colon\coprod_{i=1}^{n}[V_{i,0}/H_{i}]\to\cX_{0}$,
where $V_{i,0}$ are affine schemes $\Spec A_{i}$ and $H_{i}$ are
finite groups,
\item an étale surjective morphism $\psi\colon\coprod_{s\in S}\Spec R_{s}\to\cJ_{\infty}\cX$
from the disjoint union of affine schemes $\Spec R_{s}$ with $S$
an at most countable index set,
\item for each $s\in S$, a uniformizable $R_{s}\tbrats$-algebra $O_{s}$
(see \cite[Definition 7.1]{tonini2023moduliof} or \cite[Definition 7.1]{yasuda2024motivic2})
given with an action of a Galoisian finite group $G_{s}$ such that
$O_{s}^{G_{s}}=R_{s}\tbrats$,
\item for each $s\in S$, an index $i(s)\in\{1,\dots,n\}$, an embedding
$\iota_{s}\colon G_{s}\hookrightarrow H_{i(s)}$ and an $\iota_{s}$-equivariant
morphism $\beta_{s}\colon\Spec O_{s}\to V_{i(s),0}$ over $k$ which
induces a representable morphism $\gamma_{s}\colon[\Spf O_{s}/G_{s}]\to[V_{i(s),0}/H_{i(s)}]$,
\end{enumerate}
such that for each $s\in S$, the induced morphism of stacks,
\[
[\Spf O_{s}/G_{s}]\xrightarrow{\gamma_{s}}[V_{i(s),0}/H_{i(s)}]\xrightarrow{\phi_{s}}\cX_{0},
\]
is the family of twisted arcs over $\Spec R_{s}$ corresponding to
$\psi_{s}\colon\Spec R_{s}\to\cJ_{\infty}\cX=\cJ_{\infty}\cX_{0}$.
Here $\phi_{s}$ and $\psi_{s}$ denote the restrictions of $\phi$
and $\psi$ to the $s$-components, respectively. 
\end{prop}

\begin{proof}
As is well-known, there exists an étale, surjective and stabilizer-preserving
morphism $\phi\colon\coprod_{i=1}^{n}[V_{i,0}/H_{i}]\to\cX_{0}$,
where $V_{i,0}$ are affine schemes $\Spec A_{i}$ and $H_{i}$ are
finite groups. Let $V_{i}:=V_{i,0}\times_{\Spec k}\Df_{k}$. As in
the proof of Proposition \ref{prop:ccft}, there exists an étale surjective
morphism
\[
\coprod_{\substack{1\le i\le n\\
G\in\GalGps\\
\iota\colon G\hookrightarrow H_{i}
}
}\J_{\infty}\ulHom_{\Df_{\Gamma_{[V_{i}/H_{i}],G}}}^{\iota}(E_{\Gamma_{\cX,G}},(V_{i})_{\Gamma_{\cX,G}})\to\cJ_{\infty}\cX.
\]
Note that the coproduct of the source is countable. For each $i,G,\iota$
and for each connected component $\Upsilon\subset\Gamma_{[V_{i}/H_{i}],G}$,
which is a DM stack of finite type over $k$, the substack $\ulHom_{\Df_{\Upsilon}}^{\iota}(E_{\Upsilon},(V_{i})_{\Upsilon})\subset\ulHom_{\Df_{\Gamma_{[V_{i}/H_{i}],G}}}^{\iota}(E_{\Gamma_{\cX,G}},(V_{i})_{\Gamma_{\cX,G}})$
is a formal scheme of finite type over $\Df_{\Upsilon}$. Let $T=\Spec C\to\Upsilon$
be an étale surjective morphism from an affine scheme and let 
\[
\Spf B_{s}\to\ulHom_{\Df_{T}}^{\iota}(E_{T},(V_{i})_{T})=\ulHom_{\Df_{\Upsilon}}^{\iota}(E_{\Upsilon},(V_{i})_{\Upsilon})\times_{\Upsilon}T
\]
be an étale surjective morphism of formal DM stacks over $\Df_{T}$.
Here $s$ denotes an index representing the chosen data, $i$, $G$,
$\iota$ and $\Upsilon$. The arc scheme $\J_{\infty}\Spf B_{s}$
(relative to the base $\Df_{T}$) is an affine scheme, say $\J_{\infty}\Spf B_{s}=\Spec R_{s}$.
Then, the natural morphism 
\[
\Spec R_{s}=\J_{\infty}\Spf B_{s}\to\J_{\infty}\ulHom_{\Df_{\Upsilon}}^{\iota}(E_{\Upsilon},(V_{i})_{\Upsilon}).
\]
corresponds to a $\Df_{\Upsilon}$-morphism 
\[
\Df_{R_{s}}=\Spf R_{s}\tbrats\to\ulHom_{\Df_{\Upsilon}}^{\iota}(E_{\Upsilon},(V_{i})_{\Upsilon}),
\]
which, in turn, corresponds to an $\iota$-equivariant $\Df_{k}$-morphism
\[
E_{R_{s}}=E_{\Upsilon}\times_{\Upsilon}\Spec R_{s}\to V_{i}.
\]
Here $E_{R_{s}}$ can be written as $\Spf O_{s}$ for a finite flat
$R_{s}\tbrats$-algebra with a $G$-action such that $O_{s}^{G}=R_{s}\tbrats$.
From \cite[Proposition 5.19(1)]{yasuda2024motivic2}, replacing $\Spec R_{s}$
with an étale cover of it if necessary, we may assume that $O_{s}$
is uniformizable. 

Since $V_{i}$ is affine, the induced morphism 
\[
\beta_{s}\colon\Spf O_{s}\to V_{i}\to V_{i,0}=\Spec A_{i}
\]
 uniquely factors as 
\[
\Spf O_{s}\to\Spec O_{s}\xrightarrow{\beta_{s}}\Spec A_{i}.
\]
Indeed, since $O_{s}=\projlim O_{s}/t^{n}O_{s}$ and we have a system
of homomorphisms $A_{i}\to O_{s}/t^{n}O_{s}$, we get the induced
homomorphism $A_{i}\to O_{s}$. Putting together all what we constructed
so far gives the desired data. 
\end{proof}
\begin{cor}
\label{cor:versal-affine-2}There exist an étale surjective morphism
\[
\psi\colon\coprod_{s\in S}\Spec R_{s}\to\cJ_{\infty}\cX=\cJ_{\infty}\cX_{0}
\]
with $S$ an at most countable index set and a morphism 
\[
\beta\colon\coprod_{s\in S}\Spec R_{s}\tpars\to\cX_{0}
\]
such that for every $s\in S$ and for every geometric point $\Spec k'\to\Spec R_{s}$,
the induced morphism
\[
\Spec k'\tpars\to\Spec R_{s}\tpars\xrightarrow{\beta_{s}}\cX_{0}
\]
corresponds to the twisted arc $\psi(r)\in(\cJ_{\infty}\cX_{0})(k')$. 
\end{cor}

\begin{proof}
We take data as in Proposition \ref{prop:versal-affine}. In particular,
we have a morphism $\psi\colon\coprod_{s\in S}\Spec R_{s}\to\cJ_{\infty}\cX$
and for each $s\in S$, we have an $\iota_{s}$-equivariant morphism
$\Spec O_{s}\to V_{i(s),0}$. Since the open subscheme $\Spec(O_{s})_{t}\subset\Spec O_{s}$
is a $G_{s}$-torsor over $\Spec R_{s}\tpars$, we get a morphism
\begin{multline*}
\Spec R_{s}\tpars=\Spec(O_{s})_{t}/G_{s}=[\Spec(O_{s})_{t}/G_{s}]\hookrightarrow[\Spec O_{s}/G_{s}]\\
\to[V_{i(s),0}/H_{i(s)}]\to\cX_{0}.
\end{multline*}
Putting these morphisms for $s\in S$ together gives a morphism $\coprod_{s\in S}\Spec R_{s}\tpars\to\cX_{0}$
satisfying the desired property. 
\end{proof}
For a ring $R$ and $a\in R\tpars$, we define the function 
\[
\ord_{a}\colon\Spec R\to\ZZ\cup\{\infty\},\,r\mapsto\ord a_{r},
\]
where $a_{r}$ is the image of $a$ in $\kappa(r)\tpars$ and we follow
the convention $\ord0=\infty$. 
\begin{lem}
\label{lem:const-order-unit}Let $R$ be a reduced ring and let $a\in R\tpars$.
Suppose that $\ord_{a}$ takes a constant finite value $n\in\ZZ$.
Then, $a$ is a unit of $R\tpars$.
\end{lem}

\begin{proof}
Let us write $a=\sum_{i\in\ZZ}a_{i}t^{i}$, where $a_{i}\in R$ and
$a_{i}=0$ for $i\ll0$. Let $m<n$. For every $r\in\Spec R$, the
image $(a_{m})_{r}\in\kappa(r)$ of $a_{m}$ is zero. This means that
$a_{m}$ is contained in every prime ideal of $R$. But, the intersection
of all prime ideals is the nilradical, which is the zero ideal since
$R$ is reduced. Thus, $a_{m}=0$ for every $m<n$. For every $r\in\Spec R$,
$(a_{n})_{r}\in\kappa(r)$ is nonzero. Namely, $a_{n}$ is not contained
in any prime ideal. This means that $a_{n}$ is a unit of $R$. Then,
$at^{-n}$ is a unit of $R\tbrats$. It follows that $a$ is a unit
of $R\tpars$. 
\end{proof}
\begin{defn}[{\cite[p. 356]{grothendieck1961elements2}}]
Let $W$ be a topological space. A subset $C\subset W$ is said to
be \emph{retrocompact }if the inclusion map $C\hookrightarrow W$
is quasi-compact. A subset $C\subset W$ is said to be \emph{constructible
}if it is the union of finitely many subsets of the form $C\setminus C'$
with $C$ and $C'$ being retrocompact open subsets. 
\end{defn}

\begin{lem}
\label{lem:ord-constr}Let $R$ be a ring and let $a\in R\tpars$.
\end{lem}

\begin{enumerate}
\item For each $m\in\ZZ$, the subset $\{\ord_{a}\le m\}\subset\Spec R$
is retrocompact and open.
\item For each $m\in\ZZ$, and the subset $\{\ord_{a}=m\}\subset\Spec R$
is locally closed and constructible. 
\item If $\ord_{a}(x)<\infty$ for every $x\in\Spec R$, then the subset
$\ord_{a}(\Spec R)\subset\ZZ$ is finite. 
\end{enumerate}
\begin{proof}
\begin{enumerate}
\item Let us write $a=\sum_{i}a_{i}t^{i}$, $a_{i}\in R$. We have
\[
\{\ord_{a}\le m\}=\bigcup_{i\le m}\Spec R_{a_{i}}.
\]
Each $\Spec R_{a_{i}}$ is a distinguished open subset, in particular,
a retrocompact open subset of $\Spec R$. Since $a_{i}=0$ for $i\ll0$,
the union above is in fact a finite union. It follows that $\{\ord_{a}\le m\}$
is a retrocompact open subset. 
\item This follows from
\[
\{\ord_{a}=m\}=\{\ord_{a}\le m\}\setminus\{\ord_{a}\le m-1\}.
\]
\item This follows from the fact that $\Spec R$ is quasi-compact for the
constructible topology \cite[Chapter I, Proposition 1.9.15]{grothendieck1964elements}. 
\end{enumerate}
\end{proof}
\begin{prop}
\label{prop:factor-open}Let us take a collection of data as in Proposition
\ref{prop:versal-affine}, which we represent in the same notation.
Let $\cZ_{0}\subset\cX_{0}$ be a closed substack and let $\cU_{0}:=\cX_{0}\setminus\cZ_{0}\subset\cX_{0}$
be its complement open substack. For each $s\in S$, let $C_{s}\subset\Spec R_{s}$
be the subset of the points $r\in\Spec R_{s}$ such that the induced
morphism 
\[
\Spec\kappa(r)\tpars\to\Spec R_{s}\tpars\xrightarrow{\beta_{s}}\cX_{0}
\]
 factors through $\cU_{0}$. Then, $C_{s}$ is the disjoint union
of at most countably many locally closed, constructible and affine
subschemes $\Spec R_{s,m}\subset\Spec R_{s}$, $m\in M_{s}$ such
that for every $m\in M_{s}$, the induced morphism 
\[
\Spec R_{s,m}\tpars\to\Spec R_{s}\tpars\xrightarrow{\beta_{s}}\cX_{0}
\]
 factors through $\cU_{0}$. 
\end{prop}

\begin{proof}
Since $O_{s}$ is uniformizable, we can write $O_{s}=R_{s}\llbracket u\rrbracket$
with $u$ an indeterminate. Let us write $V_{i(s),0}=\Spec A$. Let
$Z\subset V_{i(s),0}$ be the pullback of $\cZ$ to $V_{i(s),0}$
and let $I\subset A$ be its defining ideal. Since $A$ is finitely
generated over $k$, the ideal $I$ is finitely generated, say $I=\langle f_{1},\dots,f_{l}\rangle$.
Let $\beta_{s}^{-1}I=\langle\beta_{s}^{*}f_{1},\dots,\beta_{s}^{*}f_{l}\rangle\subset R_{s}\llbracket u\rrbracket$.
We have 
\begin{align*}
C_{s} & =\{\ord_{\beta_{s}^{-1}I}<\infty\}\\
 & =\bigcup_{i=1}^{l}\{\ord_{\beta_{s}^{*}f_{i}}<\infty\}.\\
 & =\bigcup_{i=1}^{l}\bigcup_{n\in\ZZ_{\ge0}}\{\ord_{\beta_{s}^{*}f_{i}}=n\}.
\end{align*}
For each $i$ and $n$, the subset $\{\ord_{\beta_{s}^{*}f_{i}}=n\}$
is locally closed and constructible, in particular, quasi-compact.
Giving this subset the reduced scheme structure, we have a finite
affine cover $\{\ord_{\beta_{s}^{*}f_{i}}=n\}=\bigcup_{j=1}^{e}\Spec B_{j}$
with $B_{j}$ reduced. For each $j$, let $g\in B_{j}\llbracket u\rrbracket$
be the image of $\beta_{s}^{*}f_{i}$ by the natural map $R_{s}\llbracket u\rrbracket\to B_{j}\llbracket u\rrbracket$.
Then, the function $\ord_{g}$ is constantly $n$ on $\Spec B_{j}$.
From Lemma \ref{lem:const-order-unit}, $g$ is a unit in $B_{j}\llparenthesis u\rrparenthesis$.
Hence, the map $A\to B_{j}\llparenthesis u\rrparenthesis$ uniquely
factors through the localization map $A\to A_{\beta_{s}^{*}f_{i}}$.
Equivalently, the morphism 
\[
\Spec B_{j}\llparenthesis u\rrparenthesis\to\Spec R_{s}\llparenthesis u\rrparenthesis\xrightarrow{\beta_{s}}\Spec A=V_{i(s),0}
\]
has the image contained in the distinguished open subset $\{\beta_{s}^{*}f_{i}\ne0\}\subset V_{i(s),0}$,
which is contained in the open subset $U_{i(s),0}:=V_{i(s),0}\setminus Z=\bigcup_{i=1}^{l}\{\beta_{s}^{*}f_{i}\ne0\}$.
From the construction, $U_{i(s),0}$ is the preimage of $\cU_{0}$
in $V_{i(s),0}$ and stable under the $H_{i(s)}$-action. The morphism
$\Spec B_{j}\llparenthesis u\rrparenthesis\to U_{i(s),0}$ is $\iota$-equivariant
and induces a morphism 
\[
\Spec B_{j}\tpars=[\Spec B_{j}\llparenthesis u\rrparenthesis/G]\to[U_{i(s),0}/H_{i(s)}]\to\cU_{0}.
\]
The subset $C_{s}$ is covered by the countably many subschemes $\Spec B_{j}\subset\Spec R$
as $n$, $i$ and $j$ vary. We have proved the proposition.
\end{proof}

\subsection{The Zariski topology of a projective limit of DM stacks}

Results of this section are not restricted to stacks of twisted arcs,
although they are helpful to understand the Zariski topology of $|\cJ_{\infty}\cX|$.
Let $\cX_{n}$, $n\in\ZZ_{\ge0}$ be quasi-compact and separated DM
stacks and let $\phi_{n}\colon\cX_{n+1}\to\cX_{n}$, $n\in\ZZ_{\ge0}$
be affine morphisms. Then, the limit $\cX:=\projlim\cX_{n}$ is again
a quasi-compact and separated DM stack. The natural morphisms $\pi_{n}\colon\cX\to\cX_{n}$
are affine. In this section, we give a few results on the Zariski
topology of the point set $|\cX|$ of $\cX$. 
\begin{lem}[{\cite[Remark B.4]{rydh2015noetherian}}]
\label{lem:Rydh} Let $\cU\subset\cX$ be a quasi-compact open substack.
Then, there exist $n\in\ZZ_{\ge0}$ and an open substack $\cV\subset\cX_{n}$
such that $\cU=\cV\times_{\cX_{n}}\cX$. 
\end{lem}

\begin{lem}
Let us regard the point sets $|\cX|$ and $|\cX_{n}|$ as topological
spaces by giving them the Zariski topology. Let us give $\projlim|\cX_{n}|$
the limit topology. Then, the natural map $|\cX|\to\projlim|\cX_{n}|$
is a homeomorphism.
\end{lem}

\begin{proof}
We first show that the map is bijective. Let $x_{n}\in|\cX_{n}|$,
$n\in\ZZ_{\ge0}$ be a sequence of points such that $x_{n+1}$ maps
to $x_{n}$ by the map $|\cX_{n+1}|\to|\cX_{n}|$. Inductively, we
can construct the following 2-commutative diagram:
\[
\xymatrix{\cdots & \Spec L_{n}\ar[d]^{x_{n}'}\ar[l] & \Spec L_{n+1}\ar[l]\ar[d]^{x_{n+1}'} & \cdots\ar[l]\\
\cdots & \cX_{n}\ar[l] & \cX_{n+1}\ar[l] & \cdots\ar[l]
}
\]
Here $x_{n}'$ are geometric points representing $x_{n}$, respectively.
Passing to the limit, we obtain a geometric point $\Spec\injlim L_{n}\to\cX$
that maps to $(x_{n})\in\projlim|\cX_{n}|$. We have proved that the
map of the proposition is surjective.

To show that the map is injective, take two geometric points $x,y\in\cX(K)$
such that the induced elements $(x_{n})$ and $(y_{n})$ of $\projlim|\cX_{n}|$
are the same. There exists $n_{0}$ such that for every $n\ge n_{0}$,
we have $\Aut(x)=\Aut(x_{n})$ and $\Aut(y)=\Aut(y_{n})$. For each
$n$, we have an isomorphism $\alpha_{n}\colon x_{n}\to y_{n}$ in
$\cX_{n}(K)$. The problem is that $\alpha_{n+1}$ may not map to
$\alpha_{n}$. But, $\alpha_{n}=\beta\circ\phi_{n}(\alpha_{n+1})$
for some automorphism $\beta$ of $y_{n}$. If $n\ge n_{0}$, then
there exists a unique lift $\beta'$ of $\beta$ to an automorphism
of $y_{n+1}$. We have $\phi_{n}(\beta'\circ\alpha_{n+1})=\alpha_{n}$.
Applying this argument inductively, we get a compatible system of
isomorphisms $x_{n}\to y_{n}$ that induces an isomorphism $x\to y$.
Thus, the two points $x$ and $y$ determine the same point of $|\cX|$,
which shows the desired injectivity. 

For an open subset $U\subset|\cX_{n}|$, $\pi_{n}^{-1}(U)$ is an
open subset of $|\cX|$. This shows that the Zariski topology of $|\cX|$
is stronger than the limit topology. To show the converse, let $U\subset|\cX|$
be an open subset. We can cover $U$ with quasi-compact open subsets
$U_{i}$, $i\in I$. From Lemma \ref{lem:Rydh}, for each $i$, there
exist $n(i)\in\ZZ_{\ge0}$ and an open subset $V_{i}\subset|\cX_{n(i)}|$
such that $\pi_{n(i)}^{-1}(V_{i})=U_{i}$. This shows that the limit
topology is stronger than the Zariski topology. Combining the two
claims completes the proof. 
\end{proof}
\begin{cor}
\label{cor:constr-finite-level}For every constructible subset $C\subset|\cX|$,
there exist $n\in\ZZ_{\ge0}$ and a constructible subset $D\subset|\cX_{n}|$
such that $C=\pi_{n}^{-1}(D)$.
\end{cor}

\begin{proof}
Let us write $C=\bigcup_{i=1}^{l}C_{i}\setminus C_{i}'$, where $C_{i}$
and $C_{i}'$ are retrocompact open subsets of $|\cX|$. Since $\cX$
is affine over $\cX_{n}$'s, $|\cX|$ is quasi-compact and hence $C_{i}$
and $C_{i}'$ are quasi-compact as well. From Lemma \ref{lem:Rydh},
for sufficiently large $n$, there are open subsets $D_{i}$ and $D_{i}'$
of $|\cX_{n}|$ such that $C_{i}=\pi_{n}^{-1}(D_{i})$ and $C'_{i}=\pi_{n}^{-1}(D'_{i})$.
Then, $C=\pi_{n}^{-1}(\bigcup_{i}(D_{i}\setminus D_{i}'))$. This
shows the corollary. 
\end{proof}

\section{A product lemma for asymptotic growth rate\label{sec:A-product-lemma}}

\selectlanguage{english}%
Let $(\beta_{1},\beta_{2})\in\ZZ_{\geq0}^{2}$. For $k\geq-1$, we
estimate the integrals: 
\[
\Phi_{k}(B,\beta_{1},\beta_{2}):=\int_{1}^{B}u^{k}\log(u)^{\beta_{1}}\log(B/u)^{\beta_{2}}du
\]

\begin{lem}
\label{kbigger} Let $k>-1$ be a real number. One has that 
\[
\Phi_{k}(B,\beta_{1},\beta_{2})\sim\frac{\beta_{2}!B^{k+1}\log(B)^{\beta_{1}}}{(k+1)^{\beta_{2}+1}}+O(f(B)),
\]
where $f(B):=B^{k+1}\log(B)^{\beta_{1}-1}$ if $\beta_{1}\geq1$ and
$f(B):=\log(B)^{\beta_{2}}$ if $\beta_{1}=0$.
\end{lem}

\begin{proof}
We write $\Phi$ for $\Phi_{k}$. The integration by parts for the
functions $u^{k+1}/(k+1)$ and $\log(u)^{\beta_{1}}\log(B/u)^{\beta_{2}}$
gives: 
\begin{multline*}
\frac{u^{k+1}\log(u)^{\beta_{1}}\log(B/u)^{\beta_{2}}}{k+1}\bigg|_{u=1}^{B}=\Phi(B,\beta_{1},\beta_{2})+\int_{1}^{B}\frac{u^{k+1}}{k+1}(\log(u)^{\beta_{1}}\log(B/u)^{\beta_{2}})'du\\
=\Phi(B,\beta_{1},\beta_{2})+\frac{\beta_{1}\Phi(B,\beta_{1}-1,\beta_{2})-\beta_{2}\Phi(B,\beta_{1},\beta_{2}-1)}{k+1},
\end{multline*}
where one can put $\Phi(B,-1,\beta_{2})=\Phi(B,\beta_{1},-1)=0$.
One has that 
\[
\frac{u^{k+1}\log(u)^{\beta_{1}}\log(B/u)^{\beta_{2}}}{k+1}\bigg|_{u=1}^{B}=\begin{cases}
0, & \text{ if }\beta_{1}>0\text{ and }\beta_{2}>0;\\
B^{k+1}\log(B)^{\beta_{1}}/(k+1) & \text{ if }\beta_{1}>0\text{ and }\beta_{2}=0\\
-\log(B)^{\beta_{2}}/(k+1) & \text{ if }\beta_{1}=0\text{ and }\beta_{2}>0\\
(B^{k+1}-1)/(k+1) & \text{if }\beta_{1}=\beta_{2}=0.
\end{cases}.
\]
When $\beta_{2}=0$, we have for $\beta_{1}>0$ that 
\[
\Phi(B,\beta_{1},0)=\frac{B^{k+1}\log(B)^{\beta_{1}}}{k+1}-\frac{\beta_{1}\Phi(B,\beta_{1}-1,0)}{k+1}.
\]
As $\Phi(B,0,0)=(B^{k+1}-1)/(k+1)$, the induction gives that 
\[
\Phi(B,\beta_{1},0)=B^{k+1}\bigg(\sum_{i=0}^{\beta_{1}-1}\frac{(-1)^{i}\beta_{1}!\log(B)^{\beta_{1}-i}}{(\beta_{1}-i)!(k+1)^{i+1}}+\frac{(-1)^{\beta_{1}}\beta_{1}!(1-B^{-(k+1)})}{(k+1)^{\beta_{1}+1}}\bigg).
\]
We deduce 
\[
\Phi(B,\beta_{1},0)\sim\frac{B^{k+1}\log(B)^{\beta_{1}}}{k+1}+O(B^{k+1}\log(B)^{\beta_{1}-1})
\]
if $\beta_{1}>0$ and 
\[
\Phi(B,0,0)=\frac{B^{k+1}}{k+1}+O(1).
\]
Let us now suppose that $\beta_{1}=0$. When $\beta_{2}>0$, we have
that 
\[
\Phi(B,0,\beta_{2})=\frac{-\log(B)^{\beta_{2}}}{k+1}+\beta_{2}\Phi(B,0,\beta_{2}-1).
\]
As $\Phi(B,0,0)=0$, the induction gives that 
\[
\Phi(B,0,\beta_{2})=\sum_{i=0}^{\beta_{2}-1}\frac{-\beta_{2}!\log(B)^{\beta_{2}-i}}{(\beta_{2}-i)!(k+1)^{i+1}}+\frac{\beta_{2}!(B^{k+1}-1)}{(k+1)^{\beta_{2}+1}}.
\]
We deduce 
\[
\Phi(B,0,\beta_{2})\sim\frac{\beta_{2}!B^{k+1}}{(k+1)^{\beta_{2}+1}}+O(\log(B)^{\beta_{2}}).
\]
The claim is hence valid if $\beta_{1}=0$ or $\beta_{2}=0$. Suppose
the claim is valid for all $(\beta_{1}',\beta_{2}')$ with $0\leq\beta_{1}'\leq\beta_{1}$
and $0\leq\beta_{2}'\leq\beta_{2}$. Let us verify it for $(\beta_{1}+1,\beta_{2})$
and $(\beta_{1},\beta_{2}+1)$. First for for $(\beta_{1}+1,\beta_{2})$.
If $\beta_{2}=0$, we have already verified it, so suppose $\beta_{2}>0$.
One has that 
\[
\Phi(B,\beta_{1}+1,\beta_{2})=\frac{-(\beta_{1}+1)\Phi(B,\beta_{1},\beta_{2})+\beta_{2}\Phi(B,\beta_{1}+1,\beta_{2}-1)}{k+1}.
\]
The induction gives that $\Phi(B,\beta_{1},\beta_{2})\asymp B^{k+1}\log(B)^{\beta_{1}},$
hence 
\[
\Phi(B,\beta_{1}+1,\beta_{2})\sim\frac{\beta_{2}\Phi(B,\beta_{1}+1,\beta_{2}-1)}{k+1}+O(B^{k+1}\log(B)^{\beta_{1}}).
\]
By iterating this process, we eventually obtain that 
\begin{align*}
\Phi(B,\beta_{1}+1,\beta_{2}) & \sim\frac{\beta_{2}!\Phi(B,\beta_{1}+\beta_{2},0)}{(k+1)^{\beta_{2}+1}}+O(B^{k+1}\log(B)^{\beta_{1}})\\
 & \sim\frac{\beta_{2}!B^{k+1}\log(B)^{\beta_{1}+1}}{(k+1)^{\beta_{2}+1}}+O(B^{k+1}\log(B)^{\beta_{1}}),
\end{align*}
and the induction step is completed. Let us verify the claim for $(\beta_{1},\beta_{2}+1)$.
If $\beta_{1}=0$, we have already verified it, so suppose $\beta_{1}>0$.
One has that 
\[
\Phi(B,\beta_{1},\beta_{2}+1)=\frac{-\beta_{1}\Phi(B,\beta_{1}-1,\beta_{2}+1)+(\beta_{2}+1)\Phi(B,\beta_{1},\beta_{2})}{k+1}.
\]
The induction gives 
\[
\Phi(B,\beta_{1},\beta_{2})\sim\frac{\beta_{2}!B^{k+1}\log(B)^{\beta_{1}}}{(k+1)^{\beta_{2}+1}}+O(B^{k+1}\log(B)^{\beta_{1}-1}),
\]
thus 
\[
\Phi(B,\beta_{1},\beta_{2}+1)\sim\frac{(\beta_{2}+1)!B^{k+1}\log(B)^{\beta_{1}}}{(k+1)^{\beta_{2}+2}}-\frac{\beta_{1}\Phi(B,\beta_{1}-1,\beta_{2}+1)}{(k+1)}.
\]
If $\beta_{1}-1=0$, we have already verified it, so suppose $\beta_{1}-1>0$.
It suffices to have $\Phi(B,\beta_{1}-1,\beta_{2}+1)=O(B^{k+1}\log(B)^{\beta_{1}-1}).$
Now 
\[
\Phi(B,\beta_{1}-1,\beta_{2}+1)=\frac{-(\beta_{1}-1)\Phi(B,\beta_{1}-2,\beta_{2}+1)+(\beta_{2}+1)\Phi_{B}(B,\beta_{1}-1,\beta_{2})}{k+1}.
\]
The induction gives $\Phi(B,\beta_{1}-1,\beta_{2})=O(B^{k+1}\log(B)^{\beta_{1}-1})$
and thus it suffices to have $\Phi(\beta_{1}-2,\beta_{2}+1)=O(B^{k+1}\log(B)^{\beta_{1}-3})$
By iterating this process, we reach $\beta_{1}=0$, when we know that
$\Phi(B,0,\beta_{2}+1)=O(B^{k+1})$. The proof is completed. 
\end{proof}
\begin{lem}
\label{kminus} One has that 
\begin{align*}
\Phi_{-1}(B,\beta_{1},\beta_{2}) & =\sum_{i=0}^{\beta_{2}}\frac{\beta_{1}!\beta_{2}!\log(B)^{\beta_{1}+1+i}}{(\beta_{1}+i+1)!(\beta_{2}-i)!}\\
 & \sim\frac{\beta_{1}!\beta_{2}!\log(B)^{\beta_{1}+\beta_{2}+1}}{(\beta_{1}+\beta_{2}+1)!}+O(\log(B)^{\beta_{1}+\beta_{2}}).
\end{align*}
\end{lem}

\begin{proof}
We write $\Phi$ for $\Phi_{-1}$. One has that 
\begin{multline*}
(\log(u)^{\beta_{1}}\log(B/u)^{\beta_{2}})'\\
=\beta_{1}u^{-1}\log(u)^{\beta_{1}-1}\log(B/u)^{\beta_{2}}-\beta_{2}u^{-1}\log(u)^{\beta_{1}}\log(B/u)^{\beta_{2}-1}.
\end{multline*}
Integration by parts for the functions $\log(u)$ and $\log(u)^{\beta_{1}}\log(B/u)^{\beta_{2}}$
gives: 
\begin{multline*}
\log(u)^{\beta_{1}+1}\log(B/u)^{\beta_{2}}|_{u=1}^{u=B}\\
=\Phi(B,\beta_{1},\beta_{2})+\beta_{1}\Phi(B,\beta_{1},\beta_{2})-\beta_{2}\Phi(B,\beta_{1}+1,\beta_{2}-1),
\end{multline*}
with convention that $\Phi(B,\beta_{1}+1,-1)=0$. We have that 
\[
\log(u)^{\beta_{1}+1}\log(B/u)^{\beta_{2}}|_{u=1}^{u=B}=\begin{cases}
\log(B)^{\beta_{1}+1}, & \text{if }\beta_{2}=0\\
0, & \text{otherwise}
\end{cases}.
\]
If $\beta_{2}=0$, we have that 
\[
\Phi(B,\beta_{1},0)=\frac{\log(B)^{\beta_{1}+1}}{1+\beta_{1}},
\]
which proves the claim in this case. Suppose that $\beta_{2}>0$.
We have that 
\[
\Phi(B,\beta_{1},\beta_{2})=\frac{\log(B)^{\beta_{1}+1}+\beta_{2}\Phi(B,\beta_{1}+1,\beta_{2}-1)}{1+\beta_{1}}.
\]
We obtain 
\[
\Phi(B,\beta_{1},\beta_{2})=\sum_{i=0}^{\beta_{2}}\frac{\beta_{1}!\beta_{2}!\log(B)^{\beta_{1}+1+i}}{(\beta_{1}+i+1)!(\beta_{2}-i)!}.
\]
\end{proof}
\begin{thm}
\label{thm:product-lemma}Let $X,Y$ be countable sets with heights
$H_{1}:X\to\RR_{>0}$ and $H_{2}:X\to\RR_{>0}$. Suppose that 
\[
N_{X}(B):=\#\{x\in X|H_{1}(x)\leq B\}\asymp B^{\alpha_{1}}\log(B)^{\beta_{1}}
\]
and that 
\[
N_{Y}(B):=\#\{y\in Y|H_{2}(y)\leq B\}\asymp B^{\alpha_{2}}\log(B)^{\beta_{2}},
\]
where $\alpha_{1}\geq\alpha_{2}>0$ and $\beta_{1},\beta_{2}\in\ZZ_{\geq0}$.
Let us define $(\alpha,\beta)$ by $(\alpha,\beta):=(\alpha_{1},\beta_{1})$
if $\alpha_{1}>\alpha_{2}$ and $(\alpha,\beta):=(\alpha_{1},\beta_{1}+\beta_{2}+1)$
otherwise. Then 
\[
N_{X\times Y}(B):=\{(x,y)\in X\times Y|H_{1}(x)H_{2}(y)\leq B\}\asymp B^{\alpha}\log(B)^{\beta}.
\]
\end{thm}

\begin{proof}
Note that if $(H_{1},H_{2})$ is a pair of heights satisfying the
conditions, then for any $c_{1},c_{2}>0$ the heights $(c_{1}H_{1},c_{2}H_{2})$
also satisfy the conditions. We can thus suppose that $H_{1}$ and
$H_{2}$ are such that $1=\min(H_{1})=\min(H_{2})$. We can now find
constants $C_{1},C_{1}',C_{2},C_{2}'>0$ such that for all $B\geq1$,
one has that 
\[
C_{1}'B^{\alpha_{1}}\log(B)^{\beta_{1}}<N_{X}(B)<C_{1}B^{\alpha_{1}}\log(B)^{\beta_{1}}
\]
and 
\[
C_{2}'B^{\alpha_{2}}\log(B)^{\beta_{2}}<N_{Y}(B)<C_{2}B^{\alpha_{2}}\log(B)^{\beta_{2}}.
\]
%Moreover, up to replacing $H_1$ by $CH_1$ for some $C>0$, we may assume that all the heights take values in $[1,\infty[$.
Denote by $a(j)$ the number of $x\in X$ such that $j\leq H_{1}(x)<j+1$.
One has that 
\begin{align*}
N_{X\times Y}(B) & =\sum_{x\in X}N_{Y}(B/H_{1}(x))\\
 & \leq\sum_{j=1}^{B}a(j)N_{Y}(B/j)\\
 & \leq C_{2}B^{\alpha_{2}}\sum_{j=1}^{B}a(j)j^{-\alpha_{2}}\log(B/j)^{\beta_{2}}
\end{align*}
Set $\phi(x):=x^{-\alpha_{2}}\log(B/x)^{\beta_{2}}$ which is a continuously
differentiable function on $]1,B[$ with derivative 
\begin{multline*}
\phi'(x)=-\alpha_{2}x^{-\alpha_{2}-1}\log(B/x)^{\beta_{2}}+x^{-\alpha_{2}}\beta_{2}\log(B/x)^{\beta_{2}-1}\cdot(x/B)\cdot(-B/x^{2})\\
=-\alpha_{2}x^{-\alpha_{2}-1}\log(B/x)^{\beta_{2}}-\beta_{2}x^{-\alpha_{2}-1}\log(B/x)^{\beta_{2}-1}
\end{multline*}
The Abel's summation formula gives that 
\[
\sum_{j=1}^{B}a(j)j^{-\alpha_{2}}\log(B/j)^{\beta_{2}}=N_{X}(B)\phi(B)%-N_{X}(1)\phi(1)
-\int_{1}^{B}N_{X}(u)\phi'(u)du.
\]
One has 
\begin{align*}
N_{X}(B)\phi(B) & \leq C_{1}B^{\alpha_{1}-\alpha_{2}}\log(B)^{\beta_{1}}\log(1)^{\beta_{2}}\ll B^{\alpha_{1}-\alpha_{2}}\log(B)^{\beta_{1}}
\end{align*}
with the convention that $\log(1)^{0}=1$. One has that 
\begin{multline*}
-\int_{1}^{B}N_{X}(u)\phi'(u)du\\
=\int_{1}^{B}N_{X}(u)\cdot(\alpha_{2}u^{-\alpha_{2}-1}\log(B/u)^{\beta_{2}}+\beta_{2}u^{-\alpha_{2}-1}\log(B/u)^{\beta_{2}-1})du.
\end{multline*}
Using that $N_{X}(u)\leq C_{1}u^{\alpha_{1}}\log(u)^{\beta_{1}},$
we estimate the last integral by 
\begin{align*}
-\int_{1}^{B}N_{X}(u)\phi'(u)du\hskip-3cm\\
 & \leq C_{1}\int_{1}^{B}\alpha_{2}u^{\alpha_{1}-\alpha_{2}-1}\log(u)^{\beta_{1}}\log(B/u)^{\beta_{2}}+\beta_{2}u^{\alpha_{1}-\alpha_{2}-1}\log(u)^{\beta_{1}}\log(B/u)^{\beta_{2}-1}du\\
 & \ll\Phi_{\alpha_{1}-\alpha_{2}-1}(B,\beta_{1},\beta_{2})+\Phi_{\alpha_{1}-\alpha_{2}-1}(B,\beta_{1},\beta_{2}-1)\\
 & \ll\Phi_{\alpha_{1}-\alpha_{2}-1}(B,\beta_{1},\beta_{2}),
\end{align*}
where the last inequality follows from Lemmas~\ref{kbigger} and~\ref{kminus}.
Suppose first that $\alpha_{1}-\alpha_{2}>0$. Lemma~\ref{kbigger}
gives that 
\[
-\int_{1}^{B}N_{X}(u)\phi'(u)\ll B^{\alpha_{1}-\alpha_{2}}\log(B)^{\beta_{1}}.
\]
We obtain that 
\begin{align*}
N_{X\times Y}(B)\ll B^{\alpha_{2}}\cdot\bigg(N_{X}(B)\phi(B)-\int_{1}^{B}N_{X}(u)\phi'(u)du\bigg)\ll B^{\alpha_{1}}\log(B)^{\beta_{1}},
\end{align*}
as claimed. Suppose that $\alpha_{1}=\alpha_{2}$. Lemma \ref{kminus}
gives 
\[
-\int_{1}^{B}N_{X}(u)\phi'(u)\ll\log(B)^{\beta_{1}+\beta_{2}+1}
\]
and thus 
\[
N_{X\times Y}(B)\ll B^{\alpha_{2}}\cdot\bigg(N_{X}(B)\phi(B)-\int_{1}^{B}N_{X}(u)\phi'(u)du\bigg)\ll B^{\alpha_{2}}\log(B)^{\beta_{1}+\beta_{2}+1},
\]
as claimed. Let us prove the lower bound. 
\begin{align*}
N_{X\times Y}(B) & =\sum_{x\in X}N_{Y}(B/H_{1}(x))\\
 & \geq\sum_{j=1}^{B}a(j)N_{Y}(B/(j+1))\\
 & \geq C_{2}'B^{\alpha_{2}}\sum_{j=1}^{B}a(j)(j+1)^{-\alpha_{2}}\log(B/(j+1))^{\beta_{2}}.
\end{align*}
Let us set $\psi(x):=(x+1)^{-\alpha_{2}}\log(B/(x+1))^{\beta_{2}}$.
One has that 
\[
\psi'(x)=\phi(x+1)'=\phi'(x+1).
\]
The Abel's summation formula gives that 
\[
\sum_{j=1}^{B}a(j)(j+1)^{-\alpha_{2}}\log(B/(j+1))^{\beta_{2}}=N_{X}(B)\psi(B)-\int_{1}^{B}N_{X}(u)\phi'(u+1)du.
\]
We have that 
\begin{multline*}
-\int_{1}^{B}N_{X}(u)\phi'(u+1)du\\
=\int_{1}^{B}N_{X}(u)(\alpha_{2}(u+1)^{-\alpha_{2}-1}\log(B/(u+1))^{\beta_{2}}+\beta_{2}(u+1)^{-\alpha_{2}-1}\log(B/(u+1))^{\beta_{2}-1})du.
\end{multline*}
Using that $N_{X}(u)\geq C_{1}'u^{\alpha_{1}}\log(u)^{\beta_{2}}$,
we can estimate 
\begin{multline*}
-\int_{1}^{B}N_{X}(u)\phi'(u+1)du\\
\geq C_{1}\int_{1}^{B}\frac{\alpha_{2}u^{\alpha_{1}}\log(u)^{\beta_{1}}\log(\frac{B}{u+1})^{\beta_{2}}}{(u+1)^{-\alpha_{2}-1}}+\frac{\beta_{2}u^{\alpha_{1}}\log(u)^{\beta_{1}}\log(\frac{B}{u+1})^{\beta_{2}-1}}{(u+1)^{-\alpha_{2}-1}}du
\end{multline*}
On the interval $[1,B]$ one has that $u+1\leq2u,$ thus one can estimate
\begin{multline}
-\int_{1}^{B}N_{X}(u)\phi'(u+1)du\\
\gg\int_{1}^{B}u^{\alpha_{1}-\alpha_{2}-1}\log(u)^{\beta_{1}}\log\bigg(\frac{B}{u+1}\bigg)^{\beta_{2}}+u^{\alpha_{1}-\alpha_{2}-1}\log(u)^{\beta_{1}}\log\bigg(\frac{B}{u+1}\bigg)^{\beta_{2}-1}du.\label{estint}
\end{multline}
When $B>4$, on the interval $u\in[1,B-2]$, one can estimate 
\[
\log\bigg(\frac{B}{u+1}\bigg)>\frac{\log(B/u)}{2},
\]
because this is equivalent to $\big(\frac{B}{u+1}\big)^{2}>B/u$ which
is true because 
\[
B>(B-3+(B-3)^{-1}+2)\geq(u+u^{-1}+2)=(u+1)^{2}/u
\]
for $u\in[1,B-2]$. The function under the last integral of (\ref{estint})
is nonnegative in the domain $[B-2,B-1]$. Whenever $B>4$ in the
domain $u\in[B-1,B]$, we have that 
\[
\log(B/(u+1))>\log(B/(B+1))>-\log(4/5).
\]
We deduce that for large~$B$, one has that 
\begin{multline*}
-\int_{1}^{B}N_{X}(u)\phi'(u+1)du\\
\gg\Phi_{\alpha_{1}-\alpha_{2}-1}(B-2,\beta_{1},\beta_{2})+\Phi_{\alpha_{1}-\alpha_{2}-1}(B-2,\beta_{1},\beta_{2}-1)-\int_{B-1}^{B}\alpha_{2}u^{\alpha_{1}-\alpha_{2}-1}\log(u)^{\beta_{1}}du.
\end{multline*}
Note that 
\begin{align*}
\int_{B-1}^{B}u^{\alpha_{1}-\alpha_{2}-1}\log(u)^{\beta_{1}}du & \sim\Phi_{\alpha_{1}-\alpha_{2}-1}(B,\beta_{1},0)-\Phi_{\alpha_{1}-\alpha_{2}-1}(B-1,\beta_{1},0)+O(f(B))\\
 & =\frac{(B^{\alpha_{1}-\alpha_{2}}\log(B)^{\beta_{1}}-(B-1)^{\alpha_{1}-\alpha_{2}}\log(B-1)^{\beta_{1}})}{\alpha_{1}-\alpha_{2}}+O(f(B)).\\
 & =O(B^{\alpha_{1}-\alpha_{2}-1}\log(B)^{\beta_{1}}),
\end{align*}
where we recall that~$f$ is the error term which is as in the respective
Lemmas~\ref{kbigger} and~\ref{kminus}. 
\begin{align*}
-\int_{1}^{B}N_{X}(u)\phi'(u+1)du & \gg\Phi_{\alpha_{1}-\alpha_{2}-1}(B,\beta_{1},\beta_{2})+\Phi_{\alpha_{1}-\alpha_{2}}(B,\beta_{1},\beta_{2}-1)\\
 & \gg\Phi_{\alpha_{1}-\alpha_{2}-1}(B,\beta_{1},\beta_{2}).
\end{align*}
We have that 
\begin{align*}
N_{X}(B)\psi(B)%-N_{X}(1)\psi(1)
\geq C_{1}'B^{\alpha_{1}}\log(B)^{\beta_{1}}(B+1)^{-\alpha_{2}}\log(B/(B+1))^{\beta_{2}}.%-2^{-\alpha_{2}}N_{X}(1)\log(B/2)^{\beta_{2}}.
\end{align*}
We want to establish that 
\begin{equation}
\Phi_{\alpha_{1}-\alpha_{2}-1}(B,\beta_{1},\beta_{2})+N_{X}(B)\psi(B)\gg\Phi_{\alpha_{1}-\alpha_{2}-1}(B,\beta_{1},\beta_{2}).\label{phist}
\end{equation}
Suppose that $\alpha_{1}=\alpha_{2}$. For $B\gg0$, one has that
\[
-N_{X}(B)\psi(B)=O(\log(B)^{\beta_{1}})
\]
so as 
\[
\Phi_{\alpha_{1}-\alpha_{2}-1}(B,\beta_{1},\beta_{2})\asymp\log(B)^{\beta_{1}+\beta_{2}+1},
\]
the inequality (\ref{phist}) is valid in this case. Suppose that
$\alpha_{1}>\alpha_{2}$. For $B\gg0$, we have that 
\[
-N_{X}(B)\psi(B)=O(B^{\alpha_{1}-\alpha_{2}})
\]
and as 
\[
\Phi_{\alpha_{1}-\alpha_{2}-1}(B,\beta_{1},\beta_{2})\asymp B^{\alpha_{1}-\alpha_{2}}\log(B)^{\beta_{1}},
\]
the inequality (\ref{phist}) is valid in this case. Finally, we estimate
\begin{align*}
N_{X\times Y}(B) & \gg B^{\alpha_{2}}\sum_{j=1}^{B}a(j)(j+1)^{-\alpha_{2}}\log(B/(j+1))^{\beta_{2}}\\
 & =B^{\alpha_{2}}\cdot\bigg(N_{X}(B)\psi(B)-\int_{1}^{B}N_{X}(u)\psi'(u+1)du\bigg)\\
 & \gg B^{\alpha_{2}}\cdot\Phi_{\alpha_{1}-\alpha_{2}-1}(B,\beta_{1},\beta_{2}).
\end{align*}
The proof is completed. 
\end{proof}
\selectlanguage{american}%

\section{Erratum for the preceding paper\label{sec:Erratum}}

We considered only positive raising functions in the first draft of
the paper \cite{darda2024thebatyrevtextendashmanin}, then dropped
this constraint in the published version. However, in Lemma 8.16 of
this paper, we need to assume that the raising function $c$ in question
be positive. This change does not affect the rest of the paper, since
the lemma is not used elsewhere. 

\bgroup\inputencoding{utf8}\bibliographystyle{amsalpha}
\bibliography{Positive-Char}
\egroup
\end{document}